\documentclass[11pt,leqno]{amsart}
\usepackage{amsmath,amstext,amsthm,amssymb,amsxtra}
\usepackage{txfonts} %also pxfonts
\usepackage[colorlinks,citecolor=red,pagebackref,hypertexnames=false]{hyperref}
\usepackage{pgf,tikz}
\usepackage{pdfsync}
\allowdisplaybreaks
\newcommand{\RR}{{\mathbb{R}}}
\newcommand{\NN}{{\mathbb{N}}}
\newcommand{\ZZ}{{\mathbb{Z}}}
\newcommand{\CC}{{\mathbb{C}}}
\newcommand{\eps}{\varepsilon}
\newcommand{\bp}{\noindent {\it Proof}.\,\,}
\newcommand{\ep}{\hfill$\Box$ \vskip 0.06in}
\newcommand{\dint}{\int\!\!\!\int}
\newcommand{\Sq}{\mathcal{A}}
 
\def\HPML { H^p_{L, {mol}, M}(\mathbb{R}^n) } 
 
\def\HPHML { \mathbb{H}^p_{L, {mol}, M}(\mathbb{R}^n) } 
\def\HPL { H^p_{L}(\mathbb{R}^n) }

\numberwithin{equation}{section}

\theoremstyle{plain}
\newtheorem{theorem}[equation]{Theorem}
\newtheorem{lemma}[equation]{Lemma}
\newtheorem{corollary}[equation]{Corollary}
\newtheorem{proposition}[equation]{Proposition}

\theoremstyle{definition}
\newtheorem{definition}[equation]{Definition}

\theoremstyle{remark}

\usepackage{color}

\definecolor{gr}{rgb}   {0.,   0.8,   0. }
\definecolor{bl}{rgb}   {0.,   0.1,   1. }
\definecolor{mg}{rgb}   {0.7,  0.,    0.7}

\newcommand{\Bk}{\color{black}}

\def\HRL { H^p_{L, {\rm Riesz}}(\mathbb{R}^n) }
\def\HPL { H^p_{L}(\mathbb{R}^n) }
\def\hpl { H^p_{L} }

\begin{document}

\author{Steve Hofmann, Svitlana Mayboroda and Alan McIntosh}

\title[Second order elliptic operators]{Second order elliptic operators with complex bounded measurable coefficients in $L^p$, Sobolev and Hardy spaces}

\thanks{The first and second authors were partially supported by NSF grants DMS 0801079 and 0758500, respectively.  The third author was supported by the Australian Government through the Australian Research Council.}

%\date{ }

\begin{abstract} Let $L$ be a  second order divergence form elliptic operator with complex bounded measurable coefficients. The operators arising in connection with  $L$, such as the heat semigroup and Riesz transform, are not, in general,  of Calder{\'o}n-Zygmund type and exhibit behavior
different from their counterparts built upon the Laplacian. The current paper aims at a thorough description of the properties of such operators in $L^p$, Sobolev, and  some new Hardy spaces naturally associated to $L$.

First, we show that the known ranges of boundedness in $L^p$ for the heat semigroup and Riesz transform  of $L$, are sharp. In particular, the heat semigroup $e^{-tL}$ need not be bounded in $L^p$ if $p\not\in [2n/(n+2),2n/(n-2)]$. Then we provide a complete description  of {\it all} Sobolev spaces in which $L$ admits a bounded functional calculus, in particular, where  $e^{-tL}$ is bounded. 

Secondly, we develop a comprehensive theory of Hardy and Lipschitz spaces associated to $L$, that serves the range of $p$ beyond  $[2n/(n+2),2n/(n-2)]$. It includes, in particular, characterizations by the sharp maximal function and the Riesz transform (for certain ranges of $p$), as well as
the molecular decomposition and duality and interpolation theorems. 

\end{abstract}

\maketitle

\tableofcontents

\newpage

\section{Introduction}\label{s1}
\setcounter{equation}{0}

Let $A$ be an $n\times n$ matrix with entries
\begin{equation}\label{eq1.1}
a_{jk}:\RR^n\longrightarrow \CC,\quad j=1,...,n, \quad
k=1,...,n,
\end{equation}

\noindent satisfying the ellipticity condition
\begin{equation}\label{eq1.2}
\lambda |\xi|^2\leq \Re e A\xi\cdot\bar \xi \quad \mbox{and} \quad
|A\xi\cdot\bar \zeta|\leq \Lambda |\xi||\zeta|, \quad
\forall\,\xi,\zeta\in \CC^n,
\end{equation}

\noindent for some constants $0<\lambda\leq \Lambda <\infty$. For such matrices $A$, our aim in this paper is to present a detailed investigation of Hardy spaces and their duals associated to the
second order divergence form operator
\begin{equation}\label{eq1.3}
Lf:=-{\rm div}(A\nabla f),
\end{equation}

\noindent which we interpret in the usual weak sense via a sesquilinear form.

In the case that $A$ is the $n\times n$ identity matrix (i.e., so that $L$ is the usual Laplacian
$\Delta :=-{\rm div}\cdot \nabla$),
this theory reduces to the classical Hardy space theory of Stein-Weiss \cite{StWe} and Fefferman-Stein
\cite{FeSt}.  For more general operators $L$ whose heat kernel satisfies a pointwise Gaussian
upper bound, an adapted Hardy space theory has been introduced by 
Auscher, Duong and McIntosh \cite{AuDM}, and by Duong and Yan, \cite{DL}, \cite{DL2}. 
In the absence of such pointwise kernel bounds, the theory has been developed more recently
in \cite{AMcR} by Auscher, McIntosh and Russ (when $L$ is the Hodge-Laplace operator on a manifold with doubling measure), and in \cite{HM} by the first two authors of the present paper, for the complex divergence form elliptic operators considered here.  In \cite{AMcR,HM}, 
the pointwise Gaussian bounds are replaced by 
the weaker  ``Gaffney estimates" (cf. (\ref{eq2.19}) and (\ref{eq2.21}) below), 
whose $L^2$ version is a refined parabolic ``Caccioppoli" inequality which may also be proved 
via integration by parts using only ellipticity and the divergence form structure of $L$.    
The present paper may be viewed in part as a sequel to
\cite{HM}, in which we extend results for the case $p=1$ given there, to the case of general $p$
(although we also obtain here some results, pertaining to the characterization of adapted Hardy spaces via Riesz transforms, that are new even in the case $p=1$).
In particular, it is in the nature of our present setting, in which pointwise kernel bounds may fail, that
the Hardy space theory for $p>1$ becomes non-trivial (i.e., the $L$-adapted 
$H^p$ spaces may be strictly smaller than $L^p$, even when $p>1$). We shall return to this point momentarily.  We note also that general non-negative self-adjoint operators satisfying an $L^2$ Gaffney estimate have recently been treated in \cite{HLMMY}.

We now proceed to discuss some relevant history, and to present a more detailed overview of the paper.
In \cite{KatoMain}, the authors solved a long-standing conjecture, known as the
Kato problem,
by identifying the domain of the square root of $L$. More precisely,
they showed that the domain of $\sqrt L$ is the Sobolev space $W^{1,2}(\RR^n)=\{f\in L^2: \,\nabla f\in L^2\}$ with
\begin{equation}\label{eq1.4}
\|\sqrt Lf\|_{L^2(\RR^n)}\approx \|\nabla f\|_{L^2(\RR^n)},
\end{equation}

\noindent In particular, the Riesz transform $\nabla L^{-1/2}$ is
bounded in $L^2(\RR^n)$.

Since then, substantial progress has been made in the development of the $L^p$ theory of
elliptic operators of the type described above.  Let us define 
$$p_-(L) := \inf \{p: \nabla L^{-1/2}:L^p(\RR^n)\longrightarrow L^p(\RR^n)\}.$$
It is now known
that $1\leq p_-(L)
<  2n/(n+2)$ (with $1<p_-(L)$  for some $L$; we shall return to the latter point momentarily),
and that there exists $\eps(L)>0$ such that
\begin{equation}\label{eq1.5}
\nabla L^{-1/2}:L^p(\RR^n)\longrightarrow L^p(\RR^n)\quad
\iff \quad p_-(L) <p< 2+\eps(L),
\end{equation}
(given (\ref{eq1.4}) as a starting point, (\ref{eq1.5}) with $p_-(L)<  2n/(n+2)$  
is established by combining the results and methods of \cite{HoMa} or \cite{BK} 
with those of \cite{AuscherSurvey}; see also \cite{A}, \cite{AT1},
Chapter~4 of \cite{AT}, and  \cite{BK2} for related theory). 
Moreover,  again given (\ref{eq1.4}) as a starting point, one has the reverse inequality
\begin{equation}\label{eq1.Kp}
\|\sqrt Lf\|_{L^p(\RR^n)}\lesssim \|\nabla f\|_{L^p(\RR^n)},\qquad
\mbox{for}\qquad (p_-(L))_* <p< (p_-(L^*))',
\end{equation}
where in general $p_*:=pn/(p+n)$ denotes the ``lower" Sobolev exponent, and as usual 
$p':=p/(p-1)$ is the  exponent dual to $p.$  The case $p<2$ of (\ref{eq1.Kp}) is due to Auscher  \cite{AuscherSurvey}, while  the case $p>2$ is simply dual to
the adjoint version of (\ref{eq1.5}).
Combining (\ref{eq1.5}) and (\ref{eq1.Kp}), we have that
\begin{equation}\label{eq1.Kp-Rp}
\|\sqrt Lf\|_{L^p(\RR^n)}\approx \|\nabla f\|_{L^p(\RR^n)}\quad
\iff \quad p_-(L) <p< 2+\eps.
\end{equation}
One of the main goals of the present 
paper is to understand the sense
in which (\ref{eq1.Kp-Rp}) extends to the range $p\leq p_-(L)$.  This extension may be viewed as solving the Kato problem below the critical exponent $p_-(L)$.  We discuss this question in more detail
in subsection \ref{ss1.2} below;  the proofs are given in Section \ref{s5} (cf. Theorem \ref{t5.1}).

Let us now discuss optimality of the range of $p$ in (\ref{eq1.5}) (hence also that in (\ref{eq1.Kp-Rp})),
for the entire class of $L$ under consideration.   
Even in the case of real symmetric coefficients, 
the upper bound cannot be improved, in general:  for each $p>2$, 
Kenig\footnote{Kenig's example is described in \cite{AT}, Section~4.2.2.} has constructed 
an operator $L$ whose Riesz transform is not bounded in $L^p$. In addition, the 
counterexamples in
\cite{MNP}, \cite{ACT}, \cite{Davies} showed that for some elliptic
operator $L$ satisfying (\ref{eq1.1})--(\ref{eq1.3}) there is a $p\in(1,2)$ such that 
 the Riesz transform is not bounded in $L^p$; i.e., for such $L$, one has $p_-(L) >1$.
 Moreover, the latter fact permeates all the $L^p$ 
 results in the theory: as shown in \cite{AuscherSurvey}, $p_-(L)$ is also the lower bound for 
 the respective intervals of 
 $p$ for which the heat semigroup and
the $L$-adapted square function (cf. (\ref{eq1.8}) below)
are $L^p$ bounded, and for which the semigroup enjoys 
$L^p\to L^2$ off diagonal estimates.  However,
identification of the sharp lower bound 
 $p_-(L)$ remained an open problem 
(posed, along with related questions,  in \cite{AuscherSurvey}, 
Conjecture~3.14, and in \cite{AOpen}, Problem 1.4, Problem 1.5, Problem 1.13).  

In Section \ref{s2} of the present paper, we observe that the example constructed by Frehse in \cite{Frehse} may be used to resolve these remaining sharpness issues, i.e., to show that $p_\pm(L)
= 2n/(n\mp2)\pm\varepsilon_\pm(L)$, where $(p_-(L), p_+(L))$ is the interior of the interval of $L^p$ boundedness of the heat semigroup $e^{-tL}, t>0$. More precisely, we have
\begin{eqnarray}\label{eq1.6}
&&  \forall\, p\not\in [2n/(n+2),2],\,\, \mbox{ $ \exists \,L$ with } \quad \nabla
L^{-1/2}:L^p(\RR^n)\not\!\!\longrightarrow L^p(\RR^n),\\[4pt]
\label{eq1.7}
&&   \forall\,p\not\in [2n/(n+2),2n/(n-2)],\,\, \mbox{ $\exists\,L$ with } \quad 
e^{-tL}:L^p(\RR^n)\not\!\!\longrightarrow L^p(\RR^n).
\end{eqnarray}

\noindent It follows, in particular, that in dimensions 
$n\geq 3$, the kernel of the heat semigroup
may fail to satisfy the pointwise Gaussian estimate  
$$|K_t(x,y)|\leq \, C t^{-n/2}\,e^{-c|x-y|^2/t},\quad t>0 \text{ and } x,y\in\RR^n.$$
This solves an open problem in \cite{AT}, p. 33. 

Thus, in dimensions $n>2$, 
the Riesz transform may fail to be bounded in $L^p$ for some $p\in (1,2)$, as may
the heat semigroup $e^{-tL}$, $t>0$, as well as the
other natural
operators associated with such $L$ (e.g., square function, non-tangential
maximal function).
Consequently, in the case that the endpoint $p_-(L) >1,$
the $L$-adapted Riesz transforms, semigroup and square function cannot be bounded from
the classical Hardy space $H^1$ into $L^1$, since interpolation with the known $L^2$ bound
would then produce a contradiction with (\ref{eq1.6}), (\ref{eq1.7}) (or with 
the analogous statement for the square function).
These operators therefore lie beyond
the scope of the Calder{\'o}n-Zygmund theory and exhibit behavior
different to their counterparts built upon  the \Bk
Laplacian.

By analogy to the classical theory then, this motivates the introduction of a family of $L$-adapted 
Hardy spaces $H^p_L$ for all $0<p<\infty$, {\bf not} equal to $L^p$ in the range $p\leq  p_-(L)$, 
on which the $L$-adapted semigroup, Riesz transforms and square function are well behaved, and which comprise a complex interpolation scale including $L^p$
for $p_-(L)<p<p_+(L).$  We note that the endpoint $p_-(L)$ plays a similar role to the exponent
$p=1$ in the classical theory.

In particular, in Section \ref{s5} we give a suitable Hardy space extension of
(\ref{eq1.5})  to the case  $p\leq p_-(L)$ (the case $p=1$ already appeared
in \cite{HM}), and, in one of the main results of this paper, we present
an appropriate converse, thus obtaining a Riesz transform
characterization of $L$-adapted $H^p$ spaces, for some range of $p$ 
depending on $n$.  As observed above, this characterization may be viewed as a sharp extension of the Kato square root estimate (\ref{eq1.4}), and of its $L^p$ version (\ref{eq1.Kp-Rp}), to the endpoint $p_-(L)$ and below.  In order to make these notions precise, we should first define our adapted $H^p_L$ spaces. 

\subsection{Definition of $H^p_L$} The first step in the development of an $L$-adapted Hardy space theory, in the case that pointwise kernel bounds may fail\footnote{In the presence of {\bf pointwise} Gaussian heat kernel bounds, an $L$-adapted $H^1$ and $BMO$ theory was previously introduce by Duong and Yan \cite{DL}, \cite{DL2}.}, 
was taken in \cite{HM} (and independently in
\cite{AMcR}), in which
the authors considered the model case of $H^1_L(\RR^n)$ and, on
the dual side, the appropriate analogue of the space $BMO$. The definition
of $H^1_L$ given in \cite{HM}\footnote{and in \cite{AMcR} for $H^p_L, p\geq 1$.} (by means of an $L$-adapted square
function) can be extended immediately to $0<p\leq 2$ and with
some additional care to $2\leq p<\infty$ as well. To this end,
consider the square function associated with the heat semigroup
generated by $L$
\begin{equation}\label{eq1.8}
Sf(x)=\left(\dint_{\Gamma(x)}|t^2Le^{-t^2L}f(y)|^2\,\frac
{dydt}{t^{n+1}}\right)^{1/2},\qquad x\in\RR^n,
\end{equation}

\noindent where, as usual, $\Gamma(x)=\{(y,t)\in\RR^n\times(0,
\infty):\,|x-y|<t\}$ is a non-tangential cone with vertex at
$x\in\RR^n$. Analogously to \cite{HM}, we define the space
$H^p_L(\RR^n)$ for $0<p\leq 2$ as  the \Bk completion of
$\{f\in L^2(\RR^n): Sf\in L^p(\RR^n)\}$  in the norm
\begin{equation}\label{eq1.9}
\|f\|_{H^p_L(\RR^n)}:=\|Sf\|_{L^p(\RR^n)}.
\end{equation}

\noindent For $2<p<\infty$ we assign
\begin{equation}\label{eq1.10}
H^p_L(\RR^n):=\left(H^{p'}_{L^*}(\RR^n)\right)^*,
\end{equation}

\noindent where $1/p+1/p'=1$ and $L^*$ is the adjoint operator to
$L$. These spaces also have an appropriate square function
characterization  as will be discussed in Section~\ref{s4}. \Bk

\medskip

\subsection{Riesz Transform characterization of $H^p_L$}\label{ss1.2} We shall show in Section \ref{s5}
that the Riesz transforms are bounded from $H^p_L$ into $L^p$, $0<p<2+\varepsilon(L)$,
and even into classical $H^p, \, n/(n+1) < p \leq 1.$
Conversely, for some restricted range of $p$, we show 
that these estimates are reversible, thus obtaining a Riesz transform characterization of 
the corresponding $H^p_L$.
Let us describe these results in more detail.

As preliminary steps, we establish two results that are also of independent interest: 
in Section \ref{s3}, we shall obtain a molecular decomposition of $H^p_L$
spaces, $0<p\leq 1$, analogous to the classical atomic decompositions of Coifman \cite{Co} and Latter
\cite{La} and in Section \ref{s4}, we observe that the spaces $H^p_L$ 
form a complex interpolation scale, including $L^p$ in the range $p_-(L)<p<p_+(L)$
(see (\ref{eq1.13})). As in the classical case, we are then able to use
these fundamental properties of Hardy spaces  to prove in Section \ref{s5} that
\begin{eqnarray}\label{eq1.11}&\nabla L^{-1/2}: H^p_L(\RR^n)\to L^p(\RR^n)\,,& \qquad
0<p<2+\eps(L)\,,\\\label{eq1.11a}
&\nabla L^{-1/2}: H^p_L(\RR^n)\to H^p(\RR^n)\,,& \qquad
\frac{n}{n+1}<p\leq1\, ,\end{eqnarray} 
where $H^p(\RR^n)$ denotes the classical Hardy space \cite{FeSt}.
Observe that these results extend (\ref{eq1.5}) to 
the range of $p$ below the endpoint $p_-(L)$
(the case $p=1$ has already appeared in \cite{HM}). 
The $H^p_L$ spaces  in (\ref{eq1.11})--(\ref{eq1.11a}) do not, in general, 
coincide with $L^p$ or classical $H^p$ (we recall that $H^p(\RR^n) = L^p(\RR^n)$ if  $1<p<\infty$).
In fact, we can ascertain only that
\begin{eqnarray}\label{eq1.13}
H^p_L(\RR^n)=L^p(\RR^n)\,,&\quad &p_-(L)<p<p_+(L)\,, \\ \label{eq1.13a}
L^2\cap H^p_L\subset L^2\cap H^p\,,&
\quad & n/(n+1)<p\leq p_-(L)\,,
\\ \label{eq1.13b}
L^p(\RR^n)/\mathcal{N}_p(L)\hookrightarrow H^p_L(\RR^n)\,,&
\quad & p\geq p_+(L),
\end{eqnarray}
where $\mathcal{N}_p(L)$ is the null space of $L$ in $L^p(\RR^n)$ (cf. Section \ref{s9} for details).  
In addition, the containments 
in (\ref{eq1.13a})\footnote{We note that 
$L^2\cap H^p_L$ is dense in
$H^p_L(\RR^n)$, so by \eqref{eq1.13a} there is a natural ``embedding" of 
$\HPL$ into $H^p(\RR^n)$ which extends the identity map on a dense subset. 
Intuitively then, one might expect that 
the stronger containment $\HPL \subset H^p(\RR^n)$ should
hold in (\ref{eq1.13a}).  In practice, however,
matters appear to be more subtle, so we present a more detailed discussion of this matter, along with proofs of 
(\ref{eq1.13}) - (\ref{eq1.13b}),
in an Appendix, Section \ref{s9}.} (resp. (\ref{eq1.13b}))
are {\it strict} if $p_-(L) > 1$ 
(resp. $p_+(L)<\infty$).

By contrast, when $L=\Delta$, the space $H^p_\Delta (\RR^n)$ is the
usual Hardy space for $0<p\leq 1$ 
and $L^p$ for
$1<p<\infty$. Hence, (\ref{eq1.11})--(\ref{eq1.11a}) recover the well-known mapping properties of 
$\nabla\Delta^{-1/2}$ in $L^p$ and $H^p$.

Moreover,  we have that $H^1_L=H^1$, and $H^p_L=L^p$ for 
$1<p<\infty$ whenever the heat kernel of $L$ satisfies a Gaussian upper
bound and local Nash type H\"{o}lder continuity (as in (\ref{eq2.15}-\ref{eq2.17})); indeed, in that case the square function (\ref{eq1.8}) is a standard Hilbert space valued Calder\'{o}n-Zygmund operator, which therefore maps $H^1(\RR^n)$ into $L^1(\RR^n)$; whence it follows readily that $H^1(\RR^n)$ embeds continuously into $H^1_L(\RR^n),$ and thus $H^1(\RR^n)=H^1_L(\RR^n)$, by (\ref{eq1.13a}).
The case $p>1$ is obtained by interpolation and duality.  The ``Gaussian" property (\ref{eq2.15}-\ref{eq2.17}) holds always in dimensions $n=1,2,$ and for real coefficients, it holds in all dimensions.
However, as we mentioned earlier, it may fail for complex coefficients in dimensions $n\geq 3$.

We turn now to the matter of characterizing $H^p_L$, for some range of $p\leq p_-(L)$, 
via the Riesz transform operator
$\nabla L^{-1/2}$.  In the classical setting (i.e., $L = \Delta$), the Riesz transform
provided the foundation for the development, beginning in \cite{StWe} and \cite{FeSt},
of the real variable theory of $H^p$, and furnished also a 
link between that theory and PDEs, via sub-harmonic functions.
The classical Riesz transform characterization
says that
\begin{equation}\label{eq1.14}
f \in H^p(\RR^n)\quad\mbox{if and only if } \quad f \in
L^p(\RR^n)\quad\mbox{and}\quad \nabla\Delta^{-1/2}\in L^p(\RR^n),
\end{equation}

\noindent for all $(n-1)/n<p\leq 1$ (assuming some growth restriction at infinity when $p<1$; see, e.g. \cite{St2}, p. 123).
There are analogous, but more complicated results involving higher order Riesz transforms
when $p\leq (n-1)/n$.
Apparently, no such characterization has been obtained for operators substantially different
from the Laplacian (although we mention that some results in this direction have been obtained for
lower order perturbations of the Laplacian \cite{DP,DZ}).

Upon attempting to generalize the Riesz transform characterization to $H^p_L$ spaces, one
immediately encounters several difficulties. The original
argument relied on the subharmonicity of  small powers of the
gradient of a harmonic function. No analogue of such a property
exists (or even makes sense) in our context. In addition, that
(\ref{eq1.14}) holds only for the values of $p$ close to 1
suggests that in our case, in which $\HPL$ is strictly contained in $L^p(\RR^n)$ if $p\leq
p_-(L)$, the Riesz
transform characterization should be proved for $p$ close to
$p_-(L)$. In fact, in Section \ref{s5} of this paper we show that 
\begin{equation}\label{eq1.rieszchar}\HPL = \HRL\,, \qquad 
\frac{p_-(L)n}{n+p_-(L)}<p<2+\eps(L),\end{equation}
where for $p$ in the stated range, $\HRL$ is defined as the completion of 
the set $\{f\in L^2(\mathbb {R}^n):
\nabla L^{-1/2} f \in H^p(\mathbb{R}^n)\},$ 
with respect to the norm 
\begin{equation}\label{e1.hrldef}\|f\|_{\HRL} := \|\nabla L^{-1/2} f\|_{H^p(\mathbb{R}^n)}\end{equation}
(bearing in mind that classical $H^p(\mathbb{R}^n) = L^p(\mathbb{R}^n)$ if $p>1$).
Observe that the lower bound $\frac{p_-(L)n}{n+p_-(L)}>\frac{n-1}{n}$ (cf. \eqref{eq1.14}). 
The equivalence (\ref{eq1.rieszchar}) amounts to proving that for $f\in L^2(\RR^n)$,
\footnote{By definition, $\HPL\cap L^2(\RR^n)$ is dense in
$\HPL$; similarly for $\HRL$.}
\begin{eqnarray}\label{eq1.15}
\|f\|_{H^p_L(\RR^n)}\approx \|\nabla L^{-1/2}f\|_{L^p(\RR^n)},&
\max\left\{1,\frac{p_-(L)n}{n+p_-(L)}\right\}<p<2+\eps(L) ,
\\ \label{eq1.16}
\|f\|_{H^p_L(\RR^n)}\approx \|\nabla L^{-1/2}f\|_{H^p(\RR^n)},&\frac{p_-(L)n}{n+p_-(L)}<p\leq 1.
\end{eqnarray}
We note that (\ref{eq1.15}) and (\ref{eq1.16}) can be viewed as sharp extensions of the Kato square root estimate (\ref{eq1.4}) to the endpoint $p_-(L)$ and below.
\footnote{We remark also that the direction
$\|f\|_{H^p_L(\RR^n)}\lesssim \|\nabla L^{-1/2}f\|_{L^p(\RR^n)}$ of (\ref{eq1.15}) is a sharp
version of the bound
$\|f\|_{L^p(\RR^n)}\lesssim \|\nabla L^{-1/2}f\|_{L^p(\RR^n)}$, proved in \cite{AuscherSurvey}
for the same range of $p$.   Indeed, as mentioned above $\HPL$ may be ``strictly 
smaller" (in the sense of \eqref{eq1.13a}) than $L^p(\RR^n)$.  We shall discuss this point in more detail in Sections \ref{s5} and \ref{s9}.}

Consequently,
for this same range of $p$, (\ref{eq1.16}) together with (\ref{eq1.14}) imply that 
\begin{equation}\label{eq1.17}
\|f\|_{H^p_L(\RR^n)}\approx\|\Delta^{1/2} L^{-1/2}f\|_{H^p(\RR^n)}\approx \|\nabla L^{-1/2}f\|_{L^p(\RR^n)}+\|\Delta^{1/2} L^{-1/2}f\|_{L^p(\RR^n)},
\end{equation}

\noindent for suitable $f$. 
Indeed, since the classical Riesz transforms $\partial_{x_j}\Delta^{-1/2}=
\Delta^{-1/2}\partial_{x_j}$ are
bounded on classical $H^p$, we have that
$$\|\nabla L^{-1/2}f\|_{H^p(\RR^n)}=\|\nabla \Delta^{-1/2}\Delta^{1/2} L^{-1/2}f\|_{H^p(\RR^n)}
\lesssim\|\Delta^{1/2} L^{-1/2}f\|_{H^p(\RR^n)},$$
and by (\ref{eq1.14}), that
\begin{equation*}\|\Delta^{1/2} L^{-1/2}f\|_{H^p(\RR^n)} =
\| \Delta^{-1/2}\,{\rm div}\,\nabla L^{-1/2}f\|_{H^p(\RR^n)}
\lesssim \|\nabla L^{-1/2}f\|_{H^p(\RR^n)}.\end{equation*}

As a consequence of \eqref{eq1.17}, one obtains the following {\it new characterization of the classical Hardy spaces}. Namely, 
\begin{multline}\label{eq1.18}
f \in H^1(\RR^n)\\\mbox{if and only if } \quad \nabla L^{-1/2}f
\in L^1(\RR^n)\quad\mbox{and}\quad \Delta^{1/2} L^{-1/2}\in
L^1(\RR^n),
\end{multline}

\noindent for any operator $L$ whose heat kernel satisfies
Gaussian bounds.

Finally, we remark that in \cite{May}, the second named author has recently developed further the circle of ideas related to the Riesz transform characterization 
of $\HPL$ to establish sharp $L^p$ solvability results
for the regularity problem for the equation $u_{tt} - Lu = 0$ in the half-space $\RR^{n+1}_+.$

\subsection{The Dual of $H^p_L$ , $0<p\leq1$} Another important aspect of the theory is the identification of the
duals of Hardy spaces, and the elaboration of their properties.   In the classical setting,
the duality result for $p=1$ is the celebrated theorem of Fefferman \cite{FeSt};  the
case $0<p<1$ was treated in one dimension by Duren, Romberg and Shields \cite{DRS}, 
and in general by Fefferman and Stein \cite{FeSt}.
Just as $H^1$ provides a substitute for $L^1$ in harmonic
analysis,  so too does the dual of $H^1$, the space of functions with bounded
mean oscillation ($BMO$), substitute for $L^\infty$. Furthermore, the
duals of $H^p$ for $p<1$ are Lipschitz spaces, whose norms measure
fractional order
smoothness. In our setting they can be introduced as
follows.

Let $\alpha$ be a non-negative real number and $M\in\NN$ be such
that $M>\frac 12\left(\alpha+\frac n2\right)$. For $\eps >0$  we
define the space ${\bf M}^{\eps,M}_{\alpha,L}$ as the collection of
all $\mu \in L^2(\mathbb{R}^n)$ such that $\mu$ belongs to the range of $L^k$ in $L^2(\RR^n)$, $k=1,...,M$, and
$$\|\mu\|_{{\bf M}^{\eps,M}_{\alpha,L}} \equiv  \sup_{j\geq 0}
2^{j(n/2+\alpha+\eps)}\sum_{k=0}^M\|L^{-k}\mu\|_{L^2(S_j(Q_0))}<\infty,$$
where $Q_0$ is the unit cube centered at $0$ and $S_j(Q_0)$, $j\in\NN$, are the corresponding dyadic annuli (see (\ref{eq3.2})). We say that an
element
\begin{equation}\label{eq1.19}f \in \cap_{\eps > 0}\left({\bf M}^{\eps,M}_{\alpha,L}\right)^*=: 
{\bf M}^{M,\,*}_{\alpha,L}\end{equation}

\noindent \Bk belongs to the space $\Lambda_{L^*}^\alpha(\RR^n)$
if \footnote{We note that in the presence of a pointwise Gaussian bound, similar spaces
were previously introduced in the work of Duong and Yan \cite{DL,DL2,DL3}.  We shall discuss this point in more detail at the end of this section.}
\begin{equation}\label{eq1.20}
\|f\|_{\Lambda_{L^*}^\alpha(\RR^n)} :=\sup_{Q}\frac{1}{|Q|^{\alpha/n}}\left(\frac{1}{|Q|}\int_Q
\left|(I-e^{-l(Q)^2{L^*}})^Mf(x)\right|^2\,dx\right)^{1/2}<\infty,
\end{equation}
where the supremum runs over all cubes $Q\subset \mathbb{R}^n.$
Here and throughout the paper $|Q|$ stands for the
Euclidean volume of the cube $Q$, and $l(Q)$ denotes its sidelength.
For $\alpha>0$ the spaces $\Lambda_{L^*}^\alpha(\RR^n)$ are the
analogues of the classical Lipschitz spaces,\footnote{ 
Indeed, for $\alpha >0$, the norm in (\ref{eq1.20}) is clearly modeled
on the mean oscillation characterization, due to
N. Meyers \cite{N-M}, of the classical homogeneous ``Lip$_\alpha$" space
$\Lambda^\alpha(\RR^n)$.  For $0<\alpha <1$, we define the latter
to be the space of continuous functions modulo constants, for which the norm $\|\varphi\|_{\Lambda^\alpha(\RR^n)}:= \sup_{x\neq y}\frac{|\varphi(x)-\varphi(y)}{|x-y|^\alpha}<\infty.$}  while the case
$\alpha=0$ corresponds to $BMO$. Accordingly, we denote
$BMO_{L^*}(\RR^n):=\Lambda_{L^*}^0(\RR^n)$.  We refer the reader
to \cite{HM}, where the authors also established some further
properties of $BMO_{L^*}$ such as a Carleson measure characterization and
an analogue of the John-Nirenberg inequality.
\Bk  
In addition, the authors showed in \cite{HM} that $(H^1_L)^*=BMO_{L^*}$. In Section \ref{s3}
of the present paper, we extend this duality as follows:
\begin{equation}\label{eq1.21}
(H^p_L(\RR^n))^*=\Lambda_{L^*}^{\alpha}(\RR^n), \qquad 0<p\leq
1,\qquad \alpha=n(1/p-1).
\end{equation}

\noindent Moreover, the dual of $\Lambda_{L^*}^{\alpha}(\RR^n)$,
in turn, provides an ambient space for $H^p_L$, for the elements
of $H^p_L$, $p<1$, are not necessarily functions, they are linear
functionals on $\Lambda_{L^*}^{\alpha}(\RR^n)$ (recall that the
elements of $H^p$ are tempered distributions).

Finally, as we already mentioned, $H^p_\Delta(\RR^n)=H^p(\RR^n)$
for all $0<p<\infty$, which reduces to $L^p(\RR^n)$ when $p>1$. Then, by
duality, $BMO_\Delta(\RR^n)=BMO(\RR^n)$ and
$\Lambda_{\Delta}^{\alpha}(\RR^n)=\Lambda^{\alpha}(\RR^n)$, the
classical $BMO$ and Lipschitz spaces. In general, one has only the
proper inclusions (\ref{eq1.13a}) 
and on the dual side $BMO(\RR^n)\subset BMO_L(\RR^n)$,
$\Lambda^{\alpha}(\RR^n)\subset\Lambda_L^{\alpha}(\RR^n)$ for $0<\alpha< 1$.

\subsection{The Dual of $H^p_L$ ,  $1<p<2$}  In the case $2<p<\infty$, the spaces
$H^p_L$  were originally defined by 
the duality relationship (\ref{eq1.10}).
We shall give two intrinsic characterizations of these spaces:
one, in Section \ref{s4} (cf. Corollary \ref{c4.4}),
in terms of square functions, analogous to (\ref{eq1.8})--(\ref{eq1.9}), and another one, in 
Section \ref{s6}, in
terms of a variant of the sharp maximal function.   
The former characterization is a consequence of tent space duality, and is similar to the analogous results presented in \cite{AMcR}. The latter is new (although rooted in ideas 
of \cite{FeSt} and also \cite{ChemaSM}), and we discuss it in a bit more detail at this point.

Following \cite{FeSt} and \cite{ChemaSM}, consider the operator
\begin{equation}\label{eq1.22}
{\mathcal M}^{\sharp}f(x) :=\sup_{Q\ni \,x}\left(\frac{1}{|Q|}\int_Q
\left|(I-e^{-l(Q)^2{L}})^Mf(y)\right|^2\,dy\right)^{1/2},\qquad x\in\RR^n,
\end{equation}

\noindent where $M\in\NN$ and $\sup_{Q\ni \,x}$ is the supremum over all cubes in $\RR^n$ containing $x$.
We shall refer to ${\mathcal M}^{\sharp}$ as the sharp maximal operator and write ${\mathcal M}_M^{\sharp}$ to underline the dependence on $M$ whenever necessary. By definition, we have that
$f\in {\bf M}^{M,\,*}_{0,L}$, $M>n/4$, belongs to the space $BMO_L(\RR^n)$ if and only if ${\mathcal M}^{\sharp}f\in L^{\infty}(\RR^n)$.
In the current paper we show that an analogous characterization holds for all spaces in the Hardy-BMO scale when $p>2$. That is, roughly speaking, for $2<p<\infty$, we have
$f\in H^p_L(\RR^n)$ if and only if ${\mathcal M}^{\sharp}_Mf
\in L^p(\RR^n),\, M>n/4$, and
\begin{equation}\label{eq1.23}
\|f\|_{H^p_L(\RR^n)}\approx \|{\mathcal M}^{\sharp}_Mf\|_{L^p(\RR^n)},\quad M>n/4.
\end{equation}
We shall prove a precise version of this statement in Section \ref{s6}.

\subsection{ Sobolev spaces and fractional powers of $L$}  The last topic that we shall treat, in Sections
\ref{s7} and \ref{s8}, concerns the adapted $H^p_L$ spaces and their relationship to the behavior of $L$ in classical Sobolev spaces. In fact, we find a {\it complete range of all Sobolev spaces} which naturally interact with the operators associated to $L$, and one of the major ingredients in the argument is the Riesz transform characterization of $H^p_L$.  Let us describe these results in more detail.

We first prove in Section \ref{s7} 
that the fractional powers of $L$ satisfy
\begin{equation}\label{eq1.12}
L^{-\alpha}:H^p_L(\RR^n)\longrightarrow H^r_L(\RR^n), \qquad
\alpha=\frac 12\left(\frac np-\frac nr\right),\quad 0<p<r<\infty,
\end{equation}
thereby extending
the mapping properties of $L^{-\alpha}$ in $L^p$ (cf.
\cite{AuscherSurvey}, Proposition~5.3) to the range of $p$  beyond $(p_-(L),p_+(L))$.

In Section \ref{s8}, we then consider the action of operators associated to $L$
in the classical Sobolev spaces.
As is customary,  we define the homogeneous Sobolev spaces 
$\dot W^{1,p}(\RR^n),\, 1\leq p<\infty,$ to be the completion of 
$C_0^\infty(\RR^n)$ in
the seminorm
\begin{equation}\label{eq5.36*}\|f\|_{\dot{W}^{1,p}(\RR^n)}=\|\nabla f\|_{L^p(\RR^n)}.\end{equation}
More generally (except for the case $p=1$),
we let $\dot W^{s,p}(\RR^n),\, 1< p<\infty,$  denote the completion of $C_0^\infty(\RR^n)$ in
the seminorm \begin{equation}\label{eq5.36}\|f\|_{\dot
W^{s,p}(\RR^n)}=\|\Delta^{s/2}f\|_{L^p(\RR^n)},\qquad s>0,\end{equation} 
and set $\dot W^{-s,p}(\RR^n)=(\dot W^{s,p'}(\RR^n))^*$, $\frac 1p+\frac1{p'}=1$.   

Consider first the case $n\geq 5$.
We prove that for any operator $L$ defined in (\ref{eq1.1})--(\ref{eq1.3}), for every function $\varphi$ holomorphic in a certain sector of a complex plane $\Sigma_\mu^0$ (the exact definitions will be given in the body of the paper), and for every $ f\in \dot W^{\alpha,p}(\RR^n),$
\begin{equation}\label{eq1.24}
\|\varphi(L)f\|_{\dot W^{\beta,q}(\RR^n)}\leq C
\left\|z^{\frac{\beta-\alpha}{2}+\frac 12\left(\frac np-\frac
nq\right)}\varphi\right\|_{L^\infty(\Sigma_\mu^0)}\|f\|_{\dot W^{\alpha,p}(\RR^n)},
\end{equation}
 provided that the function $z\mapsto z^{\frac{\beta-\alpha}{2}+\frac 12\left(\frac np-\frac
nq\right)}\varphi(z)$ belongs to $L^\infty(\Sigma_\mu^0)$ and the indices $\alpha,\beta,p\leq q$ are such that 
the points $(\beta,1/q)$ and $(\alpha,1/p)$ belong to the closed
region ${\mathcal R}_1$, depicted on Figure~1.

%%%%%%%%%%%%%%%%%%%%%%%%%%%%%%%%%%%%%%%%%%%%%%%% picture R_1 %%%%%%%%%%%%%%%%%
\setlength{\unitlength}{0.4 cm}

\begin{picture}(14,19)(-7,-3)

\thinlines
\put(-4,2){\line(1,0){24}}    % x-Coordinate axis
\put(8,1){\line(0,1){13}}    % y-Coordinate axis

\put(7.92,13.8){\vector(0,1){0}}     %
\put(8.08,13.8){\vector(0,1){0}}      %  Arrow on the y-Coordinate axis
\put(8,14){\vector(0,1){0}}           %  (made up of 3 triangles)

\put(19.8,1.92){\vector(1,0){0}}     %
\put(19.8,2.08){\vector(1,0){0}}      %  Arrow on the x-Coordinate axis
\put(20,2){\vector(1,0){0}}           %  (made up of 3 triangles)

\multiput(8.00,7.0)(0.5,0){21}{\line(-1,0){0.25}}    % horizontal dashed line
\multiput(8.00,7.0)(-0.5,0){20}{\line(-1,0){0.25}}    % horizontal dashed line
%on the level 1/2

\put(8.00,12.00){\circle*{0.2}}  %point 1 on y-axis
\put(8.00,7.00){\circle*{0.2}} %point 1/2 on y-axis

\put(18.00,2){\circle*{0.2}} %point 1 on x-axis
\put(-2.00,2.00){\circle*{0.2}} %point -1 on x-axis

%vertices
%high
\put(-2.00,7){\circle*{0.3}} %point B
\put(18.00,7){\circle*{0.3}} %point D

%low
\put(-2.00,3.67){\circle*{0.3}} %point A
\put(18.00,10.33){\circle*{0.3}} %point C
\put(8.00,5.32){\circle*{0.2}} %point F
\put(8.00,8.67){\circle*{0.2}} %point E

%labels on axis
\put(8.1,7.5){\scriptsize{{$\frac{\rm 1}{\rm 2}$}}}
\put(-2,1){\rm\scriptsize -1} \put(18,1){\rm \scriptsize 1}
\put(8.2,12){\rm \scriptsize 1}
\put(8.15,13.5){$\textstyle{{\frac{1}{p}}}$}
\put(19.5,1.25){$\textstyle{s}$}

%labels with formulas
\put(16,11){\scriptsize{{$C=\left(1,\frac{n+4}{2n}\right)$}}}
\put(-4,2.8){\scriptsize{{$A=\left(-1,\frac{n-4}{2n}\right)$}}}
\put(4.3,9.3){\scriptsize{{$E=\left(0,\frac{n+2}{2n}\right)$}}}
\put(8.1,4.5){\scriptsize{{$F=\left(0,\frac{n-2}{2n}\right)$}}}
\put(-4,7.6){\scriptsize{{$B=\left(-1,\frac{1}{2}\right)$}}}
\put(16,6){\scriptsize{{$D=\left(1,\frac{1}{2}\right)$}}}

\linethickness{0.12\unitlength} \thicklines
\put(-2,7){\line(0,-1){3.33}}            % vertical1
\put(18,7){\line(0,1){3.33}}            % vertical1

\thicklines
\put(-2,7){\line(6,1){20}}            % high slant
\put(18,7){\line(-6,-1){20}}            % low slant

\put(-3.00,-1.25){{\bf Figure 1 -- the region ${\mathcal R}_1$.}}

\end{picture}
%%%%%%%%%%%%%%%%%%%%%%%%%%%%%%%%%%%%%%%% end picture %%%%%%%%%%%%%%%%%%%%%%%%

In particular, for every $t>0$
\begin{equation}\label{eq1.25}
e^{-tL}:\dot W^{\alpha,p}(\RR^n)\longrightarrow \dot
W^{\alpha,p}(\RR^n), \,\,\mbox{if}\,\,(\alpha,1/p)\in {\mathcal R}_1,
\end{equation}

\noindent and 
\begin{equation}\label{eq1.26}
L^{-s}:\dot W^{\alpha,p}(\RR^n)\longrightarrow \dot
W^{\beta,q}(\RR^n), \,\,\mbox{if}\,\, (\beta,1/q)\in {\mathcal R}_1\,\,\mbox{and}\,\,(\alpha,1/p)\in {\mathcal R}_1,
\end{equation}

\noindent with $s=\frac{\beta-\alpha}{2}+\frac 12\left(\frac
np-\frac nq\right)$, $p\leq q$.

The region ${\mathcal R}_1$ is closed and is also sharp, in the sense that for every pair $\alpha, p$ such that $(\alpha,1/p)\not\in {\mathcal R}_1$ there is an operator $L$ for which the property (\ref{eq1.25}) is not satisfied and hence, (\ref{eq1.24}) is not generally satisfied. 

Furthermore, all the results in (\ref{eq1.24})--(\ref{eq1.26}) have analogues for $n\leq 4$. In this case $\frac{2n}{n+4}\leq 1$, and just as the classical Hardy spaces provide a natural extension of $L^p$ to the range $p\leq 1$,  so too do the Triebel-Lizorkin (or ``$H^p$ Sobolev") spaces $\dot F_s^{p,2}$ extend 
$\dot W^{s,p}$ in this range;  i.e., the spaces $\dot F_s^{p,2}$ coincide with $\dot W^{s,p}$ when $p>1$ and otherwise naturally extend the Sobolev scale to small values of $p$. We prove that 
\begin{equation}\label{eq1.27}
\|\varphi(L)f\|_{\dot F_{\beta}^{q,2}(\RR^n)}\leq C
\left\|z^{\frac{\beta-\alpha}{2}+\frac 12\left(\frac np-\frac
nq\right)}\varphi\right\|_{L^\infty(\Sigma_\mu^0)}\|f\|_{\dot F_{\alpha}^{p,2}(\RR^n)},\quad \forall f\in \dot F_{\alpha}^{p,2}(\RR^n),
\end{equation}

\noindent provided that the function $z\mapsto z^{\frac{\beta-\alpha}{2}+\frac 12\left(\frac np-\frac
nq\right)}\varphi(z)$ belongs to $L^\infty(\Sigma_\mu^0)$ and the indices $\alpha,\beta,p\leq q$ are such that 
the points $(\beta,1/q)$ and $(\alpha,1/p)$ belong to the region ${\mathcal R}_2$, depicted on Figure~2. 
In particular, the analogues of (\ref{eq1.25})--(\ref{eq1.26}) hold in this context as well. Moreover, all the results are once again sharp, in the sense that for every point outside of the region  ${\mathcal R}_2$ even the heat semigroup is not necessarily bounded in the corresponding Triebel-Lizorkin space.

%%%%%%%%%%%%%%%%%%%%%%%%%%%%%%%%%%%%%%%%%%%%%%%% picture R_2 %%%%%%%%%%%%%%%%%
\setlength{\unitlength}{0.4 cm}

\begin{picture}(14,19)(-7,-3)

\thinlines
\put(-4,2){\line(1,0){24}}    % x-Coordinate axis
\put(8,1){\line(0,1){13}}    % y-Coordinate axis

\put(7.92,13.8){\vector(0,1){0}}     %
\put(8.08,13.8){\vector(0,1){0}}      %  Arrow on the y-Coordinate axis
\put(8,14){\vector(0,1){0}}           %  (made up of 3 triangles)

\put(19.8,1.92){\vector(1,0){0}}     %
\put(19.8,2.08){\vector(1,0){0}}      %  Arrow on the x-Coordinate axis
\put(20,2){\vector(1,0){0}}           %  (made up of 3 triangles)

\multiput(8.00,7.0)(0.5,0){21}{\line(-1,0){0.25}}    % horizontal dashed line
\multiput(8.00,7.0)(-0.5,0){20}{\line(-1,0){0.25}}    % horizontal dashed line
%on the level 1/2

\multiput(8.5,12.05)(0.5,0){20}{\line(-1,0){0.25}}    % horizontal dashed line on level 1

\put(8.00,12.05){\circle*{0.2}}  %point 1 on y-axis
\put(8.00,7.00){\circle*{0.2}} %point 1/2 on y-axis

\put(18.00,2){\circle*{0.2}} %point 1 on x-axis
\put(-2.00,2.00){\circle*{0.2}} %point -1 on x-axis

%vertices
%high
\put(-2.00,7){\circle*{0.3}} %point B
\put(18.00,7){\circle*{0.3}} %point D

%low
\put(-2.00,2){\circle*{0.3}} %point A
\put(18.00,13.67){\circle*{0.3}} %point C
\put(8.00,3.67){\circle*{0.2}} %point F
\put(8.00,10.33){\circle*{0.2}} %point E
\put(3,2){\circle*{0.2}} %point \widetilde{F}

%labels on axis
\put(8.1,7.5){\scriptsize{{$\frac{\rm 1}{\rm 2}$}}}

\put(18,1){\rm \scriptsize 1} \put(8.1,12.2){\rm \scriptsize 1}
\put(8.15,13.5){$\textstyle{{\frac{1}{p}}}$}
\put(19.5,1.25){$\textstyle{s}$}

%labels with formulas
\put(16,14.2){\scriptsize{{$C=\left(1,\frac{n+4}{2n}\right)$}}}
\put(-2.8,1.2){\scriptsize{{$\widetilde A=\left(-1,0\right)$}}}
\put(4.3,10.8){\scriptsize{{$E=\left(0,\frac{n+2}{2n}\right)$}}}
\put(8.1,2.83){\scriptsize{{$F=\left(0,\frac{n-2}{2n}\right)$}}}
\put(-4,7.6){\scriptsize{{$B=\left(-1,\frac{1}{2}\right)$}}}
\put(16,5.8){\scriptsize{{$D=\left(1,\frac{1}{2}\right)$}}}
\put(2.2,1.2){\scriptsize{{$\widetilde{F}=\left(\frac{2-n}{2},0\right)$}}}

\linethickness{0.12\unitlength} \thicklines
\put(-2,7){\line(0,-1){5}}            % vertical1
\put(18,7){\line(0,1){6.67}}            % vertical1

\thicklines
\put(-2,7){\line(3,1){20}}            % high slant
\put(18,7){\line(-3,-1){15}}            % low slant
\put(3,2){\line(-1,0){5}}              % low horizontal

\put(-3.00,-1.25){{\bf Figure 1 -- the region ${\mathcal R}_2$.}}

\end{picture}
%%%%%%%%%%%%%%%%%%%%%%%%%%%%%%%%%%%%%%%% end picture %%%%%%%%%%%%%%%%%%%%%%%%

The study of the properties of the operators associated to $L$ in Sobolev spaces stems from the work of P.~Auscher in \cite{AuscherSurvey} (Sections~5.3, 5.4). Our results extend the theorems in \cite{AuscherSurvey} in several directions: to the range of $p$ beyond the range of $L^p$-boundedness of the heat semigroup (i.e. to the cases $p<p_-(L)<2n/(n+2)$ and $p>p_+(L)>2n/(n-2)$),  and in particular to $p\leq 1$, and are accompanied by the negative results which lead to sharpness of the obtained range of indices.  In particular, we resolve the conjecture posed at the end of Section~5 in \cite{AuscherSurvey}. 

The results we describe in this paper generalize most of
the important aspects of the real variable Hardy space theory to
a context in which the standard tools of the Calder{\'o}n-Zygmund  theory are
not applicable. Besides the aforementioned works \cite{AMcR} and
\cite{HM}, some properties of the Hardy and BMO spaces associated with different operators were introduced previously in \cite{AR}, \cite{DL},
\cite{DL2}, \cite{Lixin}. 

In particular, we note that the theory of $L$-adapted $H^1$ and $BMO$
spaces, including an appropriate analogue of Fefferman's duality theorem,
originates in the work of Duong and Yan \cite{DL}, \cite{DL2}
who treated the case that the associated heat kernel satisfies a pointwise Gaussian
bound.  Their $BMO$ norm is the same as that in (\ref{eq1.20}), with $\alpha = 0$ and $M=1$,
and they have also considered Morrey-Campanato type spaces corresponding to the case
$\alpha >0$ \cite{DL3}.  
As we have observed above, the theory and techniques of the present paper, 
which we develop in the absence of pointwise kernel bounds, assuming only decay estimates of ``Gaffney" type, are necessarily somewhat different.

We note also that, while this manuscript was in preparation, we learned that some of
the results presented here in the case $0<p<1$ have been obtained independently by
R. Jiang and D. Yang \cite{JY} (molecular decomposition, duality, and some mapping properties of linear and non-negative sublinear operators in spaces with integrability $0<p<1$).  As mentioned above, the case $p=1$ was already treated in
\cite{HM} (and in \cite{AMcR}, in a somewhat different context).  
Our main results in the case $p>1$, as well as our 
Riesz transform characterization (\ref{eq1.rieszchar}), appear to be unique to this paper.\footnote
{Although as mentioned above, our tent space/square function definition of 
adapted $H^p$ spaces with $p>1$ follows that given in \cite{AMcR}.}

\medskip

\noindent{\it Acknowledgements}. A part of this work was conducted during the second author's visit to  the Centre for Mathematics and its Applications at the Australian National University, and during the
third author's visit to the University of Missouri. The
authors thank the faculty and staff of those institutions for their warm hospitality.   We also thank
Pascal Auscher for reading a draft of the manuscript and for offering several helpful suggestions.
In addition, the third named author thanks Prof. Auscher and Andrew Morris for illuminating discussions.

We thank Dachun Yang for making a preliminary version of 
his joint work with R. Jiang \cite {JY} available to us while this manuscript was in preparation. 
As mentioned above, their results and ours, which overlap substantially in the case
$0<p<1$, have been obtained independently,
but we have incorporated a simplification introduced in \cite{JY} in the proof of the duality result
for $0<p<1$ (cf. Step II of Theorem \ref{t3.5} below).  Our original proof here had been based on the more complicated argument in \cite{HM}.

Finally, the authors would like to thank the referee for careful reading of the paper and numerous helpful suggestions.

\section{The heat semigroup and functions of $L$ in $L^p$.}\label{s2}
\setcounter{equation}{0}

\subsection{Definitions and $L^2$ theory}\label{s2.1}

Let $L$ be a second order elliptic operator satisfying
(\ref{eq1.1})--(\ref{eq1.3}) viewed as an accretive
operator in $L^2(\RR^n)$. There exists some $\omega\in[0,\pi/2)$
such that the operator $L$ is of type $\omega$ on $L^2(\RR^n)$. In
particular, $-L$ generates a complex semigroup which extends to an
analytic semigroup $\{e^{-zL}\}_{z\in \Sigma_{\pi/2-\omega}^0}$ on $L^2(\RR^n)$.
Here
\begin{eqnarray}\label{eq2.1}
\Sigma_\mu^0 & := & \{z\in\CC\setminus\{0\}:\,|\arg z|<\mu\},\quad \mu\in(0,\pi).
\end{eqnarray}

Furthermore, $L$ has bounded holomorphic functional calculus on
$L^2(\RR^n)$ (see \cite{Mc} and \cite{ADM}).
To be more precise, let us define
\begin{eqnarray}\label{eq2.2}
H^\infty(\Sigma_\mu^0)&:=&\{\psi:\Sigma_\mu^0\to\CC:\,\psi \mbox{ is
analytic and }\|\psi\|_{L^\infty(\Sigma_\mu^0)}<\infty\},\\[4pt]
\Psi_{\sigma,\tau}(\Sigma_\mu^0)&:=&\{\psi:\Sigma_\mu^0\to\CC:\,\psi
\mbox{ is
analytic and } \nonumber \\[4pt]
&&\quad |\psi(\xi)|\leq
C\inf\{|\xi|^\sigma,|\xi|^{-\tau}\}\,\,{\mbox{for every}}\,\xi\in
\Sigma_\mu^0\}.\label{eq2.3}
\end{eqnarray}

\noindent Alternatively, one can say that
\begin{equation}\label{eq2.4}
\psi\in \Psi_{\sigma,\tau}(\Sigma_\mu^0)\quad\Longleftrightarrow \quad \psi \in H^\infty(\Sigma_\mu^0)\mbox{ and  $|\psi(\xi)|\leq C\,\frac{|\xi|^\sigma}{1+|\xi|^{\sigma+\tau}}$},\quad \sigma,\tau>0.
\end{equation}

\noindent Whenever $\psi\in H^\infty(\Sigma_\mu^0)$
\begin{equation}\label{eq2.5}
\|\psi(L)f\|_{L^2(\RR^n)}\leq C \|\psi\|_{L^\infty(\Sigma_\mu^0)}
\|f\|_{L^2(\RR^n)} \quad \mbox{for every}\quad f\in L^2(\RR^n).
\end{equation}

Let $\Psi(\Sigma_\mu^0):=\cup_{\sigma,\tau>0} \Psi_{\sigma,\tau}(\Sigma_\mu^0)$.
If $\psi\in \Psi(\Sigma_\mu^0)$ then $\psi(L)$ can be represented as
\begin{equation}\label{eq2.6}
\psi(L)=\int_{\Gamma_+}e^{-zL}\eta_+(z)\,dz+\int_{\Gamma_-}e^{-zL}\eta_-(z)\,dz,
\end{equation}

\noindent where
\begin{equation}\label{eq2.7}
\eta_{\pm}(z)=\frac{1}{2\pi i}\int_{\gamma_{\pm}}e^{\xi z}\psi(\xi)\,d\xi,\quad z\in \Gamma_{\pm},
\end{equation}

\noindent and $\Gamma_{\pm}=\RR^+e^{\pm i(\pi/2-\theta)}$, $\gamma_{\pm}=\RR^+e^{\pm i\nu}$, $\omega<\theta<\nu<\mu<\pi/2$. In general, when  $\psi\in H^\infty(\Sigma_\mu^0)$, $\psi(L)$ can be defined using (\ref{eq2.6})--(\ref{eq2.7}) and a limiting procedure (see \cite{AuscherSurvey}, Chapter~2, and references therein).

Finally, let us introduce 
\begin{eqnarray}
\Psi'_{\sigma,\tau}(\Sigma_\mu^0)&=&\{\psi:\Sigma_\mu^0\to\CC:\,\psi
\mbox{ is
analytic and there are some } \sigma, \tau, C> 0 \nonumber \\[4pt]
&&\quad \mbox{ such that }|\psi(\xi)|\leq
C\sup\{|\xi|^\sigma,|\xi|^{-\tau}\}\,\,{\mbox{for every}}\,\xi\in
\Sigma_\mu^0\}.\label{eq2.8}
\end{eqnarray}

\noindent For every $\psi\in \Psi'_{\sigma,\tau}$ one can define an unbounded operator $\psi(L)$ on $L^2(\RR^n)$ following the procedure in \cite{Mc}. In particular, the fractional powers of $L$ arise in this way.

\subsection{$L^p$ boundedness of the heat semigroup: sharp results}\label{s2.2}

Following \cite{AuscherSurvey}, let
us denote by ${\mathcal J}(L)$ the maximal interval of exponents
$p\in[1,\infty]$ for which the heat semigroup $\{e^{-tL}\}_{t>0}$ is
$L^p$-bounded and let us write $int\,{\mathcal J}(L)=(p_-(L),p_+(L))$.
It was proved in \cite{AuscherSurvey} (Sections~3.2 and 4.1) that
\begin{equation}\label{eq2.9}
{\textstyle{p_-(L) < \frac{2n}{n+2}}}\qquad\mbox{and} \qquad
{\textstyle{ p_+(L)>\frac{2n}{n-2},}}
\end{equation}

\noindent for $L$ as in (\ref{eq1.1})--(\ref{eq1.3}), and that $p_-(L)$ is also the lower bound
for the interval of $p$ for which $\nabla L^{-1/2}:L^p\to L^p$ (hence this notation is consistent 
with that in Section \ref{s1}). We shall show that the bounds in
(\ref{eq2.9}) are sharp, in the following sense.

\begin{proposition}\label{p2.1} Given any $\widetilde p_- $ with $1\leq \widetilde{p}_-
<\frac{2n}{n+2}$ there
exists an operator $L$ such that the heat semigroup
$\{e^{-tL}\}_{t>0}$ is not bounded in $L^{\widetilde p_- }$. And
similarly, given any $\widetilde p_+ $ with $\frac{2n}{n-2}<\widetilde{p}_+ \leq \infty,$ there exists
an operator $L$ such that the heat semigroup $\{e^{-tL}\}_{t>0}$ is
not bounded in $L^{\widetilde p_+}$.
\end{proposition}

\bp We argue as in \cite{AT}, Section~1.3, but using the example of
\cite{Frehse} rather than that of \cite{MNP}.

Let $n\geq 3$. By \cite{Frehse}, for every $q<n/2$ and $\lambda>0$, there is an $n\times
n$ matrix $A=A(q,\lambda)$ satisfying (\ref{eq1.1})--(\ref{eq1.2})
and such that
\begin{equation}\label{eq2.10}
u=\frac{x_1}{|x|^q}\,e^{i\lambda \ln |x|}
\end{equation}

\noindent is a classical solution of the equation $Lu=-{\rm div}(A\nabla u)=0$ in
$\RR^n\setminus\{0\}$, and is a weak solution globally in $\RR^n$.

More precisely, $A$ has a form
\begin{equation}\label{eq2.11}
A=\left\{(\alpha+i)\delta_{jk}+\beta
\frac{x_jx_k}{|x|^2}\right\}_{j,k=1}^n,
\end{equation}

\noindent where $\alpha\in\RR$ and $\beta\in\CC$ are some constants.
For any fixed $\alpha\in\RR$, $\lambda\neq 0$, $q\neq 0$ there
exists $\beta=\beta(\alpha,q,\lambda)$ (explicitly written in
\cite{Frehse}) such that $u$ in (\ref{eq2.10}) solves the equation
$-{\rm div}(A\nabla u)=0$, and moreover, for $q<n/2$, $\lambda>0$,
$\alpha>0$ sufficiently small and $\beta=\beta(\alpha,q,\lambda)$,
the corresponding  matrix $A$ satisfies the ellipticity conditions.

Now let us return to the properties of the heat semigroup. First of
all, take some $\phi \in C_0^\infty(\RR^n)$, supported in the unit
ball $B_1$, such that $\phi=1$ in the ball of radius $1/2$ centered
at the origin. Then $\nabla \phi \in C_0^\infty(B_1)$ and $\nabla
\phi=0$ in a neighborhood of 0. Since the only singularity of $u$
(and of $A$) is at 0, we have
\begin{equation}\label{eq2.12}
L(u\phi)= -{\rm div}(A\nabla (u\phi))=-{\rm div}(Au\nabla
\phi)-A\nabla u\cdot\nabla \phi =:f\in C_0^\infty(B_1),
\end{equation}

\noindent where the second equality follows from the fact that
$Lu=0$.

Fix some $\widetilde p_+ >\frac{2n}{n-2}$ and assume that the heat
semigroup $\{e^{-tL}\}_{t>0}$ is bounded in $L^{\widetilde p_+}$ for
an operator $L$. Then, according to \cite{AuscherSurvey}, Proposition~5.3, we have
\begin{equation}\label{eq2.13}
L^{-1}:L^p(\RR^n)\longrightarrow L^r(\RR^n), \quad n/p-n/r=2,
\end{equation}

\noindent provided $r<\widetilde p_+$ and $p>p_-(L)$. But since
$p_-(L)$ is always smaller than $\frac{2n}{n+2}$, (\ref{eq2.13}) is
valid for any $\frac{2n}{n-2}\leq r<\widetilde p_+$.

The function $f\in C_0^\infty(B_1)$ in the right-hand side of
(\ref{eq2.12}) belongs, in particular, to all $L^p$ spaces, $1\leq
p\leq\infty$, and therefore, by (\ref{eq2.13}) the solution
\begin{equation}\label{eq2.14}
\mbox{$L^{-1}f=u\phi$ must belong to all $L^r$,\qquad
$\frac{2n}{n-2}\leq r<\widetilde p_+$.} \end{equation}
However, $u\phi=u$ in a neighborhood of the origin and $u$ given by
(\ref{eq2.10}) does not belong to $L^r$ when $r(1-q)+n<0$.
We can take $\eps>0$ sufficiently small so that
$2n/(n-2-2\eps)<\widetilde p_+$ and take $q=n/2-\eps$. Then
$u\phi \not\in L^r$ for any $r>2n/(n-2-2\eps)$ which
contradicts (\ref{eq2.14}).

Since $p_-(L)=(p_+(L^*))'$, this computation also shows that assuming
boundedness of $\{e^{-tL}\}_{t>0}$ in $L^{\widetilde p_- }$ for all
$L$ will lead to a contradiction. \ep

Let $L$ be a divergence form elliptic operator with complex bounded coefficients given by (\ref{eq1.1})--(\ref{eq1.3}). Let $K_t(x,y)$, $t>0$, $x,y\in\RR^n$, denote the Schwartz kernel of the heat semigroup generated by $L$. We say that it satisfies the Gaussian property if for each $t>0$ the kernel $K_t(x,y)$ is H\"older continuous in $x$ and $y$ and there exist  some constants $C,c,\alpha>0$ such that for every $x,y,h\in\RR^n$
\begin{eqnarray}\label{eq2.15}
|K_t(x,y)|&\leq & \frac{C}{t^{n/2}}\,e^{-\frac{|x-y|^2}{ct}}, \\[4pt]
\label{eq2.16}
|K_t(x,y)-K_t(x+h,y)|&\leq& \frac{C}{t^{n/2}}\,\left(\frac{|h|}{t^{1/2}+|x-y|}\right)^{\alpha}\,
e^{-\frac{|x-y|^2}{ct}}, \\[4pt]
\label{eq2.17}
|K_t(x,y)-K_t(x,y+h)|&\leq& \frac{C}{t^{n/2}}\,\left(\frac{|h|}{t^{1/2}+|x-y|}\right)^{\alpha}\,
e^{-\frac{|x-y|^2}{ct}}, 
\end{eqnarray}
\noindent whenever $2|h|\leq t^{1/2}+|x-y|$.
For every elliptic operator defined in (\ref{eq1.1})--(\ref{eq1.3}) the heat kernel satisfies the Gaussian bounds in dimensions $n=1,2$, and for every elliptic operator with real coefficients this property holds in all dimensions. It was known that in general the Gaussian bounds may fail in dimensions $n\geq 5$. Whether or not they necessarily hold when $n=3,4$ has been an open problem (see, e.g., \cite{AT}, \S{1.2} and the Remark on p. 33). The Corollary below answers this question to the negative.

\begin{corollary}\label{c2.2} Let $n\geq 3$.
There exists an elliptic operator $L$ given by (\ref{eq1.1})--(\ref{eq1.3}) such that the kernel of the heat semigroup generated by $L$ does not satisfy (\ref{eq2.15}). In particular, for such $L$ the Gaussian property does not hold.
\end{corollary}

\bp The estimate (\ref{eq2.15})
implies that the integral kernel $G(x,y)$, $x,y\in\RR^n$, of the operator $L^{-1}=\int_0^{\infty} e^{-tL}\,dt$ is controlled by $C|x-y|^{2-n}$.  Hence, (\ref{eq2.13}) holds, which yields as before a contradiction. An analogous argument was used in \cite{AT}, \S{1.3}. 

Alternatively, one could check directly that (\ref{eq2.15}) implies the $L^p$ boundedness of the heat semigroup for all $1\leq p\leq \infty$. Indeed, the boundedness in $L^1$ follows applying the Fubini theorem to the $L^1$ norm of $e^{-tL}f$ and integrating the upper bound of the kernel, given by (\ref{eq2.15}), in $x$. The boundedness in $L^\infty$ is also trivial, since bringing out the $L^\infty$ norm of $f$ in an integral expression for $e^{-tL}f$, one just ends up with the integral of the right-hand side of (\ref{eq2.15}) in $y$. The range $1<p<\infty$ then follows by interpolation. However, the $L^p$ boundedness of the heat semigroup for all $1\leq p\leq \infty$ contradicts Proposition~\ref{p2.1}. We thank the referee for pointing out this, perhaps simpler, route.  
\ep

\begin{corollary}\label{c2.3} For each
$p<\frac {2n}{n+2}$ and each $p>2$ there exists $L$ such that $\nabla L^{-1/2}$
is not bounded in $L^{p}$. 
\end{corollary}

\bp The counterexample for $p>2$ is due to C.\,Kenig (see \cite{AT}, Section 4.2.2). The case $p<\frac {2n}{n+2}$ follows from Proposition~\ref{p2.1}
along with the fact, proved in \cite{AuscherSurvey} and noted above, that the lower endpoint of the interval of boundedness of Riesz transform coincides with the lower endpoint of the interval of boundedness of the heat semigroup .\ep

\subsection{Off-diagonal estimates and $L^p-L^q$ bounds}\label{s2.3}

We say that  a \Bk family of operators $\{S_t\}_{t>0}$
satisfies $L^2$ off-diagonal estimates (``Gaffney estimates") if
there are some constants $c, C >0$ such that for arbitrary
closed sets $E,F\subset \RR^n$
\begin{equation}\label{eq2.19}
\|S_tf\|_{L^2(F)}\leq C\,e^{-\frac{{\rm
dist}\,(E,F)^2}{ct}}\,\|f\|_{L^2(E)},
\end{equation}

\noindent for every $t>0$ and every $f\in L^2(\RR^n)$ supported in
$E$.  Similarly,
a family $\{S_z\}_{z\in \Sigma_{\mu}^0}$, $0<\mu<\pi/2$, satisfies $L^2$ off-diagonal estimates in $z$ if the analogue of (\ref{eq2.19}) holds with $|z|$ in place of $t$ on the right-hand side. For example, 
if $0<\mu <\pi/2-\omega$, 
the families $\{e^{-zL}\}_{z\in \Sigma_{\mu}^0}$ and 
$\{(zL)^k e^{-zL}\}_{z\in \Sigma_{\mu}^0},\, 
k=1,2,...,$  satisfy $L^2$ off-diagonal estimates in $z$ (see \cite{AuscherSurvey}, \S{2.3}). For  later reference we record the following result.

\begin{lemma}\label{l2.4}\, (\cite{HoMa})
If two families of operators, $\{S_t\}_{t>0}$ and $\{T_t\}_{t>0}$,
satisfy Gaffney estimates (\ref{eq2.19}) then so does
$\{S_tT_t\}_{t>0}$. Moreover, there exist $c, C>0$ such that for
arbitrary closed sets $E,F\subset \RR^n$
\begin{equation}\label{eq2.20}
\|S_sT_tf\|_{L^2(F)}\leq C\,e^{-\frac{{\rm
dist}\,(E,F)^2}{c\max\{t,s\}}}\,\|f\|_{L^2(E)},
\end{equation}

\noindent for all $t,s>0$ and all $f\in L^2(\RR^n)$ supported in
$E$.
\end{lemma}

A family of operators $\{S_t\}_{t>0}$
satisfies $L^p-L^q$ off-diagonal estimates, $1<p,q<\infty$, if for arbitrary
closed sets $E,F\subset \RR^n$
\begin{equation}\label{eq2.21}
\|S_tf\|_{L^q(F)}\leq Ct^{\frac 12\left(\frac nq-\frac
np\right)}\,e^{-\frac{{\rm dist}\,(E,F)^2}{ct}}\,\|f\|_{L^p(E)},
\end{equation}

\noindent for every $t>0$ and every $f\in L^p(\RR^n)$ supported in
$E$.

\begin{lemma}(\cite{AuscherSurvey}) \label{l2.5}
For every
$p$ and $q$ such that $p_-(L)<p\leq q<p_+(L)$ the family $\{e^{-tL}\}_{t>0}$
satisfies $L^p-L^q$ off-diagonal estimates.
In particular, the operator $e^{-tL}$, $t>0$, maps $L^p(\RR^n)$
to $L^q(\RR^n)$ with norm controlled by $Ct^{\frac 12\left(\frac
nq-\frac np\right)}$. \Bk
\end{lemma}

The Lemma has been essentially proven in \cite{AuscherSurvey}, Proposition 3.2. There, $q\equiv 2$, but the argument directly extends to the full range stated in Lemma~\ref{l2.5} above (see also the Remark following Proposition~3.2 in  \cite{AuscherSurvey}).

\begin{lemma}\label{l2.6}
Assume that for some  $1\leq r\leq 2$ the family
$\{e^{-tL}\}_{t>0}$ satisfies $L^r-L^2$ off-diagonal estimates.
Then the family $\{tLe^{-tL}\}_{t>0}$ also satisfies $L^r-L^2$
off-diagonal estimates and the operators $e^{-tL}$, $tLe^{-tL}$,
$t>0$, are bounded from $L^r(\RR^n)$ to $L^2(\RR^n)$ with norms
bounded by $Ct^{\frac 12\left(\frac n2-\frac nr\right)}$, and from
$L^r(\RR^n)$ to $L^r(\RR^n)$ with norms independent of $t$. \Bk
\end{lemma}

\bp The fact that $L^r-L^2$ off-diagonal
estimates implies boundedness in $L^r(\RR^n)$ is rather standard, see e.g.,
\cite{AuscherSurvey}, Lemma~3.3, or \cite{BK2}.  

As we mentioned above Lemma~\ref{l2.4}, the family of operators $\{tLe^{-tL}\}_{t>0}$
satisfies $L^2-L^2$ off-diagonal estimates and, in particular, is bounded in
$L^2(\RR^n)$. We can combine this information with the properties
of the heat semigroup, stated in Lemma~\ref{l2.5}, and Lemma~\ref{l2.4} to deduce that  $tLe^{-tL}=2\left(\frac
t2 Le^{-\frac t2 L}\right)e^{-\frac t2L}$, \Bk $t>0$, also
satisfies $L^r-L^2$ off-diagonal estimates and is $L^r-L^2$
bounded.\ep

We say that  a family of operators $\{S_t\}_{t>0}$
satisfies $L^2$ off-diagonal estimates {\it of order $N$}, $N>0$, $N\in\RR$, if
there is a constant $C >0$ such that for arbitrary
closed sets $E,F\subset \RR^n$
\begin{equation}\label{eq2.22}
\|S_tf\|_{L^2(F)}\leq C\,\min\left\{1,\frac{t}{{\rm
dist}\,(E,F)^2}\right\}^N\,\|f\|_{L^2(E)},
\end{equation}

\noindent for every $t>0$ and every $f\in L^2(\RR^n)$ supported in
$E$.

\begin{lemma}\label{l2.7} Let $\mu\in(\omega,\pi/2)$, $\psi\in \Psi_{\sigma, \tau}(\Sigma_\mu^0)$ for some $\sigma, \tau>0$, and $f\in H^\infty(\Sigma_\mu^0)$. Then the family of operators $\{\psi(tL)f(L)\}_{t>0}$ satisfies $L^2$ off-diagonal estimates of order $\sigma$, with the constant controlled by $\|f\|_{L^\infty(\Sigma_\mu^0)}$.
\end{lemma}

An analogous fact has been established for the Hodge-Dirac operator on a complete Riemannian manifold in \cite{AMcR}, Lemma~3.6.
\vskip 0.08in

\bp Recall the representation formulas (\ref{eq2.6}), (\ref{eq2.7}). We use them for the function  $\psi(tL)f(L)$, $t>0$. First of all,
\begin{multline}\label{eq2.23}
|\eta_{\pm}(z)|\leq \frac Ct\,\int_{\gamma_{\pm}}|\psi(t\xi)|\,|f(\xi)|\,d(t\xi)\\ 
\leq \frac Ct\,\|f\|_{L^\infty(\Sigma_\mu^0)}\int_{\gamma_{\pm}}\frac{|t\xi|^\sigma}{1+|t\xi|^{\sigma+\tau}} 
\,d(t\xi)\leq \frac Ct\,\|f\|_{L^\infty(\Sigma_\mu^0)},
\end{multline}

\noindent for all $z\in \Gamma_{\pm}$, in particular, for $z$ with $|z|\leq t$. 

When $|z|>t$ we break $\eta_{\pm}(z)$ into two integrals: one over $\{\xi\in \gamma_{\pm}:\,|\xi|\leq 1/t\}$ (called $J_1$) and the second one over  $\{\xi\in \gamma_{\pm}:\,|\xi|\geq 1/t\}$ (called $J_2$). Then
\begin{eqnarray}\label{eq2.24}\nonumber
J_1&\leq & C \,\|f\|_{L^\infty(\Sigma_\mu^0)}\int_{\xi\in \gamma_{\pm}:\,|\xi|\leq 1/t}e^{-\delta |z||\xi|}|t\xi|^\sigma\,d\xi \\[4pt]& \leq & \frac{C}{|z|}\,\|f\|_{L^\infty(\Sigma_\mu^0)}\,\frac{t^\sigma}{|z|^\sigma}\,\,\int_0^{\infty}
e^{-\delta \rho}\,\rho^\sigma\,d\rho
\leq \frac {C}t\,\|f\|_{L^\infty(\Sigma_\mu^0)}\,\Bigl(\frac{t}{|z|}\Bigr)^{\sigma+1},
\end{eqnarray}

\noindent where $\delta=-\cos{\left(\frac \pi 2-\theta+\nu\right)}\in (0,1)$, and
\begin{multline}\label{eq2.25}
J_2\leq C\, \|f\|_{L^\infty(\Sigma_\mu^0)}
\int_{\xi\in \gamma_{\pm}:\,|\xi|\geq 1/t}|z\xi|^{-\sigma-1}|t\xi|^{-\tau}\,d\xi \\
\leq  C\,\|f\|_{L^\infty(\Sigma_\mu^0)}\,\Bigl(\frac{t}{|z|}\Bigr)^{\sigma+1}\,t^{-\tau-\sigma-1}\,\,\int_{\xi\in \gamma_{\pm}:\,|\xi|\geq 1/t}|\xi|^{-\sigma-1-\tau}\,d\xi
\\\leq \frac {C}t\,\|f\|_{L^\infty(\Sigma_\mu^0)}\,\Bigl(\frac{t}{|z|}\Bigr)^{\sigma+1}.
\end{multline}

\noindent Hence,
\begin{equation}\label{eq2.26}
|\eta_{\pm}(z)|\leq \frac Ct\,\|f\|_{L^\infty(\Sigma_\mu^0)}\,\min\left\{1, \Bigl(\frac{t}{|z|}\Bigr)^{\sigma+1}\right\},\quad \forall\,z\in \Gamma_{\pm}.
\end{equation}

Armed with this estimate, we proceed to the bounds on $\psi(tL)f(L)$, $t>0$.
Take some $g\in L^2(\RR^n)$ supported in a closed set $E$. Then for any closed set $F\subset \RR^n$
\begin{equation}\label{eq2.27}
\|\psi(tL)f(L)g\|_{L^2(F)}\leq \int_{\Gamma_+}\|e^{-zL}g\|_{L^2(F)}|\eta_+(z)|\,dz
+\int_{\Gamma_-}\|e^{-zL}g\|_{L^2(F)}|\eta_-(z)|\,dz.
\end{equation}

\noindent Further,
\begin{eqnarray}\label{eq2.28}\nonumber
&& \int_{\Gamma_{\pm}}\|e^{-zL}g\|_{L^2(F)}|\eta_{\pm}(z)|\,dz\leq
C \,\|g\|_{L^2(E)}\int_{\Gamma_{\pm}}e^{-\frac{{\rm dist}\,(E,F)^2}{c|z|}}\, |\eta_{\pm}(z)|\,dz\\[4pt]
&& \qquad \leq C \,\|f\|_{L^\infty(\Sigma_\mu^0)}\,\|g\|_{L^2(E)}\int_{\Gamma_{\pm}}e^{-\frac{{\rm dist}\,(E,F)^2}{c|z|}}\, \min\left\{1, \Bigl(\frac{t}{|z|}\Bigr)^{\sigma+1}\right\}\,\frac 1t\,dz.
\end{eqnarray}

\noindent Now we split the last integral in (\ref{eq2.28}) according to whether $|z|\leq t$ or $|z|\geq t$, and denote the corresponding parts of it by $I_1$ and $I_2$, respectively. Then
\begin{eqnarray}\label{eq2.29}
I_1=\int_{z\in \Gamma_{\pm}:|z|\leq t}e^{-\frac{{\rm dist}\,(E,F)^2}{c|z|}}\,\frac 1t\,dz\leq e^{-\frac{{\rm dist}\,(E,F)^2}{ct}}.
\end{eqnarray}

\noindent On the other hand,
\begin{eqnarray}\label{eq2.30}
I_2=\int_{z\in \Gamma_{\pm}:|z|\geq t}e^{-\frac{{\rm dist}\,(E,F)^2}{c|z|}}\,\Bigl(\frac{t}{|z|}\Bigr)^{\sigma+1}\,\frac 1t\,dz.
\end{eqnarray}

\noindent If $t\geq {\rm dist}\,(E,F)^2$, we obtain the bound
\begin{eqnarray}\label{eq2.31}
I_2\leq\int_{z\in \Gamma_{\pm}:|z|\geq t}\Bigl(\frac{t}{|z|}\Bigr)^{\sigma+1}\,\frac 1t\,dz\leq C.
\end{eqnarray}

\noindent If $t\leq {\rm dist}\,(E,F)^2$, then
\begin{eqnarray}\label{eq2.32}
&&I_2\leq\int_{z\in \Gamma_{\pm}:t\leq |z|\leq {\rm dist}\,(E,F)^2} \Bigl(\frac{|z|}{{\rm dist}\,(E,F)^2}\Bigr)^{N}\,\Bigl(\frac{t}{|z|}\Bigr)^{\sigma+1}\,\frac 1t\,dz\nonumber\\[4pt]
&&\qquad + \int_{z\in \Gamma_{\pm}:|z|\geq {\rm dist}\,(E,F)^2}\,\Bigl(\frac{t}{|z|}\Bigr)^{\sigma+1}\,\frac 1t\,dz,
\end{eqnarray}

\noindent for any $N>0$. Let us take $N>\sigma$. Then
\begin{eqnarray}\label{eq2.33}\nonumber
I_2&\leq & C \Bigl(\frac{1}{{\rm dist}\,(E,F)^2}\Bigr)^{N}t^{\sigma}{\rm dist}\,(E,F)^{2(N-\sigma)}+C \Bigl(\frac{t}{{\rm dist}\,(E,F)^2}\Bigr)^{\sigma}\\[4pt]
&\leq & C \Bigl(\frac{t}{{\rm dist}\,(E,F)^2}\Bigr)^{\sigma}.
\end{eqnarray}

\noindent  This finishes the proof of the Lemma.\ep

Finally, we establish the following Lemma (cf. Lemma~3.7 in \cite{AMcR}).

\begin{lemma}\label{l4.2} Let $\mu\in(\omega,\pi/2)$ and $\sigma_1,\sigma_2,\tau_1,\tau_2>0$. Suppose further that $\psi\in \Psi_{\sigma_1,\tau_1}(\Sigma_\mu^0)$, $\widetilde \psi\in \Psi_{\sigma_2,\tau_2}(\Sigma_\mu^0)$
and $f\in H^\infty(\Sigma_\mu^0)$. Then for any $0<a<\min\{\sigma_1,\tau_2\}$ and $0<b<\min\{\sigma_2,\tau_1\}$ there is a family of operators $T_{s,t}$, $s,t>0$ such that
\begin{equation}\label{eq4.9}
\psi(sL)f(L)\widetilde{\psi}(tL)=\min\left\{\Bigl(\frac st\Bigr)^a, \Bigl(\frac ts\Bigr)^b\right\}T_{s,t},
\end{equation}
\noindent where\\
\noindent $(1)$ $\{T_{s,t}\}_{s\leq t}$ satisfy the $L^2$ off-diagonal estimates in $t$ of order $\sigma_2+a$ uniformly in $s\leq t$, \\
\noindent $(2)$ $\{T_{s,t}\}_{t\leq s}$ satisfy the $L^2$ off-diagonal estimates in $s$ of order $\sigma_1+b$ uniformly in $t\leq s$,\\
\noindent with the constants bounded by $\|f\|_{L^\infty(\Sigma_\mu^0)}$.
\end{lemma}

\bp Let us consider first $s\leq t$. Then
\begin{equation}\label{eq4.10}
\psi(sL)f(L)\widetilde{\psi}(tL)=\Bigl(\frac st\Bigr)^a\,(sL)^{-a}\psi(sL)f(L)(tL)^{a}\widetilde{\psi}(tL)=:\Bigl(\frac st\Bigr)^a\,T_{s,t}.
\end{equation}

\noindent The function $(s\xi)^{-a}\psi(s\xi)f(\xi)$, $\xi \in \Sigma_\mu^0$, belongs to $H^\infty(\Sigma_\mu^0)$ and
\begin{equation}\label{eq4.11}
\|(s\xi)^{-a}\psi(s\xi)f(\xi)\|_{L^\infty(\Sigma_\mu^0)}\leq C\|f\|_{L^\infty(\Sigma_\mu^0)},
\end{equation}

\noindent with the constant $C$ independent of $s>0$.
Hence, by Lemma~\ref{l2.7} the operators $\{T_{s,t}\}_{s\leq t}$ satisfy the $L^2$ off-diagonal estimates in $t$ of order $\sigma_2+a$ uniformly in $s\leq t$, with the constant bounded by $\|f\|_{L^\infty(\Sigma_\mu^0)}$.
The case $s\geq t$ follows analogously, and their combination proves the Lemma. \ep

\section{Molecular decomposition and duality, $0<p\leq1$.}\label{s3}
\setcounter{equation}{0}

 To begin, we would like to make a few comments regarding the
well-definedness and the nature of the space
$\Lambda_{L^*}^{\alpha}(\RR^n)$, $\alpha\geq 0$. Let $M\in\NN$, $M>\frac 12\left(\alpha+\frac n2\right)$. First,
 $(I-e^{-t^2L^*})^M f$, $t\in\RR$, is globally well
defined in the sense of distributions for every 
$f\in {\bf M}^{M,\,*}_{\alpha,L}$, and belongs to $L^2_{loc}$.
 Indeed, if $\varphi \in L^2(Q)$ for some cube $Q$,
it follows from the Gaffney estimate (\ref{eq2.19}) that
$(I-e^{-t^2L})^M \varphi \in {\bf M}^{\eps,M}_{\alpha,L}$ for every
$\eps > 0$ (with the norm depending on $t,\ell (Q),
\textrm{dist}(Q,0)$). Thus,
\begin{equation}\label{eq3.1}
\langle (I-e^{-t^2L^*})^M f, \varphi \rangle \equiv  \langle  f,
(I-e^{-t^2L})^M \varphi \rangle \leq C_{t,\ell (Q),
\textrm{dist}(Q,0)} \|f\|_{({\bf
M}^{\eps,M}_{\alpha,L})^*}\|\varphi\|_{L^2(Q)}.\end{equation} Since
$Q$ was arbitrary, the claim follows.  Therefore, the norm in
(\ref{eq1.20}) is well-defined for such $f$.
Furthermore, the elements of ${\bf M}^{\eps,M}_{\alpha,L}$ are,
modulo translation, dilation and normalization, the molecules of the corresponding Hardy
spaces. The details are as follows. \Bk

For a cube $Q\subset \RR^n$ , \Bk by $S_i(Q)$, $i=0,1,2,...$,
we denote the dyadic annuli based on $Q$, i.e.
\begin{equation}\label{eq3.2}
S_0(Q):=Q \quad \mbox{ and }\quad S_i(Q):=2^iQ\setminus 2^{i-1}Q
\mbox{ for } i=1,2,...,
\end{equation}

\noindent where $2^iQ$ is the cube with the same center as $Q$ and
sidelength $2^il(Q)$. Let $0<p\leq 1$, $\eps>0$,  and $M\in \NN$.
We will always assume the above restrictions on $\eps$ and $M$, and typically,
given $p$, unless otherwise stated we will take
$M>\frac{n}{2}\Bigl(\frac 1p-\frac 12\Bigr)$. 
A function $m\in L^2(\RR^n)$ is called an
$(\hpl,\eps,M)$ - {\it molecule}\footnote{Molecules have been 
introduced in the classical setting corresponding to $L = -\Delta$ in \cite{TW};  see also \cite{CW}.} 
if it belongs to the range of $L^k$ in $L^2(\RR^n)$, for each
$k=1,...,M$, and there exists a cube
$Q\subset\RR^n$ such that
\begin{equation}
 \|(l(Q)^{-2}L^{-1})^k m\|_{L^2(S_i(Q))}\leq
\,(2^il(Q))^{\frac n2-\frac np}\,2^{-i\eps}, \,\, i=0,1,2,...,
\,\, k=0,1,...,M.\label{eq3.3}
\end{equation}
Observe that for $k=0$ the estimate (\ref{eq3.3}) is
the usual size control condition and for $k=1,..., M$ the condition (\ref{eq3.3}) is a quantitative version
of the requirement that $m\in R(L^k),$ which in turn is analogous to the classical requirement of vanishing moments.

We are now able to define a molecular $H^p_L$ space, which we shall eventually show is equivalent to
the space $H^p_L$ defined via square functions.

\begin{definition}\label{def 2.4} Let $0<p\leq 1$, and fix $\eps > 0$.
The Hardy space $\HPML$ is defined as follows.
We say that
$f= \sum\lambda_j m_j$, where $ \{\lambda_j\}_{j=0}^{\infty}\in {\ell}^p$, is a molecular 
$(H^p_L,2,\eps,M)$-representation (of $f$ ) if 
each $m_j$ is an $(H^p_L,\varepsilon,M)$-molecule, and the sum converges in $L^2(\mathbb{R}^n).$
Set
\begin{equation*}
\HPHML= \\\Big\{f:  \,f
\mbox{ has a molecular $(H^p_L,2,\eps,M)$-representation} \Big\},  
\end{equation*}
with the ``norm" (it is a true norm only when $p=1$), given by
\begin{multline*}
||f||_{\HPHML}=\\{\rm inf}\Big\{\left(\sum_{j=0}^{\infty}|\lambda_j|^p\right)^{1/p}:  
f=\sum\limits_{j=0}^{\infty}\lambda_jm_j\,
\mbox{ is a molecular $(H^p_L,2,\eps,M)$-representation} \Big\}. 
\end{multline*}
The space $\HPML$ is then defined as the completion of $\HPHML$ with respect to the metric 
induced by $||f||_{\HPHML}^p$.
\end{definition}
We note that this approach to the definition of adapted $H^p$ spaces has also been used 
in \cite{HLMMY}, at least in the case $p=1$.  We also remark that this approach, in the case
$p=1$, was implicit in \cite{HM}, but with a more complicated formulation in which
$L^2$ convergence of the molecular sums was achieved constructively, by means of an explicit truncation in scale.

Eventually, we shall see that any fixed choice of
$M> \frac{n}{2}(\frac{1}{p}-\frac{1}{2})$ and $\eps > 0$,  yields the same space. 
Indeed, more generally, we will show that the ``square function" and ``molecular"
$H^p$ spaces are equivalent, if the parameter $M >  \frac{n}{2}(\frac{1}{p}-\frac{1}{2}).$   
In fact, we shall prove
\begin{theorem}\label{th4.1} Let $0<p\leq1$.  Suppose that $M >\frac{n}{2}(\frac{1}{p}-\frac{1}{2})$ and that $\eps>0.$  
Then $\HPML= \HPL$.  Moreover,
$$\|f\|_{\HPML} \approx \|f\|_{\HPL},$$
where the implicit constants depend only on $M$, $n$, $p$, $\eps$ and ellipticity. 
 \end{theorem}
Consequently, one may write simply $H^p_{L,mol}(\mathbb{R}^n)$ in place of $\HPML$, when
$M > \frac{n}{2}(\frac{1}{p}-\frac{1}{2}),$ and for any fixed $\eps >0$,  as these spaces are all equivalent.  
Moreover, we could also define
$(\hpl,q,\eps,M)$-molecules as $m\in L^q(\RR^n)$ belonging to the range of $L^k$ in $L^q(\RR^n)$, $k=1,...,M$, and satisfying the estimates
\begin{multline}
 \|(l(Q)^{-2}L^{-1})^k m\|_{L^q(S_i(Q))}\leq
C\,(2^il(Q))^{\frac nq-\frac np}\,2^{-i\eps}, \\[4pt] i=0,1,2,...,
\quad k=0,1,...,M,\label{eq3.21}
\end{multline}

\noindent  These would also yield the same $\HPL$ spaces 
provided $p_-(L)< q<p_+(L)$. 
We omit the details here, although we do note that
a proof is given in \cite{HM}, \cite{HM2} in the case $p=1$.

\medskip

We now proceed to the proof of Theorem \ref{th4.1}.  The basic strategy is as follows:
by density, it is enough to show that 
\begin{equation}\label{hpequivalence}\HPHML = L^2(\mathbb{R}^n) \cap \HPL, \qquad 
M>\frac{n}{2}\left(\frac{1}{p}-\frac{1}{2}\right)
\end{equation} 
with equivalence of norms.  The proof of this fact proceeds in two steps.  

\medskip 

\noindent {\bf Step 1}:  $\HPHML \subseteq L^2(\mathbb{R}^n) \cap \HPL$, if
$M>\frac{n}{2}(1/p-1/2).$

\medskip

\noindent {\bf Step 2}:  $\HPL \cap  L^2(\mathbb{R}^n)  \subseteq   \HPHML$, for every
$M\in \NN$.

\medskip

We take these in order.
The conclusion of Step 1 is an immediate consequence of the following pair of Lemmata.

\begin{lemma}\label{l4.1} Fix $M\in \NN$, and suppose that $0<p\leq1$. 
Assume that $T$ is a linear operator, or a non-negative  
sublinear operator, satisfying the weak-type (2,2) bound 
\begin{equation}\label{e4.weak}
\mu\{x\in \mathbb{R}^n: 
|Tf(x)|>\eta \}\leq \,C_T\, \,\eta^{-2} \|f\|_{L^2(\mathbb{R}^n)}^2,\quad \forall \eta>0,
\end{equation}
and that for every 
$(H^p_L,\eps, M)$-molecule $m$, we have 
\begin{eqnarray}\label{e4.1a}
\|Tm\|_{L^p(\mathbb{R}^n)}\leq C 
\end{eqnarray}
\noindent with constant $C$ independent of $m$.
Then $T$ is bounded from $\HPHML$ to $L^p(\mathbb{R}^n),$ and
$$
\|Tf\|_{L^p(\mathbb{R}^n)}\leq C\|f\|_{\HPHML}.
$$
Consequently, by density, $T$ extends to a bounded operator from $\HPML$ to $L^p(\mathbb{R}^n).$ 
\end{lemma}

We mention that a result similar to Lemma \ref{l4.1} appears in \cite{JY} (Lemma 5.1).

\begin{lemma} \label{S-mol} Let $m$ be an $(H^p_L,\eps,M)$-molecule, with $0<p\leq1$, $M> \frac{n}{2}(\frac{1}{p}-\frac{1}{2})$ and $\eps > 0$.  Then there is a constant $C_0$ depending only on 
$p,\eps,M,n$ and ellipticity such that
$$  \|Sm\|_p \leq C_0,$$
where $S$ denotes the square function defined in (\ref{eq1.8}).
\end{lemma}

Indeed, given Lemma \ref{S-mol}, we may apply Lemma \ref{l4.1} with $T=S$
to obtain
$$\|f\|_{\HPL} := \|Sf\|_{L^p(\mathbb{R}^n)} \leq C \|f\|_{\HPHML},$$
whence Step 1 follows.  

To finish Step 1, it therefore suffices to prove the two Lemmata.

\begin{proof}[Proof of Lemma \ref{l4.1}]
Let
$f \in \HPHML$, where $f = \sum \lambda_j m_j$ is a molecular $(H^p_L,2,\eps,M)$-representation such that
$$\|f\|^p_{\HPHML} \approx  \sum_{j=0}^\infty |\lambda_j|^p .$$  Since the sum converges in $L^2$
(by definition), and since $T$ is of weak-type $(2,2)$, we have
that at almost every point,
\begin{equation}\label{eq4.4a}|T(f)| \leq \sum_{j=0}^\infty |\lambda_j| \,|T(m_j)|.\end{equation}
Indeed, for every $\eta >0$, we have that, if $f^N:= \sum_{j>N} \lambda_j m_j$, then,
 \begin{multline*}
\mu\Big\{|T(f)| - \sum_{j=0}^\infty |\lambda_j| \,|T(m_j)| >\eta\Big\}\,\leq \limsup_{N\to \infty}
\mu\left\{|T(f^N)|>\eta\right\}\\ \leq \,C_T\,\,\eta^{-2} \,\limsup_{N\to \infty} \|f^N\|_2 =0,\end{multline*}
from which (\ref{eq4.4a}) follows.
In turn, (\ref{eq4.4a}) and (\ref{e4.1a}) imply the desired $L^p$ bound for $Tf$, since $0<p\leq1$.
\end{proof}

%We recall the definition of the square function based on the heat semigroup:
%$$Sf(x):=\left(\iint_{\Gamma(x)} |t^2Le^{-t^2L}f(y)|^2\,\frac{dydt}{t^{n+1}}\right)^{1/2},$$
%where as above $\Gamma(x)$ denotes the cone with vertex $x$ and aperture 1.

\begin{proof}[Proof of Lemma \ref{S-mol}]  Fix a cube $Q$, and
let $m$ be an $(H^p_L,\eps,M)$-molecule, adapted to $Q$, with $0<p\leq1$, $M> \frac{n}{2}(\frac{1}{p}-\frac{1}{2})$ and $\eps > 0$. 
In particular, we have that for each $k \in \{0,1,...,M\}$,
\begin{equation}\label{e5.l2bound}
\|\left(\ell(Q)^2L\right)^{-k}m\|_{L^2(\mathbb{R}^n)}  \leq C_k \,|Q|^{1/2-1/p}.\end{equation}
Hence, by H\"{o}lder's inequality and the  $L^2$ boundedness of $S$, we have that
$$\|Sm\|_{L^p(16Q)} \leq C |Q|^{1/p-1/2} \,\|S m\|_{L^2(\mathbb{R}^n)} \leq C.$$
Writing now $\|Sm\|^p_p = \|Sm\|^p_{L^p(16Q)} +\sum_{j=5}^\infty \|Sm\|^p_{L^p(S_j(Q))}$,
where we recall that $S_j(Q) := 2^jQ\setminus 2^{j-1}Q,$ we see that it is enough to prove that
\begin{equation}\label{e5.2}
\|Sm\|_{L^2(S_j(Q))}\leq C2^{-j\alpha}|2^jQ|^{1/2-1/p},
\end{equation}
for some $\alpha >0$ and for each $j\geq 5$.  To this end, we write
\begin{multline*}
\|Sm\|^2_{L^2(S_j(Q))}
\\= \int_{S_j(Q)}
\int_0^\infty\!\!\int_{|x-y|<t}\,\Bigl|\left(t^2Le^{-t^2L} \,m\right)(y)\,\Bigr|^2\,
\frac{dydt}{t^{n+1}}\, dx\\=\int_{S_j(Q)}
\int_0^{2^{\theta(j-5)}\ell(Q)}\int_{|x-y|<t}\,\,+\,\,\int_{S_j(Q)}
\int_{2^{\theta(j-5)}\ell(Q)}^\infty\int_{|x-y|<t}=: I + II,
\end{multline*}
where $\theta \in (0,1)$ will be chosen momentarily.
Then by Fubini's theorem,  the definition of an $(H^p_L,\eps,M)$-molecule (cf. (\ref{eq3.3})),
the uniform $L^2$ boundedness of $t^{2K}L^Ke^{-t^2L}$ for each non-negative integer $K$, 
and (\ref{e5.l2bound}), setting  $b:= L^{-M}m$, we have
\begin{multline*}II \leq 
\int_{2^{\theta(j-5)}\ell(Q)}^\infty
\int_{\mathbb{R}^n}\Bigl|\left(t^{2(M+1)}L^{M+1}e^{-t^2L} \,b\right)(y)\,\Bigr|^2\, dy\,
\frac{dt}{t^{4M+1}}\\\leq C\left( 2^{\theta j}\ell(Q)\right)^{-4M} 
\|b\|^2_{L^2(\mathbb{R}^n)}\leq C2^{-j(4\theta M+n(1-2/p))}
2^{jn(1-2/p)}\,|Q|^{1-2/p}\\= \,C 2^{-j(4\theta M-n(2/p-1))}\,|2^jQ|^{1-2/p}\,.
\end{multline*}
Taking square roots, and choosing $\theta$ sufficiently close to $1$, we obtain (\ref{e5.2}) for the contribution of the term $II$, with 
$\alpha = (2\theta M-n(1/p-1/2)) >0.$

We now treat the term $I$.   We set 
$$\widetilde{S}_j (Q) := 2^{j+1}Q \setminus 2^{j-2}Q,\quad
\widehat{S}_j (Q) := 2^{j+2}Q \setminus 2^{j-3}Q,$$
and observe that, by Fubini's Theorem
\begin{multline*} I \leq\,\int_0^{2^{\theta(j-5)}\ell(Q)}
\int_{\widetilde{S}_j(Q)}\Bigl|\left(t^2Le^{-t^2L} \,m\right)(y)\,\Bigr|^2\,
dy\,\frac{dt}{t} \\\lesssim\,\,\int_0^{2^{\theta(j-5)}\ell(Q)}
\int_{\widetilde{S}_j(Q)}\Bigl|\left(t^2Le^{-t^2L} \,(1_{2^{j-3}Q}m)\right)(y)\,\Bigr|^2\,
dy\,\frac{dt}{t}\\+\,\,\int_0^{2^{\theta(j-5)}\ell(Q)}
\int_{\widetilde{S}_j(Q)}\Bigl|\left(t^2Le^{-t^2L} \,(1_{\widehat{S}_j(Q)}m)\right)(y)\,\Bigr|^2\,
dy\,\frac{dt}{t}\\+\,\,\int_0^{2^{\theta(j-5)}\ell(Q)}
\int_{\widetilde{S}_j(Q)}\Bigl|\left(t^2Le^{-t^2L} 
\,(1_{\mathbb{R}^n\setminus2^{j+2}Q}\,m)\right)(y)\,\Bigr|^2\,
dy\,\frac{dt}{t}\\=:I_1 + I_2+ I_3.
\end{multline*}
By the $L^2$ boundedness of $S$ and the definition of a molecule (cf. (\ref{eq3.3})),
$$\sqrt{I_2} \leq C\, \|m\|_{\widehat{S}_j(Q)}\leq \,C\, 2^{-j\eps} |2^j Q|^{1/2-1/p},$$
which is (\ref{e5.2}) for the contribution of $I_2$. 
For the other two terms, we have that by the Gaffney estimates (cf. subsection \ref{s2.3}),
\begin{multline*}I_1 + I_3 \leq C \|m\|_{L^2(\mathbb{R}^n)}^2 
\,\int_0^{2^{\theta(j-5)}\ell(Q)}\exp\left(\frac{-(2^j \ell(Q))^2}{c\,t^2}\right)
\frac{dt}{t}\\ \leq C_N \|m\|_{L^2(\mathbb{R}^n)}^2 
\,\int_0^{2^{\theta(j-5)}\ell(Q)}\left(\frac{t}{2^j\ell(Q)}\right)^N
\frac{dt}{t}\,\, \leq \,\, C_N |Q|^{2(1/2 - 1/p)} 2^{N(\theta-1)j},
\end{multline*}
where we have used (\ref{e5.l2bound}) in the last step, and $N$ is at our disposal.
Having fixed $\theta < 1$ above, we may now choose $N$ so large that $N(1-\theta)
\geq 4M>2n(1/p-1/2)$,  to obtain in turn the desired bound
$$I_1 + I_3
\leq \,C\, |2^jQ|^{2(1/2-1/p)} 2^{-j(4M-2n(1/p-1/2))},$$
whence (\ref{e5.2}) follows.

\end{proof}

This concludes Step 1.  We now turn to Step 2.  

Our goal is to show that every $f \in L^2(\mathbb{R}^n) \cap \HPL$ has a
molecular $(H^p_L,2,\eps,M)$-representation, with appropriate quantitative control
of the coefficients.  To this end, we follow the (nowadays) standard tent space approach of \cite{CMS},
as adapted to the present setting  
in the case $p=1$ in
\cite{AMcR} (cf.  \cite{HLMMY} and \cite{JY}, as well as the earlier work \cite{DL});  yet another 
(somewhat more complicated) adaptation of the methods of \cite{CMS}
was used in \cite{HM}, \cite{HM2}.

Let us begin by recalling some basic facts from \cite{CMS}.  First,  for $0<p<\infty$, the tent spaces on
$\RR^{n+1}_+=\RR^n\times(0,\infty)$ are defined by
\begin{equation}\label{eq4.1}
T^p(\RR^{n+1}_+):=\{F:\RR^{n+1}_+\longrightarrow
\CC;\,\|F\|_{T^p(\RR^{n+1}_+)}:=\|{\Sq}F\|_{L^p(\RR^n)}<\infty\},
\end{equation}

\noindent where
\begin{equation}\label{eq4.2}
{\Sq} F(x)= \left(\dint_{\Gamma(x)}
|F(y,t)|^2\frac{dydt}{t^{n+1}}\right)^{1/2},\qquad x\in\RR^n.
\end{equation}
In addition, the case $p=\infty$ may be handled as follows.
For $F:\RR^{n+1}_+\to \CC$ let
\begin{equation}\label{eq6.9}
{\mathcal C}F(x):=\sup_{B\ni x}\left(\frac{1}{|B|}\dint_{\widehat
B}|F(y,t)|^2\,\frac{dydt}{t}\right)^{1/2}, \qquad x\in\RR^n,
\end{equation}

\noindent where $B$ stands for a ball in $\RR^n$ and
\begin{equation}\label{eq6.10}
\widehat B:=\{(x,t)\in\RR^n\times(0, \infty):\,{\rm
dist}(x,\,^cB)\geq t\}.
\end{equation}
For $p=\infty$, we then have
\begin{equation}\label{eq4.Tinfinity}
T^\infty(\RR^{n+1}_+):=\{F:\RR^{n+1}_+\longrightarrow
\CC;\,\|F\|_{T^\infty(\RR^{n+1}_+)}:=\|{\mathcal C}F\|_{L^\infty(\RR^n)}<\infty\}.
\end{equation}
Moreover, according to \cite{CMS},
%\begin{equation}\label{eq6.11}
%\|{\Sq}F\|_{L^p(\RR^n)}\leq C \|{\mathcal C} F\|_{L^p(\RR^n)}\,,\qquad\quad 0<p<\infty,
%\end{equation} and 
%\begin{equation}\label{eq6.12}
%\|{\mathcal C}F\|_{L^p(\RR^n)}\leq C \|{\Sq} F\|_{L^p(\RR^n)}\,,\qquad\quad 2<p\leq\infty,
%\end{equation}
%so that, in particular,
\begin{equation}\label{eq4.tentp>2}
\|{\mathcal C}F\|_{L^p(\RR^n)}\approx  \|{\Sq} F\|_{L^p(\RR^n)} = \|F\|_{T^p(\RR^{n+1}_+)}\,,
\quad\quad 2<p<\infty.
\end{equation}

\noindent The tent spaces satisfy the natural duality and interpolation properties:
\begin{equation}\label{eq4.3}
\left(T^q(\RR^{n+1}_+)\right)^*=T^{q'}(\RR^{n+1}_+),\quad 1/q+{1}/{q'}=1,\quad 1<q<\infty,
\end{equation}
and also $\left(T^1(\RR^{n+1}_+)\right)^*=T^{\infty}(\RR^{n+1}_+)$;  moreover,
\begin{equation} \label{eq4.4}
[T^{p_0}(\RR^{n+1}_+),T^{p_1}(\RR^{n+1}_+)]_{\theta}=T^p(\RR^{n+1}_+),\quad
1/p=(1-\theta)/p_0+\theta/p_1,\quad 0<\theta<1,
\end{equation}

\noindent for $0<p_0<p_1\leq+\infty$.  We will later discuss the precise meaning of the complex interpolation in (\ref{eq4.4}) and provide references (see the proof of Lemma~\ref{l4.5} and the preceding discussion).  

It has been proved in \cite{CMS} 
that every $F\in T^p(\RR^{n+1}_+),\, 0<p\leq1$ has an atomic 
decomposition.  For future reference, we record this result below.
We first define the notion of a $T^p(\RR^{n+1})$-atom.

\begin{definition}\label{def4.6} Let $0<p\leq 1$.  A measurable function $A$ on 
$\RR^{n+1}_+$ is said to be a $T^p$-atom if there exists a cube 
$Q\subset \RR^n$ such that $A$ is supported in the ``Carleson box"
$$R_Q:=Q\times\Big(0,\ell(Q)\Big)\,,$$ and

\begin{eqnarray}\label{e4.8}
\left(\iint_{R_Q}\,|A(x,t)|^2 {dxdt\over t}\right)^{1/2}\leq |Q|^{\frac{1}{2}-\frac{1}{p}}.
\end{eqnarray}
\end{definition}

\begin{proposition}\label{prop4.7}\cite{CMS} Let $0<p\leq1.$  For every element $F\in T^p(\RR^{n+1}_+),$ there exist a  numerical sequence $\{\lambda_j\}_{j=0}^{\infty}\subset \ell^p$ and a sequence of 
$T^p$-atoms $\{A_j\}_{j=0}^{\infty}$ such that
\begin{equation}\label{eqtentdecomp}
F=\sum_{j=0}^{\infty}\lambda_jA_j\quad\mbox{in }\,T^p(\RR^{n+1}_+) \,\text{ and a.e. in } \RR^{n+1}_+. 
\end{equation}
Moreover,  $$ 
\sum_{j=0}^{\infty}|\lambda_j|^p\approx \|F\|^p_{T^p(\RR^{n+1}_+)},$$
where the implicit constants depend only on dimension.   

Finally, if
$F \in T^p(\RR^{n+1}_+)\cap T^2(\RR^{n+1}_+)$, then the decomposition (\ref{eqtentdecomp}) 
also converges in $T^2(\RR^{n+1}_+)$.
\end{proposition} 
\begin{proof}
Except for the final part of the proposition, concerning $T^2$ convergence,
this is proved in \cite{CMS}, and we refer the reader to that paper for the proof.
The $T^2$ convergence is only implicit there, so we shall sketch the proof here.
To this end, we first note that 
\begin{multline}\label{e4.tentequiv}
\|F\|_{T^2(\RR^{n+1}_+)}^2 := \int_{\RR^n} (\mathcal{A}F)^2 dx
=\int_{\RR^n}\int_0^\infty\int_{|x-y|<t} |F(y,t)|^2 \frac{dy dt}{t^{n+1}} d(x)\\\approx
\int_0^\infty\int_{\RR^n}|F(y,t)|^2 {dydt\over t}
\end{multline}
Suppose now that $F\in T^p \cap T^2$.  We recall that, in the constructive proof of 
the decomposition (\ref{eqtentdecomp}) in \cite{CMS},
one has that
$$\lambda_j A_j = F \,1_{S_j},$$
where $\{S_j\}$ is a collection of sets which are pairwise disjoint (up to sets of measure zero),
and whose union covers $\RR^{n+1}_+$.  
Thus, by (\ref{e4.tentequiv}), \begin{multline*}\|\sum_{j>N}\lambda_jA_j\|_{T^2(\RR^{n+1}_+)}^2 \approx
\int_0^\infty\int_{\RR^n}\Big|\sum_{j>N} 1_{S_j}\,F(y,t)\Big|^2 {dydt\over t}\\
=\sum_{j>N} \iint_{S_j}|F|^2{dydt\over t} \to 0,
\end{multline*}
as $N\to \infty$, where we have used disjointness of the sets $S_j$ and dominated convergence.
It therefore follows that  $F=\sum \lambda_j A_j$ in $T^2$. 
\end{proof}

Now, given  $M\geq 1$, we define an operator 
$\pi_{M,L}$, acting initially on $T^2$, as follows:
\begin{equation}\label{eqpidef}
\pi_{M,L} (F) := \int_0^\infty \left(t^2L\right)^{M+1}e^{-t^2L} F(\cdot,t)\, \frac{dt}{t}.
\end{equation}
By a standard duality argument involving well known quadratic estimates for $L^*$, one obtains that
the improper integral converges weakly in $L^2$, and that
 \begin{equation}\label{eL2bound}
\|\pi_{M,L}(F)\|_{L^2(\RR^n)} \lesssim \|F\|_{T^2(\RR^{n+1}_+)},\,\,\,\,M\geq 0, 
\end{equation}
where the implicit bound depends only on $M$, ellipticity and dimension.

Following \cite{CMS}, we now observe that $\pi_{M,L}$ essentially maps $T^p$ atoms into 
$H^p_L$-molecules.
We have:
\begin{lemma}\label{l4.atoms}Suppose that $A$ is a $T^p(\RR^{n+1}_+)$-atom associated to a cube
$Q\subset \mathbb{R}^n$ (or more precisely, to its Carleson box $R_Q$).  Then  for each integer
$M\geq 1$, and every $\eps >0$, there is a uniform constant $C_{\eps,M}$ such that $C_{\eps,M}^{-1}\, \pi_{M,L} (A)$ is an 
$(H^p_L,\eps,M)$-molecule associated to $Q$.
\end{lemma}
\begin{proof}
Fix a cube $Q$ and let $A$ be a $T^p(\RR^{n+1}_+)$-atom associated to $R_Q$,
so that (\ref{e4.8}) holds.  
We set 
$$
m:=\pi_{M,L}(A)=L^M b,
$$
where
$$
b:=\int_0^{\infty}t^{2M}t^2L\,e^{-t^2L}\big(A(\cdot,\, t)\big){dt\over t},
$$
and we need to establish that $m$ satisfies (\ref{eq3.3}).
We first prove an $L^2$ estimate which in particular yields the desired bound ``near" $Q$.
Let $g\in L^2(\RR^n)$. 
Then for every $k=0,1,\dots,M $ we have
\begin{multline}\label{e9.9}
\Big|\int_{\RR^n}(\ell(Q)^2L)^{k}b(x)\,g(x)dx\Big| 
\\ =\Big|\lim_{\delta\to 0}\int_{\RR^n}\left(\int_\delta^{1/\delta}\ell(Q)^{2k}  
L^{k}t^{2M}t^2L\,e^{-t^2L}\big(A(\cdot,\, t)\big)(x) 
{dt\over t}\right)g(x)\,dx \Big| \\[4pt]
=   \Big|\int_{R_Q} A(x,t) \,\ell(Q)^{2k}  
(L^*)^{k}t^{2M}t^2L^*\,e^{-t^2L^*}g(x)
{dxdt\over t} \Big| \\[4pt]
 \leq  \ell(Q)^{2M} |Q|^{1/2-1/p}\Big(
\int_{ R_Q }\big|  (t^2L^*)^{k+1}e^{-t^2L^*}g(x)
\big|^2{dxdt\over t}\Big)^{1/2}.  
\end{multline}
Here, the third line is obtained by using the compactness of the $t$ interval to interchange the order of integration, and the fourth line by using that  
$A$ is a $T^p$-atom supported in $R_Q$ (so that $0<t<\ell(Q)$ and (\ref{e4.8}) holds) and the fact
that $k\leq M.$  
In turn, by standard square function estimates for $L^*$, (\ref{e9.9})
is bounded by 
$$ C\ell(Q)^{2M}|Q|^{1/2-1/p} \|g\|_{L^2(\RR^n)}.$$  Specializing to the case that
$g$ is supported in $2Q$, and taking a supremum over all such $g$ with $\|g\|_{L^2(2Q)}=1,$  
we then have the bound
\begin{equation*}
\|(\ell(Q)^2L)^{k}b\|_{L^2(2Q)} \,
\leq\, C\ell(Q)^{2M}|Q|^{1/2-1/p},\quad k=0,1,...,M,
\end{equation*}
which is clearly equivalent to the cases $i=0,1$ of (\ref{eq3.3}).

Now for $i \geq 2$, let $g$ be supported in $S_i(Q)$, with
$\|g\|_{L^2(S_i(Q)}=1.$  Applying the Gaffney estimate to $dx$ integral in
the last line in (\ref{e9.9}),
and taking a supremum over all such $g$, we find that
\begin{multline*}
\|(\ell(Q)^2L)^{k}b\|_{L^2(S_i(Q))} 
\leq C\ell(Q)^{2M}|Q|^{1/2-1/p}\,\int_0^{\ell(Q)}e^{-(2^i\ell(Q)/t)^2} \frac{dt}{t}\\[4pt]
\leq C_N 2^{-iN} \ell(Q)^{2M}|Q|^{1/2-1/p}, 
\end{multline*} 
for every $N \in \NN$ and each $k=0,1,...,M.$  The molecular bound (\ref{eq3.3}) follows, for every choice of $\eps>0.$
\end{proof}

We are now ready to establish the molecular decomposition of $\HPL \cap L^2(\RR^n).$
Our proof here is based on the approach in \cite{AMcR}\footnote{In particular, 
it is the idea of \cite{AMcR}, in the case $p=1$, to exploit  the fact that a 
$T^p$-atomic decomposition, of an element in $T^p \cap T^2$,
converges also in $T^2$.}.   A similar approach, also following \cite{AMcR}, is taken 
in \cite{HLMMY} and in \cite{JY}.  As mentioned above, a 
more complicated method 
was used in \cite{HM, HM2}.

\begin{proposition}\label{prop4.5} Let $0<p\leq 1$ and $M\geq1$.  If $f\in\HPL\cap L^2(\RR^n)$,
then there exist a family of $(H^p_L,\eps,M)$-molecules  
$\{m_j\}_{j=0}^{\infty}$ and a sequence of numbers 
$\{\lambda_j\}_{j=0}^{\infty}\subset \ell^p$ such that $f$ can be represented in the form 
$f=\sum_{j=0}^{\infty}\lambda_j m_j$, with the sum converging in $L^2(\RR^n)$, and 
$$
\|f\|^p_{\HPHML}\leq C\sum_{j=0}^{\infty}|\lambda_j|^p\leq C\|f\|^p_{\HPL},
$$
where $C$ is independent of $f$. In particular,

\begin{eqnarray}\label{eq1000}
\HPL\cap L^2(\RR^n)\subseteq \HPHML.
\end{eqnarray}
\end{proposition}

\begin{proof} Let $f\in \HPL\cap L^2(\RR^n)$, and set
$$F(\cdot,t):= t^2Le^{-t^2L}f.$$
We note that $F\in T^2(\RR^{n+1}_+)\cap T^p(\RR^{n+1}_+)$, by standard quadratic estimates for $L$ and the 
definition of $\HPL$.
Therefore, by Proposition \ref{prop4.7}, we have that
\begin{equation}\label{eqatomicrep} F =\sum \lambda_j \, A_j,\end{equation}
where each $A_j$ is a $T^p$-atom, the sum converges in both $T^p(\RR^{n+1}_+)$ 
and $T^2(\RR^{n+1}_+),$
and
\begin{equation}\label{e4.20}\sum|\lambda_j|^p\leq C\|F\|^p_{T^p(\RR^{n+1}_+)}=C\|f\|^p_{\HPL}.
\end{equation} 
Also, by $L^2$-functional 
calculus (\cite{Mc}), we have the ``Calder\'{o}n reproducing formula"
\begin{equation}\label{e4.9}
f=c_{M}\,\pi_{M,L}\,(t^{2}L e^{-t^2{L}}f) 
=c_M \,\pi_{M,L}(F) = c_M \,\sum \lambda_j\, \pi_{M,L}(A_j),
\end{equation}
where by (\ref{eL2bound}) and the $T^2$ convergence of the decomposition in (\ref{eqatomicrep}),
the last sum converges in $L^2(\RR^n).$
Moreover, by Lemma \ref{l4.atoms}, for every $M\geq 1$, we have that up to multiplication by some 
harmless constant,
each $m_j := c_M\, \pi_{M,L}(A_j)$
is  an $(H^p_L,\eps,M)$-molecule.
Consequently, the last sum in (\ref{e4.9}) is a molecular $(H^p_L,2,\eps,M)$-representation, so 
that $f\in \HPHML$, and by (\ref{e4.20}) we have
$$\|f\|_{\HPHML}\leq C\|f\|_{\HPL}.$$
\end{proof}
Step 2 is now complete.  This concludes the proof of Theorem \ref{th4.1}.
\ep

\medskip

We next discuss duality for the spaces $\HPL$ with $0<p\leq1$.

 If $m$ is an $(\hpl,\eps,M)$  - molecule, then $m \in {\bf
M}^{\eps,M}_{n(1/p-1),L}$ (this follows from the fact that, given
any two cubes $Q_1$ and $Q_2$, there exists integers $K_1$ and
$K_2$, depending upon $\ell(Q_1), \ell(Q_2)$ and dist$(Q_1,Q_2)$,
such that $2^{K_1}Q_1 \supseteq Q_2$ and $2^{K_2}Q_2 \supseteq
Q_1$), and the converse is also true (up to a normalization). Therefore, $g(m):=\langle g,m\rangle$ is
well-defined for every $(\hpl,\eps,M)$  - {\it molecule} $m$ and
every $g\in \Lambda_{L^*}^{n(1/p-1)}(\RR^n)$. Moreover, the following estimate
holds.

\begin{lemma}\label{l3.1} Suppose $0<p\leq 1$, $\eps>0$,
$M>\frac{n}{2}\Bigl(\frac 1p-\frac 12\Bigr)$. Then
\begin{equation}\label{eq3.4}
\left|g(m)\right|\leq C \|g\|_{\Lambda_{L^*}^{n(1/p-1)}(\RR^n)}
\end{equation}

\noindent for every $g\in \Lambda_{L^*}^{n(1/p-1)}(\RR^n)$ 
(in the case $p=1$ we set $\Lambda_{L^*}^0 := BMO_{L^*}$) and
every $(\hpl,\eps,M)$ -molecule $m$.
\end{lemma}

\bp The case $p=1$ was proved in \cite{HM}, so we now suppose that $p<1$.
For every $x\in\RR^n$
\begin{equation}\label{eq3.5}
m(x)= 2^{M}\left(l(Q)^{-2}\int_{l(Q)}^{\sqrt 2
l(Q)}s\,ds\right)^Mm(x),
\end{equation}

\noindent and
\begin{equation}\label{eq3.6}
\int_{l(Q)}^{\sqrt 2 l(Q)}s\,ds= \int_{l(Q)}^{\sqrt 2
l(Q)}s(I-e^{-s^2L})^M\,ds + \sum_{k=1}^M C_{k,M}
\int_{l(Q)}^{\sqrt 2 l(Q)}s e^{-ks^2L}\,ds.
\end{equation}

\noindent where $C_{k,M}\in\RR$ are some constants depending on
$k$ and $M$ only. Going further,
\begin{eqnarray}
&& 2kL\int_{l(Q)}^{\sqrt 2 l(Q)}s e^{-ks^2L}\,ds =
-\int_{l(Q)}^{\sqrt 2 l(Q)}\partial_s e^{-ks^2L}\,ds =
e^{-kl(Q)^2L}-e^{-2kl(Q)^2L}\nonumber\\[4pt]
&&\qquad =e^{-kl(Q)^2L}(I-e^{-kl(Q)^2L}) =
e^{-kl(Q)^2L}(I-e^{-l(Q)^2L})\sum_{j=0}^{k-1}
e^{-jl(Q)^2L}.\label{eq3.7}
\end{eqnarray}

\noindent Applying the procedure outlined in
(\ref{eq3.6})--(\ref{eq3.7}) $M$ times, we arrive at the
following formula
\begin{eqnarray}
&&\hskip -0.7cm m  = 2^{M}\Bigg(l(Q)^{-2}\int_{l(Q)}^{\sqrt 2
l(Q)}s(I-e^{-s^2L})^M\,ds \nonumber\\[4pt]
 &&\hskip -0.7cm\qquad + \sum_{k=1}^M C_{k,M} l(Q)^{-2}L^{-1}
e^{-kl(Q)^2L}(I-e^{-l(Q)^2L})\sum_{j=0}^{k-1}
e^{-jl(Q)^2L} \Bigg)^Mm.\label{eq3.8}
\end{eqnarray}

\noindent  Let
\begin{equation}\label{eq3.9}
m_{N_i}:=l(Q)^{-2N_i}L^{-N_i} m,\qquad 0\leq N_i\leq M.
\end{equation}

\noindent Then
\begin{multline}\label{eq3.10}
g(m) \Bk = C_{1,1}
\int_{\RR^n}(I-e^{-l(Q)^2L^*})^M
g(x)\,T_{1,1}^{l(Q)}m_{M}(x)\,dx\\[4pt] +\,\,
\sum_{i=1}^{(M+1)^M-1} C_{i,2}
\int_{\RR^n}\left(l(Q)^{-2}\int_{l(Q)}^{\sqrt 2
l(Q)}s(I-e^{-s^2L^*})^Mg(x)\,ds\right)\,
T_{i,2}^{l(Q)}m_{N_i}(x)\,dx,
\end{multline}

\noindent where $C_{i,k}$ are some constants,
$T_{i,k}^{l(Q)}$ are some operator families
satisfying Gaffney estimates (\ref{eq2.19}) with $t\approx l(Q)^2$,  and the
integrals on the right-hand side are interpreted  analogously to
(\ref{eq3.1}). More precisely, each $T_{i,k}$ is a composition of operators of the form (\ref{eq3.7}) and operators coming from
\begin{equation}\label{eq3.11}
l(Q)^{-2}\int_{l(Q)}^{\sqrt 2
l(Q)}s(I-e^{-s^2L})^M\,ds.
\end{equation}
\noindent However, according to (\ref{eq3.7})--(\ref{eq3.8}), the latter can be written as a constant plus an operator in (\ref{eq3.7}), modulo the factor $l(Q)^{-2}L^{-1}$. The negative powers of $l(Q)^2L$ are absorbed in $m_{N_i}$. Hence, each  $T_{i,k}$ is a constant (possibly, zero) plus a linear combination of the terms in the form $e^{-t^2L}$ with    $t\approx l(Q)^2$.

 Applying the Cauchy-Schwarz inequality, we
deduce that
\begin{equation}\label{eq3.12}
 |g(m)| \Bk \leq C \|g\|_{\Lambda_{L^*}^{n(1/p-1)}(\RR^n)}\,
|Q|^{\frac 1p-\frac 12} \sum_{j=0}^\infty 2^{jn\left(\frac
1p-\frac12\right)}\sum_{i,k}\|T_{i,k}^{l(Q)}m_{N_i}\|_{L^2(S_j(Q))}.
\end{equation}

\noindent If $j\leq 3$, then
\begin{equation}\label{eq3.13}
\|T_{i,k}^{l(Q)}m_{N_i}\|_{L^2(S_j(Q))}\leq
C\|m_{N_i}\|_{L^2(\RR^n)} \leq C  |Q|^{\frac 12-\frac 1p},
\end{equation}

\noindent for $i$ and $k$ as above. If $j\geq 3$, we split
\begin{equation}\label{eq3.14}
m_{N_i}=m_{N_i}\,\chi_{\widetilde S_j(Q)}+ m_{N_i}\,\chi_{\RR^n
\setminus \widetilde S_j(Q)},
\end{equation}

\noindent where, as before, 
\begin{equation}\label{eq3.15}
\widetilde S_j(Q)=2^{j+1}Q\setminus 2^{j-2}Q.
\end{equation}

\noindent Then
\begin{equation}\label{eq3.16}
\left\|T_{i,k}^{l(Q)}\left(m_{N_i}\,\chi_{\widetilde
S_j(Q)}\right)\right\|_{L^2(S_j(Q))}\leq C
\left\|m_{N_i}\,\chi_{\widetilde S_j(Q)}\right\|_{L^2(\RR^n)}\leq C
2^{j\left(\frac n2-\frac np-\eps\right)}|Q|^{\frac 12-\frac 1p},
\end{equation}

\noindent by the definition of molecule, and
\begin{equation}\label{eq3.17}
\left\|T_{i,k}^{l(Q)}\left(m_{N_i}\,\chi_{\RR^n \setminus\widetilde
S_j(Q)}\right)\right\|_{L^2(S_j(Q))}\leq C
e^{-\frac{(2^jl(Q))^2}{cl(Q)^2}}\, \|m_{N_i}\|_{L^2(\RR^n)}\leq C
2^{-jN}|Q|^{\frac 12-\frac 1p},
\end{equation}

\noindent for a number $N$ arbitrarily large. Inserting the
results into (\ref{eq3.12}), we finish the proof of
(\ref{eq3.4}). \ep

%\noindent {\it Remark 2}.  Theorem~\ref{t3.3} reveals the nature
%of the Hardy spaces $H^p_L$ for $p<1$. As in the case of the classical
%Hardy spaces, these are not spaces of functions, but rather linear functionals on
%$\Lambda_{L^*}^{n(1/p-1)}(\RR^n)$. The elements of $H^1_L$ are functions from $L^1$.

%\noindent {\it Remark 3}.  Theorem~\ref{t3.3} shows that any
%series $\sum_{j=0}^\infty\lambda_j m_j$ such that
%$\{\lambda_j\}_{j=0}^\infty \in\ell^p$  and $m_j$  are
%$(\hpl,\eps,M)$-molecules, $0<p\leq 1$, $\eps>0$,
%$M>\frac{n}{2}\Bigl(\frac 1p-\frac 12\Bigr)$, converges in $H^p_L(\RR^n)$.

We are now ready to state our duality results generalizing \cite{FeSt,DRS,DL}.

\begin{theorem}\label{t3.5}
Suppose $0<p\leq 1$. Then
\begin{equation}\label{eq3.26}
(H^p_L(\RR^n))^*=\Lambda_{L^*}^{n(1/p-1)}(\RR^n) \mbox{ if
$p<1$,\quad and \quad } (H^1_L(\RR^n))^*=BMO_{L^*}(\RR^n).
\end{equation}
\end{theorem}

\bp The statement about the duality of $H^1_L$ and $BMO_{L^*}$ was proved in \cite{HM}. 
Therefore we consider here only the case $p<1$.

\noindent {\bf Step I.} We start with the left-to-the-right
inclusion.

Assume that $g$ is a linear functional on $H^p_L(\RR^n)$. Then for
every $f\in H^p_L(\RR^n)$
\begin{equation}\label{eq3.27}
|g(f)|\leq \|g\|\,\|f\|_{H^p_L(\RR^n)}.
\end{equation}

\noindent  Theorem~\ref{th4.1}, in particular, implies that every
$(H^p_L,\eps,M)$ - molecule belongs to $H^p_L$ and $\|m\|_{H^p_L}\leq
C$. Hence,
\begin{equation}\label{eq3.28}
|g(m)|\leq C\|g\|.
\end{equation}

\noindent  However, if $\mu \in {\bf M}^{\eps,M}_{n(1/p-1)}$ with
norm $1$, then $\mu$ is a $(p,\eps,M)$ - molecule adapted to
$Q_0$. Therefore, by (\ref{eq3.28}), $g$ defines a linear
functional on ${\bf M}^{\eps,M}_{n(1/p-1)}$. It remains to prove
that the norm (\ref{eq1.20}), understood in the sense of
(\ref{eq3.1}), is finite. To do this, it is enough to show
that for every $\varphi\in L^2(Q)$ such that
$\|\varphi\|_{L^2(Q)}=1$ the function
\begin{equation}\label{eq3.29}
\frac{1}{|Q|^{\alpha/n+1/2}}(I-e^{-l(Q)^2L})^M\varphi,\qquad
\alpha=n\left(\frac 1p-1\right),\qquad M>\frac n2\left(\frac
1p-\frac 12\right),
\end{equation}

\noindent is a $(p,\eps,M)$-molecule (then the claim follows from
(\ref{eq3.28})).

\Bk

Since $\varphi$ is supported in $Q$, by Gaffney estimates
\begin{eqnarray}\nonumber
&&\frac{1}{|Q|^{\alpha/n+1/2}}\,
\|(I-e^{-l(Q)^2L})^M\varphi\|_{L^2(S_j(Q))} \leq C
\Bk\frac{1}{|Q|^{\alpha/n+1/2}}\,
\sum_{k=0}^M \|e^{-kl(Q)^2L}\varphi\|_{L^2(S_j(Q))}\\[4pt]
&&\qquad\leq \frac{C}{|Q|^{\alpha/n+1/2}}\, e^{-\frac{{\rm
dist}\,(S_j(Q),Q)^2}{c l(Q)^2}}\,\|\varphi\|_{L^2(Q)}\leq
\frac{C\, 2^{-jN}}{|Q|^{\alpha/n+1/2}}=\frac{C\,
2^{-jN}}{|Q|^{1/p-1/2}},\label{eq3.30}
\end{eqnarray}

\noindent for every $j\in\NN$ and $N\in\NN$ arbitrarily large.
Similarly, for $k=1,..., M$
\begin{multline}
\frac{1}{|Q|^{\alpha/n+1/2}}\,\|(l(Q)^{-2}L^{-1})^k
(I-e^{-l(Q)^2L})^M\varphi\|_{L^2(S_j(Q))}\\[4pt]
=\frac{1}{|Q|^{\alpha/n+1/2}}\,\left\|(l(Q)^{-2}L^{-1})^k
 \Bigl( \int_0^{l(Q)}\partial_t e^{-t^2L}\,dt\Bigr)^k\,
(I-e^{-l(Q)^2L})^{M-k}\varphi\right\|_{L^2(S_j(Q))}\\[4pt] 
=\frac{1}{|Q|^{\alpha/n+1/2}}\,\left\|
 \Bigl( \int_0^{l(Q)}\,\frac{2t}{l(Q)^2} \,e^{-t^2L}\,dt\Bigr)^k\,
(I-e^{-l(Q)^2L})^{M-k}\varphi\right\|_{L^2(S_j(Q))}\\[4pt]
\leq \frac{C}{|Q|^{\alpha/n+1/2}}\, e^{-\frac{{\rm
dist}\,(S_j(Q),Q)^2}{c l(Q)^2}}\,\|\varphi\|_{L^2(Q)}\leq
\frac{C\, 2^{-jN}}{|Q|^{1/p-1/2}},\label{eq3.31}
\end{multline}

\noindent where we employed Lemma~\ref{l2.4} for the
next-to-the-last inequality. As before, $N\in\NN$ can be taken
arbitrarily large, and that finishes the argument.

\vskip 0.08in

 \noindent {\bf Step II.}
Let us now turn to the
right-to-the-left  inclusion in (\ref{eq3.26}). Let 
$g\in\Lambda_{L^*}^{n(1/p-1)}(\RR^n)$.   We note that the mapping
$$\mathcal{L}_g(f):= \langle g,f\rangle,$$
may be defined initially (by virtue of Lemma \ref{l3.1}) 
when $f$ is a finite linear combination of $(H^p_L,\eps,M)$-molecules,
with $M>(n/2)(1/p-1/2)$, and by the density, in $\HPL$, of the collection of all such $f$,
it is enough to establish the
{\it a priori} bound
\begin{equation}\label{eq32}
|\mathcal{L}_g(f)|\leq C\,\|g\|_{\Lambda_{L^*}^{n(1/p-1)}(\RR^n)} \, \|f\|_{\HPHML}\,,
\end{equation}
for some uniform constant $C$, whenever $f$ is such a finite linear combination.
Indeed, in that case, $\mathcal{L}_g$ extends by continuity to a continuous linear functional on $\HPL$.

Our proof of (\ref{eq32}) is based in part on the approach in \cite{HM}, but we shall
incorporate a simplification
to that approach, which was introduced in \cite{JY}.  As above, 
let $g \in \Lambda_{L^*}^{\alpha}(\RR^n)$, $\alpha = n(1/p-1)$, 
and let $f$ be a finite linear combination of  
$(H^p_L,\eps,M)$-molecules,
with $M>(n/2)(1/p-1/2)$.  We begin by noting that the following 
two facts, first proved in \cite{HM} in the case $p=1$ (equivalently, $\alpha=0$), 
may be extended to the case
$0<p<1$ ($\alpha>0$) mutatis mutandi, and we omit the details.  First, as in \cite{HM}, Lemma 8.3,
we have that 
\begin{equation}\label{e3.59}
\sup_Q \frac{1}{|Q|^{1+2\alpha/n}} \iint_{R_Q}|(t^2L^*)^Me^{-t^2L^*} g|^2 \frac{dxdt}{t} \leq
C \|g\|^2_{ \Lambda_{L^*}^{\alpha}(\RR^n)}\,;
\end{equation}
second, as in \cite{HM}, Lemma 8.4, for $f,g$ as above, the following Calder\'{o}n reproducing formula is valid:
\begin{equation}\label{3.60}
\langle g,f\rangle 
= C_M\,\lim_{\delta\to 0} \int_\delta^{1/\delta}\!\!\int_{\RR^n}
(t^2L^*)^Me^{-t^2L^*}g(x)\,\overline{t^2Le^{-t^2L}f(x)}\,\frac{dxdt}{t}.
\end{equation}

At this point we follow \cite{JY}.  Since $t^2Le^{-t^2L} f \in T^p$, 
we may invoke the result of \cite{CMS} to obtain the decomposition
$$t^2Le^{-t^2L} f = \sum \lambda_j A_j,$$ 
where each $A_j$ is a $T^p$ atom, supported in a Carleson box
$R_{Q_j}$, and where
$\{\lambda_j\} \in \ell^p$, with
\begin{equation}\label{e3.61}\left(\sum|\lambda_j|^p\right)^{1/p} 
\lesssim \|F\|_{T_p} \approx \|f\|_{\HPHML}.\end{equation}
Using (\ref{3.60}), we then have
\begin{multline*}|\langle g,f\rangle|\,\leq \,C\sum |\lambda_j| 
\dint_{\RR^{n+1}_+}
|(t^2L^*)^Me^{-t^2L^*}g(x)|\,|A_j(x,t)|\,\frac{dxdt}{t}\\\leq C
\sum |\lambda_j| 
\left(\dint_{R_{Q_j}}
|(t^2L^*)^Me^{-t^2L^*}g(x)|^2\,\frac{dxdt}{t}\right)^{1/2}\,|Q_j|^{\frac{1}{2}-\frac{1}{p}}\\
\leq C \sum|\lambda_j|\,\|g\|_{\Lambda_{L^*}^\alpha},
\end{multline*}
where in the second inequality we have used the definition of a $T^p$-atom (cf. (\ref{e4.8})),
and in the last inequality we have used (\ref{e3.59}) with $\alpha = n(1/p-1).$
The desired bound (\ref{eq32}) now follows from (\ref{e3.61}), since $p<1$.

 \ep

 \Bk

\section{Square function characterizations and interpolation.}\label{s4}
\setcounter{equation}{0}

Recall the square function definition of Hardy spaces given in
(\ref{eq1.8})--(\ref{eq1.9}). In fact, there is certain
flexibility in the choice of the square function which gives an
equivalent norm in $H^p_L(\RR^n)$. It is possible to replace
$\psi(\xi)=\xi e^{-\xi}$, $\xi=t^2L$, in (\ref{eq1.8}) by another
function of $\xi$ with holomorphic extension to an open sector
of the complex plane, \Bk provided it has enough decay at zero and
infinity. One way to see this is to re-prove the molecular decomposition of Hardy spaces, this time using
a square function based on $\psi$, Lemma~\ref{l2.7} and quadratic estimates in \cite{Mc}.

Now we present a different approach, via the connection with the tent spaces
(cf. (\ref{eq4.1}), (\ref{eq4.2})), again using fundamentally the ideas of \cite{CMS}.
In a different context this has been done in \cite{AMcR}.  
Here we will follow a similar path, pointing out the aspects which are particular to our setting.

Let $\omega<\mu<\pi/2$ and $\psi\in\Psi(\Sigma_{\mu}^0)$. According to the quadratic estimates in \cite{Mc} the operator
\begin{equation}\label{eq4.5}
Q_{\psi,L}\, f(x,t):=\psi(t^2L)f(x),\quad (x,t)\in\RR^{n+1}_+,
\end{equation}

\noindent is bounded from $L^2(\RR^n)$ to $T^2(\RR^{n+1}_+)$. Then for every $\psi\in\Psi(\Sigma_{\mu}^0)$ the operator
\begin{equation}\label{eq4.6}
\pi_{\psi,L}\, F(x):=\int_0^\infty
\psi(t^2L)F(x,t)\,\frac{dt}{t},\quad x\in\RR^n,
\end{equation}

\noindent is well-defined for all $F\in T^2(\RR^{n+1}_+)$ and bounded from $T^2(\RR^{n+1}_+)$ to $L^2(\RR^n)$ by duality. Indeed, the operator $\pi_{\psi,L}$ is the 
adjoint of the operator $Q_{\psi,L^*}\,$, and vice versa.   In the sequel, for the sake of notational convenience, we shall sometimes omit the subscript $L$, and write merely $Q_{\psi},\,\pi_{\psi}$
when there is no chance of confusion.

Finally, for $\psi,\widetilde \psi\in \Psi(\Sigma_\mu^0)$
and $f\in H^\infty(\Sigma_\mu^0)$ let ${\mathcal Q}^f:=Q_{\psi}\circ f \circ
\pi_{\widetilde\psi}$, i.e.,
\begin{equation}\label{eq4.7}
{\mathcal Q}^fF(x,s):=\int_0^{\infty}\left(\psi(s^2L)f(L)
\widetilde{\psi}(t^2L)F(\cdot,t)\right)(x)
\,\frac{dt}{t},\quad (x,s)\in\RR^{n+1}_+.
\end{equation}

\noindent Then it follows from the observations above that ${\mathcal Q}^f$ is bounded in $T^2(\RR^{n+1}_+)$, with the norm bounded by $\|f\|_{L^{\infty}(\Sigma_{\mu}^0)}$. We will sometimes write ${\mathcal Q}$ in place of ${\mathcal Q}^f$ when $f=1$.

\begin{proposition}\label{p4.1} Let $\mu\in(\omega,\pi/2)$. Then for every $\psi,\widetilde \psi\in \Psi(\Sigma_\mu^0)$
and $f\in H^\infty(\Sigma_\mu^0)$ the operator
${\mathcal Q}^f$ originally defined on $T^2(\RR^{n+1}_+)$ extends by continuity to a bounded operator on $T^p(\RR^{n+1}_+)$
provided that either

\noindent (1) $0<p\leq 2$, $\psi\in \Psi_{\alpha,\beta}(\Sigma_\mu^0)$, $\widetilde \psi\in \Psi_{\beta,\alpha}(\Sigma_\mu^0)$, or\\
\noindent (2) $2\leq p<\infty$, $\psi\in \Psi_{\beta,\alpha}(\Sigma_\mu^0)$, $\widetilde \psi\in \Psi_{\alpha,\beta}(\Sigma_\mu^0)$,\\
where $\alpha>0$,   $\beta>\frac{n}{2}\left(\max\{\frac
1p,1\}-\frac 12\right)$. Moreover,
\begin{equation}\label{eq4.8}
\|{\mathcal Q}^fF\|_{T^p(\RR^{n+1}_+)}\leq C \|f\|_{L^\infty(\Sigma_\mu^0)}\|F\|_{T^p(\RR^{n+1}_+)},\quad \mbox{for all}\quad F\in T^p(\RR^{n+1}_+).
\end{equation}
\end{proposition}

\noindent{\it Proof of Proposition~\ref{p4.1}.}\, Let $0<p\leq 2$. Using the Lemma~\ref{l4.2} for any $a,b$ such that $0<a<\alpha$ and $0<b<\beta$ one can write
\begin{equation}\label{eq4.12}
{\mathcal Q}^fF(x,s)=\int_0^{\infty}\min\left\{\Bigl(\frac st\Bigr)^{2a}, \Bigl(\frac ts\Bigr)^{2b}\right\}\,T_{s^2,t^2} F(\cdot,t)(x)
\,\frac{dt}{t},\quad (x,s)\in\RR^{n+1}_+,
\end{equation}

\noindent where \\
\noindent $(1)$ $\{T_{s,t}\}_{s\leq t}$ satisfy the $L^2$ off-diagonal estimates in $t$ of order $\beta+a$ uniformly in $s\leq t$, \\
\noindent $(2)$ $\{T_{s,t}\}_{t\leq s}$ satisfy the $L^2$ off-diagonal estimates in $s$ of order $\alpha+b$ uniformly in $t\leq s$. \\
\noindent with the constant bounded by $\|f\|_{L^\infty(\Sigma_\mu^0)}$.
 Note that the constants $a,b$ can be chosen so that both $\alpha+b>\frac{n}{2}\left(\max\{\frac
1p,1\}-\frac 12\right)$ and $\beta+a>\frac{n}{2}\left(\max\{\frac
1p,1\}-\frac 12\right)$. Then
there exist some $M>\frac{n}{2}\left(\max\{\frac
1p,1\}-\frac 12\right)$ and some
$C >0$ such that for arbitrary
closed sets $E,F\subset \RR^n$
\begin{equation}\label{eq4.13}
\|T_{s^2,t^2}g\|_{L^2(F)}\leq C\,\|f\|_{L^\infty(\Sigma_\mu^0)}\,\min\left\{1,\frac{\max\{t,s\}}{{\rm
dist}\,(E,F)}\right\}^{2M}\,\|g\|_{L^2(E)},
\end{equation}

\noindent for every $s,t>0$ and every $g\in L^2(\RR^n)$ supported in
$E$.

The remainder of the proof follows the same path as that of Theorem~4.9 in \cite{AMcR}. 
Suppose first that $p\leq 1$.  By density of $T^2(\RR^n_+)\cap T^p(\RR^n_+)$ in $T^p(\RR^n_+)$ 
it is enough to establish an {\it a priori} estimate for $F(x,t) \in T^2(\RR^n_+)\cap T^p(\RR^n_+)$.
We may then use the atomic decomposition of tent spaces in \cite{CMS} (cf. Proposition \ref{prop4.7} above) to reduce (\ref{eq4.8}) to the atomic estimate
\begin{equation}\label{eq4.14}
\|{\mathcal Q}^fA\|_{T^p(\RR^n_+)}\leq C \|f\|_{L^\infty(\Sigma_\mu^0)} \mbox{\, uniformly for $T^p(\RR^n_+)$-atoms $A$.}
\end{equation}

\noindent Then one breaks down ${\mathcal Q}^fA$ into a part close to the support of $A$ and a part away from the support of $A$. Close to the support we use the boundedness of ${\mathcal Q}^f$ in $T^2(\RR^n_+)$, and away from the support we use (\ref{eq4.13}). The details can be recovered carefully following an analogous argument in \cite{AMcR}.
Then the case $1<p\leq 2$ follows by interpolation and the case $2\leq p<\infty$ is obtained by duality. \ep

\begin{proposition}\label{p4.3}
Let $\mu\in(\omega,\pi/2)$ and $\psi\in \Psi(\Sigma_\mu^0)$. The operator $Q_{\psi,L}$ originally defined on $L^2(\RR^n)$ by  the
 formula (\ref{eq4.5})
extends to a bounded operator
\begin{equation}\label{eq4.15}
Q_{\psi,L}:H^p_L(\RR^n)\longrightarrow T^p(\RR^{n+1}_+),
\end{equation}

\noindent provided that

$$
\mbox{either \quad
 $(1)$ \,$0<p\leq 2$, $\psi\in \Psi_{\alpha,\beta}(\Sigma_\mu^0)$, \quad or\quad $(2)$\, $2\leq p<\infty$, $\psi\in \Psi_{\beta,\alpha}(\Sigma_\mu^0)$,} $$
\noindent where $\alpha>0$ and $\beta>\frac{n}{2}\left(\max\{\frac
1p,1\}-\frac 12\right)$.

The operator $\pi_{\psi,L}$ defined on $T^2(\RR^{n+1}_+)$ by means of (\ref{eq4.6})
extends to a
bounded operator
\begin{equation}\label{eq4.16}
\pi_{\psi,L}:T^p(\RR^{n+1}_+)\longrightarrow H^p_L(\RR^n),
\end{equation}

\noindent provided that
$$
\mbox{either \quad
 $(1)$ \,$0<p\leq 2$, $\psi\in \Psi_{\beta,\alpha}(\Sigma_\mu^0)$, \quad or\quad $(2)$\, $2\leq p<\infty$, $\psi\in \Psi_{\alpha,\beta}(\Sigma_\mu^0)$,} $$

\noindent where $\alpha>0$ and $\beta>\frac{n}{2}\left(\max\{\frac
1p,1\}-\frac 12\right)$.

\end{proposition}

\noindent {\it Remark}.\,    Before proving the Proposition, we note
that for $\psi,\widetilde\psi\in \Psi(\Sigma_\mu^0)$ such that
$\int_0^\infty \psi(t)\widetilde \psi(t)\,\frac{dt}{t}=1,$ we have the following
Calder\'on reproducing formula:
\begin{equation}\label{eq4.17}
\pi_{\psi}\circ Q_{\widetilde \psi}=\pi_{\widetilde\psi}\circ Q_{\psi}=I \quad\mbox{in $L^2(\RR^n)$.}
\end{equation}

\noindent Moreover, for every non-trivial $\psi\in \Psi(\Sigma_\mu^0)$, 
such $\widetilde \psi$ can be found, for example, taking
\begin{equation}\label{eq4.18}
\widetilde \psi(z):=\overline{\psi}(z)\Bigl(\int_0^\infty |\psi(t)|^2\,\frac{dt}{t}\Bigr)^{-1}, \quad z\in \Sigma_\mu^0.
\end{equation}

\vskip 0.08in

\noindent {\it Proof of Proposition~\ref{p4.3}.}  Let $\psi_0(z)=ze^{-z}$, $z\in \Sigma_\mu^0$.  Then the boundedness of the corresponding $Q_{\psi_0}$ in (\ref{eq4.15}) for $0<p\leq 2$
follows directly from the definitions of $\HPL, \, 0<p\leq2$, and $T^p(\RR^{n+1}_+)$.

Now take any $\psi\in \Psi_{\beta,\alpha}$ and $0<p\leq 2$. For every $F\in T^p(\RR^{n+1}_+)\cap T^2(\RR^{n+1}_+)$
\begin{equation}\label{eq4.19}
\|\pi_{\psi}F\|_{H^p_L(\RR^n)}=\|Q_{\psi_0}\circ \pi_{\psi}F\|_{T^p(\RR^{n+1}_+)},
\end{equation}

\noindent and due to Proposition~\ref{p4.1} the last expression above is controlled by $\|F\|_{T^p(\RR^{n+1}_+)}$. Then (\ref{eq4.16}) follows by a density argument.

Next, let $\psi\in \Psi_{\alpha,\beta}(\Sigma_\mu^0)$, $0<p\leq 2$. Since $L^2$ is dense in $H^p_L$, it is enough to prove that

\begin{equation}\label{eq4.20}
\|Q_{\psi}f\|_{T^p(\RR^{n+1}_+)}\leq C \|f\|_{H^p_L(\RR^n)},\quad\mbox{for every}\quad f\in H^p_L\cap L^2.
\end{equation}

\noindent By definition $Q_{\psi_0}f\in T^p(\RR^{n+1}_+)$ for every  $f\in H^p_L\cap L^2$. Let $M$ be the smallest integer larger than $\frac{n}{2}\left(\max\{\frac
1p,1\}-\frac 12\right)$ and $\widetilde \psi_0(\xi):=\xi^M e^{-\xi}$, $\xi\in \Sigma_\mu^0$.
Then $\int_0^\infty \psi_0(t)\widetilde \psi_0(t)\,\frac{dt}{t}=C_M$, and hence, by (\ref{eq4.17}) we have

\begin{equation}\label{eq4.21}
 f=\frac{1}{C_M}\,\,\pi_{\widetilde \psi_0}\circ Q_{\psi_0}f\quad\mbox{ for }\quad f\in L^2.
\end{equation}

\noindent
 Note that $\widetilde \psi_0 \in \Psi_{M,N}$ for every $N>0$.
Therefore,
\begin{equation*}
\|Q_{\psi}f\|_{T^p(\RR^{n+1}_+)}=C\|Q_{\psi}\circ \pi_{\widetilde \psi_0}\circ Q_{\psi_0} f\|_{T^p(\RR^{n+1}_+)}\leq C\|Q_{\psi_0} f\|_{T^p(\RR^{n+1}_+)}=C \|f\|_{H^p_L(\RR^n)}, 
\end{equation*}
where the inequality is a consequence of Proposition \ref{p4.1}.

For $p>2$ we use the duality between the operators $\pi$ and  $Q$. \ep

\noindent {\it Remark.}\,
We would like to mention that  in \cite{HvNP} the authors developed an alternative approach to (\ref{eq4.15}).\vskip 0.08in

\noindent {\it Remark.}\,
The results of the Proposition~\ref{p4.3} lead to an alternative  molecular decomposition of Hardy spaces, defining molecules as the images of the atoms of tent spaces under $\pi_{\psi}$ for appropriate $\psi$ (cf. \cite{AMcR}). \vskip 0.08in

\noindent {\it Remark}.\,
The tent spaces have an appropriate counterpart when $p=\infty$
and the results of Proposition~\ref{p4.3} extend to this case as well
(see \cite{HM}). \vskip 0.08in

Proposition~\ref{p4.3}, in particular, provides the square function
characterization for the Hardy spaces $H^p_L$ with $p>2$, which were originally defined by duality (\ref{eq1.10}).

\begin{corollary}\label{c4.4}
Let $\psi$ be a nontrivial function satisfying either 
\begin{enumerate}
\item[(1)] $0<p\leq 2$, $\psi\in \Psi_{\alpha,\beta}(\Sigma_\mu^0),\,$  or
\item[(2)] $2\leq p<\infty$, $\psi\in \Psi_{\beta,\alpha}(\Sigma_\mu^0)$, 
\end{enumerate}
where $\alpha>0$ and $\beta>\frac{n}{2}\left(\max\{\frac
1p,1\}-\frac 12\right)$. 
Define $H^p_{\psi,L}(\RR^n)$ to be the completion of the space
$$\mathbb{H}^p_{\psi,L}(\mathbb{R}^n) := \{f\in L^2(\RR^n): Q_{\psi,L}\,f \in T^p(\RR^{n+1}_+)\},$$
with respect to the norm
\begin{equation}\label{eq4.23}
\|f\|_{H^p_{\psi,L}(\RR^n)}:=\|Q_{\psi,L}f\|_{T^p(\RR^{n+1}_+)}=
\left\|\left(\dint_{\Gamma(\cdot)}|\psi(t^2L)f(y)|^2\,\frac
{dydt}{t^{n+1}}\right)^{1/2}\right\|_{L^p(\RR^n)}.
\end{equation}
Then $\HPL = H^p_{\psi,L}$, with equivalence of norms.
\end{corollary}

\bp For $0<p\leq2$, by the definitions it is enough to establish equality of the dense spaces
$L^2(\RR^n)\cap \HPL$ and $\mathbb{H}^p_{\psi,L}(\RR^n)$, with equivalence of norms.
One direction is precisely the estimate (\ref{eq4.20}) above.  The opposite direction is proved in
exactly the same way as (\ref{eq4.20}), by simply interchanging 
the roles of $\psi$ and $\psi_0$, and observing that the reproducing formula (\ref{eq4.21}) 
is still valid (with a different constant), for the same choice of $\widetilde{\psi}_0$, but with
$\psi_0$ replaced by $\psi$.  We omit the routine details.

The case $2<p<\infty$ is slightly more involved.  We begin by claiming that
$L^2(\RR^n) \cap \HPL$ is dense in $\HPL$ (this fact is immediate by definition only
for the range $0<p\leq 2$).  To prove the claim, let $\chi_K$ denote the characteristic function
of the set $\{(x,t) \in \RR^{n+1}_+: |x|<K,\, 1/K < t < K\}$, so that for $F \in T^p,\, 2<p<\infty$, we have that
$F_K := F \chi_K \in T^2 \cap T^p$, and also that $F_K \to F$ in $T^p$.
Now given $\psi \in \Psi_{\beta,\alpha}(\Sigma_\mu^0)$, choose $\widetilde{\psi} \in \Psi_{\alpha,\beta}(\Sigma_\mu^0)$ satisfying the reproducing formula (\ref{eq4.17}).  Then by (\ref{eq4.15})
and (\ref{eq4.16}), the reproducing formula extends to $H^{p'}_{L^*}(\RR^n)$ 
(since $L^2 \cap H^{p'}_{L^*}$ is dense in the latter space), 
and thus by duality to $\HPL$.   Consequently, given $f \in \HPL,\, 2<p<\infty$, we may write
$$f = \pi_{\widetilde\psi,L}\circ Q_{\psi,L} \,f = \lim_{K\to \infty} \pi_{\widetilde\psi,L} \left(
(Q_{\psi,L} f) \,\chi_K\right),$$
where by our previous remarks and  (\ref{eq4.16}), the limit exists in $\HPL$.  Moreover,
$F_K:= (Q_{\psi,L} f) \,\chi_K \in T^2 \cap T^p$, so that $\pi_{\widetilde\psi,L} F_K \in L^2(\RR^n)
\cap \HPL$.  Thus, the claimed density holds.

Therefore, it is enough to prove that $L^2(\RR^n) \cap \HPL = 
\mathbb{H}^p_{\psi,L}(\mathbb{R}^n),$ with equivalence of norms.
One direction follows immediately from (\ref{eq4.15}).  We now proceed to establish the other direction, namely that for $f\in \mathbb{H}^p_{\psi,L}(\mathbb{R}^n),$ we have
$$\|f\|_{\HPL} \lesssim \|Q_{\psi,L} \,f\|_{T^p(\RR^{n+1}_+)}.$$
In turn, by the definition of $\HPL,\, 2<p<\infty,$ as a dual space, it is enough to show that for $g\in L^2(\RR^n)\cap H^{p'}_{L^*}(\RR^n)$, we have
\begin{equation}\label{eq4.dual}\left|\int_{\RR^n} f\, \overline{g}\,\right| \lesssim \|Q_{\psi,L} f\|_{T^p(\RR^{n+1}_+)}\, \|g\|_{ H^{p'}_{L^*}(\RR^n)}.\end{equation}

To this end, given $\psi \in \Psi_{\beta,\alpha}(\Sigma_\mu^0)$, as above we 
choose $\widetilde{\psi} \in \Psi_{\alpha,\beta}(\Sigma_\mu^0)$ satisfying the 
reproducing formula (\ref{eq4.17}), so that
\begin{multline*}\left|\int_{\RR^n} f\, \overline{g}\,\right| = \left|\int_{\RR^n} 
\pi_{\widetilde\psi,L}\circ Q_{\psi,L} \,f\, \overline{g}\,\right|\\[4pt] \leq\,
\|Q_{\psi,L} \,f\|_{T^p(\RR^{n+1}_+)}\,\|Q_{\widetilde{\psi},L^*} \,g\|_{T^{p'}(\RR^{n+1}_+)}\\[4pt] 
\leq\,C\,
\|Q_{\psi,L} \,f\|_{T^p(\RR^{n+1}_+)}\,\|g\|_{ H^{p'}_{L^*}(\RR^n)},
\end{multline*}
as desired, where in the last step we have used (\ref{eq4.20}).
\ep

Let us now turn to the interpolation property.
%One of the
%applications of the tent spaces approach is a fairly
%straightforward \Bk proof of the interpolation of Hardy spaces.
One of the most important features of the classical Hardy spaces lies in the
fact that they form a complex interpolation scale including, in
particular, $L^p(\RR^n)$ for some values of $p$ (in fact,
$1<p<\infty$). It has to be mentioned that Calder{\'o}n's
original method of complex interpolation was
defined for Banach spaces and could not be immediately extended to
the case when the underlying spaces were only quasi-Banach
($p<1$). One reason for that is a possible failure of the maximum
modulus principle in quasi-Banach spaces. Over the years there have
been developed several approaches to this issue (see, in
particular, {{\cite{CT}}}, {{\cite{JJ}}}, {{\cite{CwMS}}},
{{\cite{GM}}} regarding the classical Hardy spaces).  Here we
are going to employ an extension of the complex interpolation
method to analytically convex spaces described in \cite{KaMi},
\cite{KMM}.

\begin{lemma}\label{l4.5}
For each $0<\theta<1$ and $0<p_0,p_1<+\infty$,
\begin{equation}\label{eq4.24}
\left[H^{p_0}_L({\mathbb R}^n),H^{p_1}_L({\mathbb
R}^n)\right]_{\theta} =H^p_L({\mathbb R}^n),\quad\mbox{where
}1/p=(1-\theta)/p_0+\theta/p_1,
\end{equation}

\noindent and
\begin{multline}\label{eq4.25}
\left[H^{p_0}_L({\mathbb R}^n),\mbox{BMO}_L({\mathbb
R}^n)\right]_{\theta} =H^p_L({\mathbb R}^n),\\[4pt]
0<\theta<1,\,\,0<p_0<+\infty,\,\, 1/p=(1-\theta)/p_0.
\end{multline}
\end{lemma}

\bp The proof of (\ref{eq4.24}) is a combination of an analogous
result for the tent spaces and Proposition~\ref{p4.3}. First of all,
(\ref{eq4.4})
 holds for all $0<p_0<p_1\leq+\infty$ (this is stated in \cite{CMS}, Proposition 6, p. 326; complete details
 are given in \cite{VerbIP}). 
 On the other hand, by
Proposition~\ref{p4.3}, if $0<p<\infty$, Hardy spaces are the retracts of the
corresponding tent spaces, i.e. there exists an operator mapping
any tent space to the corresponding Hardy space and having the
right inverse (actually, this is also true for $p=\infty$, if we designate 
$BMO_L(\RR^n)=: H^\infty_L(\RR^n)$;  the proof is implicit in \cite{HM}, however, we shall not need to make explicit use of this fact in the sequel). More precisely, 
given any pair $0<p_0<p_1<\infty$, we can take
$\psi\in\Psi_{\beta,\beta}$, where  $\beta>\frac{n}{2}\left(\max\{\frac
1{p_0},1\}-\frac 12\right)$ and $\widetilde \psi\in\Psi_{\beta,\beta}$ as in (\ref{eq4.18}).
Then for all $p$ between $p_0$ and $p_1$ the operator $\pi_\psi$ maps $T^p$ to $H^p_L$, and $Q_{\widetilde \psi}:H^p_L\to T^p$ is its right inverse.
Therefore, (\ref{eq4.4}) implies (\ref{eq4.24})
once we make sure that $T^{p_0}(\RR^{n+1})+T^{p_1}(\RR^{n+1}_+)$
is analytically convex (see Lemma 7.11 in \cite{KMM}). This,
however, follows from Theorem 7.9 in \cite{KMM} (see also the discussion
in \cite{VerbIP}, Section~3, and in the proof of Lemma~\ref{l8.5} below). The space $BMO_L$ can then be incorporated by
duality and Wolff's reiteration theorem \cite{W}, once we have shown that, given any fixed $p_0>0$, 
there is some large ambient Banach space into which  every $\HPL, \,  p_0\leq p<\infty,$ and also $BMO_L(\RR^n)$,  may be
continuously embedded. 
We shall establish the existence of such an ambient space in an appendix (cf. Section \ref{s10} below).
\ep

\section{Riesz transform characterization of Hardy spaces.}\label{s5}
\setcounter{equation}{0}

Let us recall that for a given operator $L$ the interval $(p_-(L),p_+(L))$ is the interior of the interval of $L^p$-boundedness of the heat semigroup and $2+\eps(L)$ is an upper bound for the interval of  $L^p$-boundedness of the Riesz transform.  As pointed out in the introduction, we have
\begin{equation}\label{eq5.1}
\nabla L^{-1/2}:L^p(\RR^n)\longrightarrow L^p(\RR^n)\quad\iff\quad p_-(L)<p<2+\eps(L),
\end{equation}

\noindent and the bounds $p_-(L)<\frac{2n}{n+2}$, $\eps(L)>0$ are sharp in the sense of Corollary~\ref{c2.3}. In the present section we aim to extend (\ref{eq5.1}) to other values of $p$, passing to the Hardy $H^p_L$ spaces, and to prove the reverse estimate for a certain range of $p,$ thus establishing
for such $p$ the equivalence of the spaces $\HPL$ and $\HRL$ (cf. (\ref{eq1.9}) and (\ref{e1.hrldef}))  Our main result in this section is the following.

\begin{theorem}\label{t5.1}
Let $1<r\leq 2$ be such that the family $\{e^{-tL}\}_{t>0}$
satisfies $L^r-L^2$ off-diagonal estimates. We then have
\begin{equation}\HPL = \HRL\,,\qquad \frac{rn}{n+r}<p< 2+\eps(L) \end{equation} Moreover,  
we have the following equivalence of norms:
\begin{equation}\label{eq5.2}
\|f\|_{H^p_L(\RR^n)}\approx \|\nabla L^{-1/2}f\|_{L^p(\RR^n)}\,,\qquad \max\left\{1,\frac{rn}{n+r}\right\}<p<2+\eps(L)
\end{equation}
and if $rn/(n+r)\leq 1$, then 
\begin{equation}\label{eq5.3}
\|f\|_{H^p_L(\RR^n)}\approx \|\nabla L^{-1/2}f\|_{H^p(\RR^n)}\,,\qquad \frac{rn}{n+r}<p\leq 1.
\end{equation}
\end{theorem}

\noindent{\it Remark.}\, Note that, in particular,  (\ref{eq5.2})
holds for every $p$ such that
$\max\left\{1,\frac{p_-(L)n}{n+p_-(L)}\right\}<p<2+\eps(L)$, and if
$\frac{p_-(L)n}{n+p_-(L)}<1$, then (\ref{eq5.3}) holds for every $p$
such that $\frac{p_-(L)n}{n+p_-(L)}<p\leq 1$.

\vskip 0.08in

The proof of the Theorem will be split into
Propositions~\ref{p5.2}--\ref{p5.5}. Let us start with the case
$p\leq 1$. For the sake of notational convenience, given $p\in (0,1]$,
we shall throughout
this section fix $M>(n/2)(1/p-1/2)$ and $\eps>0$ (recall that, as we have seen, any such choice leads to an equivalent $H^p_L$ space), and we may therefore
refer to $(\hpl,\eps,M)$-molecules simply as $H^p_L$-molecules. The first result concerns the boundedness of the
Riesz transform.

\begin{proposition}\label{p5.2}
For every $p$ such that  $\frac{n}{n+1}< p \leq 1$, there is a constant $C$ depending only
on $n,p$ and ellipticity (and our fixed choices of $M$ and $\eps$), 
such that the Riesz transform $\nabla L^{-1/2},$ 
defined initially
on $L^2 \cap \HPL = \HPHML$ (cf. (\ref{hpequivalence})), satisfies
\begin{equation}\label{eq5.4}
\|\nabla L^{-1/2} f\|_{H^p(\RR^n)} \leq C \|f\|_{\HPHML}\,,
\end{equation}
and therefore extends to a bounded operator 
$\nabla L^{-1/2}: H^p_L(\RR^n)\longrightarrow H^p(\RR^n).$
\end{proposition}

\bp We begin by recalling that the classical Hardy spaces can be characterized via 
a molecular decomposition (see, e.g., \cite{CW}).  Our Theorem~\ref{th4.1} with $L=-\Delta$
provides one such characterization, but a more traditional
version is as follows.

The function $m\in L^2(\RR^n)$ is an $H^p$-molecule, $0<p\leq 1$,
if it satisfies (\ref{eq3.3}) for $k=0$ and
\begin{equation}\label{eq5.5}
\int_{\RR^n} x^{\alpha}m(x)\,dx=0, \qquad 0\leq |\alpha|\leq
\widetilde M,
\end{equation}

\noindent for some $\widetilde M\in\NN\cup\{0\}$ such that
$\widetilde M\geq [n(1/p-1)]$, with $[\gamma]$ denoting the
integer part of $\gamma\in\RR$. Given $p\in (0,1]$, fix some
$\widetilde M$ as above. Then the classical real variable
Hardy space can be realized as
\begin{equation}\label{eq5.6}
H^p(\RR^n)=\left\{\sum_{i=0}^\infty
\lambda_jm_j:\,\{\lambda_j\}_{j=0}^\infty\in \ell^p \mbox{ and
}\,m_j \mbox{ are $H^p$-molecules}\right\},
\end{equation}
with 
\begin{equation*}
\|f\|_{H^p(\RR^n)}\approx
\inf\Bigl\{\Bigl(\sum_{j=0}^\infty|\lambda_j|^p\Bigr)^{1/p} \Bigr\},
\end{equation*}
where the infimum runs over all decompositions $f=\sum_{j=0}^\infty\lambda_j m_j,$
converging in the space of tempered distributions ${\mathcal S}'$, such that
$\{\lambda_j\}_{j=0}^\infty \in\ell^p$ and each $m_j$ is an $H^p$ molecule. 
We do not know if this particular
version of molecular decomposition is explicitly stated anywhere
but it readily follows from the classical arguments (see
\cite{CW}, \S{2} of \cite{TW}, and \cite{GaCu}). 
%moreover, the $H^p_{-\Delta}$ -molecules, as defined in the 
%present paper, satisfy (\ref{eq5.5}) as well, so one may invoke Theorem \ref{th4.1} also.

Having these facts at hand, we first show that the Riesz transform maps
$H^p_L$-molecules into $H^p$-molecules.
 Let $m\in L^2(\RR^n)$ be
an $H^p_L$-molecule associated with some cube $Q\in\RR^n$ (and
$M>\frac{n}{2}\Bigl(\frac 1p-\frac 12\Bigr)$, $\eps>0$ fixed as above).
Then
\begin{equation}\label{eq5.7}
\|\nabla L^{-1/2}m\|_{L^2(2Q)}\leq C \|m\|_{L^2(\RR^n)}\leq
C\,l(Q)^{n/2-n/p},
\end{equation}

\noindent using boundedness of $\nabla L^{-1/2}$ in $L^2(\RR^n)$.
Next, for $i\geq 2$
\begin{eqnarray}\nonumber
\|\nabla L^{-1/2}m\|_{L^2(S_i(Q))}&\leq & \|\nabla
L^{-1/2}(I-e^{-l(Q)^2L})^Mm\|_{L^2(S_i(Q))}\\[4pt]
&&+ \|\nabla
L^{-1/2}[I-(I-e^{-l(Q)^2L})^M]m\|_{L^2(S_i(Q))}=:I+II.\label{eq5.8}
\end{eqnarray}

\noindent According to Theorem~3.2 in \cite{HM} (see also
Lemma~2.2 in \cite{HoMa}), for all closed sets $E$, $F$ in $\RR^n$
with ${\rm dist }(E,F)>0$, if  $f\in L^2(\RR^n)$ is supported in
$E$,  then  \Bk
\begin{eqnarray}\label{eq5.9}
\|\nabla L^{-1/2}(I-e^{-tL})^{M}f\|_{L^2(F)}&\!\leq \,
C\,\left(\frac{t}{{\rm
dist}\,(E,F)^2}\right)^{M}\,\|f\|_{L^2(E)},&\forall t>0,\\
\|\nabla L^{-1/2}(tLe^{-tL})^{M}f\|_{L^2(F)}&\!\leq \,
C\,\left(\frac{t}{{\rm
dist}\,(E,F)^2}\right)^{M}\,\|f\|_{L^2(E)},&\forall t>0.\label{eq5.10}
\end{eqnarray}

\noindent  Therefore,
\begin{eqnarray}\nonumber
I&\leq & \|\nabla
L^{-1/2}(I-e^{-l(Q)^2L})^M(m\chi_{2^{i-2}Q})\|_{L^2(S_i(Q))}\\[4pt]
&& +\,\,\|\nabla L^{-1/2}(I-e^{-l(Q)^2L})^M(m\chi_{\RR^n\setminus
2^{i-2}Q})\|_{L^2(S_i(Q))}\nonumber\\[4pt]
&\leq & C 2^{-2Mi} \|m\|_{L^2(2^{i-2}Q)}+C
\|m\|_{L^2(\RR^n\setminus 2^{i-2}Q)}\nonumber\\[4pt]
&\leq & C 2^{-2Mi} l(Q)^{n/2-n/p}+C
(2^il(Q))^{n/2-n/p}\,2^{-i\eps}.\label{eq5.11}
\end{eqnarray}

\noindent Since $M>\frac{n}{2}\Bigl(\frac 1p-\frac 12\Bigr)$, the
estimate (\ref{eq5.11}) implies
\begin{equation}\label{eq5.12}
I\leq C (2^il(Q))^{n/2-n/p}\,2^{-i\epsilon},
\end{equation}

\noindent where $\epsilon=\min\{\eps, 2M-n/p+n/2\}>0$.

Turning to the second part of (\ref{eq5.8}), we observe that
\begin{eqnarray}\nonumber
&&\|\nabla L^{-1/2}[I-(I-e^{-l(Q)^2L})^{M}]m\|_{L^2(\RR^n)}\leq C
\sup_{1\leq k\leq M}\|\nabla
L^{-1/2}e^{-kl(Q)^2L}m\|_{L^2(\RR^n)}\\[4pt]
&&\qquad  \leq C \sup_{1\leq k\leq M}\left\|\nabla
L^{-1/2}\left(\frac k{M}\, l(Q)^{2}Le^{-\frac
k{M}l(Q)^2L}\right)^{M}(l(Q)^{-2}L^{-1})^{M}m\right\|_{L^2(\RR^n)}.\label{eq5.13}
\end{eqnarray}

\noindent This allows to employ the argument above, using
(\ref{eq5.10}) in place of (\ref{eq5.9}), to prove an analogue of
(\ref{eq5.12}) for the expression $II$.

Finally, the vanishing moment condition (\ref{eq5.5}) is
satisfied, since
\begin{equation}\label{eq5.14}
\int_{\RR^n}\nabla L^{-1/2}m(x)\,dx=0,
\end{equation}

\noindent and one can take $\widetilde M=0$ when
$p>\frac{n}{n+1}$.

So far, we have established that Riesz transform maps
$H^p_L$-molecules into $H^p$-molecules for $p\in
\left(\frac{n}{n+1},1\right]$. Let us now show that this implies the desired
estimate (\ref{eq5.4}). 
To this end, let $f \in \HPHML$, so that by definition we may select 
an $L^2$ convergent molecular decomposition
$f=\sum_{i=0}^\infty\lambda_i m_i$, 
where each $m_i$ is an $H^p_L$-molecule, such that
$$\|f\|_{\HPHML}\approx \Bigl(\sum_{i=0}^\infty
|\lambda_i|^p\Bigr)^{1/p}.$$ 
By the $L^2$ convergence of the sum, we have that
$$\nabla L^{-1/2} f = \sum \lambda_i \,\left(\nabla L^{-1/2}m_i\right)=: \sum \lambda_i \, \widetilde{m}_i,$$
where by the preceeding argument each $\widetilde{m}_i$ is a classical $H^p$-molecule,
and where the last sum also converges in $L^2$ (hence in $\mathcal{S}'$).
The bound (\ref{eq5.4}) then follows immediately by the molecular characterization of classical $H^p$.
This finishes the proof. \ep

\begin{proposition}\label{p5.3}
Let $1<r\leq 2$ be such that the family $\{e^{-tL}\}_{t>0}$
satisfies $L^r-L^2$ off-diagonal estimates.  Then for every $p\leq
1$ such that $p> \frac{rn}{n+r}$
\begin{equation}\label{eq5.17}
\|h\|_{H^p_L(\RR^n)}\leq C\|\nabla L^{-1/2}h\|_{H^p(\RR^n)}
\end{equation}
for every $h\in L^2(\RR^n)\cap H^p_{L,Riesz}(\RR^n)$. In particular, if
$\frac{p_-(L)n}{n+p_-(L)}<1$, then (\ref{eq5.17}) holds for every $p$
such that $\frac{p_-(L)n}{n+p_-(L)}<p\leq 1$.
\end{proposition}

\noindent{\it Remark:}  Combining Propositions \ref{p5.2} and \ref{p5.3}, we therefore
obtain (\ref{eq1.16}), for $f\in L^2(\RR^n)$, and thus
by density, we obtain  (\ref{eq1.rieszchar}) in the case $p\leq 1$.

\medskip

\bp Let $h\in H^p_{L,Riesz}(\RR^n)\cap L^2(\RR^n)$, and set
$$f:=L^{-1/2}h.$$  
Since, in particular, $h\in L^2(\RR^n),$ we have that $f$ is well defined:
indeed, the solution of the Kato square root problem \cite{KatoMain} (cf. (\ref{eq1.4})),
implies that $f \in \dot{W}^{1,2}(\RR^n)$ (cf. (\ref{eq5.36*})). 
%We also note that by definition of $H^p_{L,Riesz}(\RR^n)$, the
%right-hand side of (\ref{eq5.17}) is finite.  

Let us denote
\begin{equation}\label{eq5.18}
S_1h(x):=\left(\dint_{\Gamma(x)}|t\sqrt Le^{-t^2L}h(y)|^2\,\frac
{dydt}{t^{n+1}}\right)^{1/2},\qquad x\in\RR^n.
\end{equation}

\noindent Then by Corollary~\ref{c4.4}
\begin{equation}\label{eq5.19}
\|S_1h\|_{L^p(\RR^n)}\approx
\|h\|_{H^p_L(\RR^n)},\qquad  0<p\leq 2.\Bk
\end{equation}
Hence, 
matters are reduced to proving the estimate
\begin{equation}\label{eq5.20} \|S_1\sqrt L
f\|_{L^{p}(\RR^n)}\leq C\|\nabla f\|_{H^p(\RR^n)},\quad \quad
\frac{rn}{n+r}< p\leq 1.
\end{equation}

Let us recall the ``Hardy-Sobolev" spaces
\begin{equation}\label{eq5.21}
H^{1,p}(\RR^n):=\{f\in {\mathcal S}'(\RR^n)/{\mathbb C}:\,\nabla f\in
H^p(\RR^n)\},
\end{equation}
where ${\mathcal S}'(\RR^n)/{\mathbb C}$ is the space of tempered
distributions modulo constants. The space
$H^{1,p}(\RR^n)$ may be identified with the corresponding Triebel-Lizorkin spaces (see, e.g., \cite{MeMi} or Section~\ref{s8.2} of the current paper),  and thus admits an atomic decomposition \cite{FrJa}.
Specifically, a
function $a$ satisfying
\begin{equation}\label{eq5.22}
{\rm supp}\,a\subset Q, \quad\|\nabla a\|_{L^2(\RR^n)}\leq
l(Q)^{n/2-n/p},
\end{equation}

\noindent  is called an $H^{1,p}$-atom, $n/(n+1)<p\leq 1$
(as  usual, \Bk for smaller $p$ one has to impose an extra
vanishing moment condition). Then
\begin{equation}\label{eq5.23}
H^{1,p}(\RR^n)=\left\{\sum_{i=0}^\infty
\lambda_ja_j:\,\{\lambda_j\}_{j=0}^\infty\in \ell^p \mbox{ and
}\,a_j \mbox{ are $H^{1,p}$-atoms}\right\},
\end{equation}

\noindent with the series understood in the sense of convergence
in $S'(\RR^n)/{\mathbb C}$, and
\begin{multline*}
\|f\|_{H^{1,p}(\RR^n)}\approx \\[4pt]
\inf\Bigl\{\Bigl(\sum_{j=0}^\infty|\lambda_j|^p\Bigr)^{1/p}:
\,\,f=\sum_{j=0}^\infty\lambda_j m_j,\,\,
\{\lambda_j\}_{j=0}^\infty \in\ell^p \mbox{ and }a_j\textrm{ are
$H^{1,p}$-atoms} \Bigr\}.
\end{multline*}

\noindent We now  claim that it is enough to show that
\begin{equation}\label{eq5.24} \|S_1\!\sqrt L\,
a\|_{L^{p}(\RR^n)}\leq C,\quad {\mbox{for every $H^{1,p}$-atom
$a$}}, \quad p>\frac{rn}{n+r},\quad p\leq 1,
\end{equation}

\noindent where $C$ is a constant not depending on $a$. 
To see that (\ref{eq5.24}) suffices to obtain the conclusion of the Proposition, we proceed as follows.
We note that in the standard constructive tent space proof of the atomic decomposition of
$H^{1,p}$, one obtains, much as in the proof of Step 2 of Theorem \ref{th4.1} above, that for
$f$ in the dense subspace $\dot{W}^{1,2}(\RR^n)\cap H^{1,p}(\RR^n),$
there is a decomposition $f =\sum \lambda_j a_j$, converging in $\dot{W}^{1,2}(\RR^n)$,
where each $a_j$ is an
$H^{1,p}$-atom, and where
$$\sum|\lambda_j|^p \lesssim \|\nabla f\|^p_{H^{p}(\RR^n)}.$$
By the solution of the Kato square root problem \cite{KatoMain} (cf. (\ref{eq1.4})), and the $L^2$ boundedness of the square function $S_1$, we have that
$$S_1\!\sqrt{L}: \dot{W}^{1,2}(\RR^n) \to L^2 (\RR^n),$$ so 
using the $\dot{W}^{1,2}$ convergence of the
atomic sum, we obtain that pointwise $a.e.,$
$$S_1\sqrt{L}\,f \leq \sum |\lambda_j| \,\,S_1\!\sqrt{L}\,a_j\,.$$
Thus, (\ref{eq5.24}) implies (\ref{eq5.20}).

It remains to prove (\ref{eq5.24}). For $j\in\NN\cup\{0\}$ let
${\mathcal R}(S_j(Q)):=\bigcup_{x\in S_j(Q)}\Gamma(x)$ be a
saw-tooth region based on $S_j(Q)\subset\RR^n$. Then
\begin{eqnarray}\label{eq5.26}
&&\|S_1\sqrt L a\|^p_{L^{p}(\RR^n)}\leq \sum_{j=0}^\infty
(2^jl(Q))^{n\left(1-\frac
p2\right)}\left(\int_{S_j(Q)}\dint_{\Gamma(x)}|tLe^{-t^2L}a(y)|^2\frac{dydt}{t^{n+1}}\,dx\right)^{\frac
p2}\nonumber\\[4pt]
&&\qquad \leq C \sum_{j=3}^\infty (2^jl(Q))^{n\left(1-\frac
p2\right)}\left(\dint_{{\mathcal
R}(S_j(Q))}|tLe^{-t^2L}a(y)|^2\frac{dydt}{t}\right)^{\frac
p2}\nonumber\\[4pt]
&&\qquad\quad + \,C l(Q)^{n\left(1-\frac
p2\right)}\|S_1\sqrt La\|_{L^2(4Q)}^p \nonumber\\[4pt]
&&\qquad \leq C \sum_{j=3}^\infty (2^jl(Q))^{n\left(1-\frac
p2\right)}\left(\int_{\RR^n\setminus
2^{j-2}Q}\int_0^\infty|t^2Le^{-t^2L}a(y)|^2\frac{dydt}{t^3}\right)^{\frac
p2}\nonumber\\[4pt]
&&\qquad\quad  + \,C \sum_{j=3}^\infty (2^jl(Q))^{n\left(1-\frac
p2\right)}\left(\int_{2^{j-2}Q}\int_{c2^jl(Q)}^\infty|t^2Le^{-t^2L}a(y)|^2\frac{dydt}{t^3}\right)^{\frac
p2}\nonumber\\[4pt]
&&\qquad\quad  + \,C l(Q)^{n\left(1-\frac p2\right)}\|S_1\sqrt
La\|_{L^2(4Q)}^p=: I+II+III.
\end{eqnarray}

\noindent Then, since $S_1$ is bounded in $L^2(\RR^n)$,
\begin{equation}\label{eq5.27} III\leq C l(Q)^{n\left(1-\frac
p2\right)}\|\sqrt La\|^p_{L^2(\RR^n)}\leq C l(Q)^{n\left(1-\frac
p2\right)}\|\nabla a\|^p_{L^2(\RR^n)} \leq C.
\end{equation}

\noindent Going further, observe that
\begin{equation}\label{eq5.28}
\|a\|_{L^2(Q)}\leq l(Q)^{n/2-n/p+1},
\end{equation}

\noindent for every $H^{1,p}$ - atom $a$ by (\ref{eq5.22}),
Sobolev inequality and H\"older inequality. Then Lemma~\ref{l2.6},
(\ref{eq5.28}) and another application of H\"older inequality
imply that
\begin{eqnarray}\label{eq5.29} II&\leq & C \sum_{j=3}^\infty (2^jl(Q))^{n\left(1-\frac
p2\right)}\left(\int_{c2^jl(Q)}^\infty t^{2\left(\frac n2-\frac
nr\right)} \frac{dt}{t^3}\right)^{\frac
p2}\|a\|_{L^r(Q)}^p\nonumber\\[4pt]
&\leq & C \sum_{j=3}^\infty (2^jl(Q))^{n\left(1-\frac p2\right)}
(2^jl(Q))^{p\left(\frac n2-\frac nr-1\right)} \|a\|_{L^r(Q)}^p\leq
C,
\end{eqnarray}

\noindent provided $p>\frac{rn}{n+r}$.

Finally, in order to handle $I$, we split the integral in $t$ into
two parts, corresponding to $0<t<2^jl(Q)$ and $t\geq 2^jl(Q)$,
respectively. The second part can be estimated closely following
the argument in (\ref{eq5.29}). As for the first one,
\begin{eqnarray}\label{eq5.30} &&\sum_{j=3}^\infty (2^jl(Q))^{n\left(1-\frac
p2\right)}\left(\int_{\RR^n\setminus 2^{j-2}Q}\int_0^{2^jl(Q)}
|t^2Le^{-t^2L}a(y)|^2\frac{dydt}{t^3}\right)^{\frac
p2}\nonumber\\[4pt]
&&\qquad \leq C\sum_{j=3}^\infty (2^jl(Q))^{n\left(1-\frac
p2\right)}\left(\int_0^{2^jl(Q)} t^{2\left(\frac n2-\frac
nr\right)}\,e^{-\frac{(2^jl(Q))^2}{ct^2}}
\frac{dt}{t^3}\right)^{\frac p2}\,\|a\|_{L^r(\RR^n)}^p\leq C,
\end{eqnarray}

\noindent using $L^r-L^2$ off-diagonal estimates.  This completes the proof.\ep

Now we turn to the case $p>1$.

\begin{proposition}\label{p5.4}
The Riesz transform of the operator $L$ satisfies
\begin{equation}\label{eq5.31}
\nabla L^{-1/2}:H^p_L(\RR^n)\longrightarrow L^p(\RR^n)\qquad
\mbox{for}\qquad 1<p<2+\eps(L).
\end{equation}
\end{proposition}

\bp Since $p_+(L)\geq 2+\eps(L)$ (see \cite{AuscherSurvey}, Theorem~4.1 combined with \S{3.4}), the property (\ref{eq5.1}) and (\ref{eq1.13}) (proved in Proposition~\ref{embed} below) yield (\ref{eq5.31}) for
$p_-(L)<p<2+\eps(L)$. Then the full range of $p$ in (\ref{eq5.31}) can
be achieved by interpolation (Lemma~\ref{l4.5}) with the result
of Proposition~\ref{p5.2}. \ep

\begin{proposition}\label{p5.5}
Let $1<r\leq 2$ be such that the family $\{e^{-tL}\}_{t>0}$
satisfies $L^r-L^2$ off-diagonal estimates.  Then for all $p$ satisfying
$\max\left\{1,\frac{rn}{n+r}\right\}<p<p_+(L)$,
\begin{equation}\label{eq5.32}
\|h\|_{H^p_L(\RR^n)}\leq C\|\nabla L^{-1/2}h\|_{L^p(\RR^n)},
\end{equation}

\noindent for every $h\in L^2(\RR^n)\cap H^p_{L,Riesz}(\RR^n)$.

In particular, (\ref{eq5.32}) holds for every $p$ such that
$\max\left\{1,\frac{p_-(L)n}{n+p_-(L)}\right\}<p<2+\eps(L)$.
\end{proposition}

\noindent {\it Remark.} This Proposition is a sharpened version of
\cite{AuscherSurvey}, Proposition~4.10:  in the latter, the left hand side of
(\ref{eq5.32}) is replaced by the $L^p$ norm.  Our proof is based on the 
circle of ideas developed in \cite{AuscherSurvey}, 
but the estimates we seek
are somewhat more delicate, since $H^p_L$ is ``strictly smaller" than $L^p$
(in the sense of Proposition \ref{embed} (ii) below)
in the range $1<p\leq p_-(L) $. 

\vskip 0.08in

\bp {\bf Step I.}  By (\ref{eq5.1}) applied to $L^*$, and a standard duality argument
we deduce that 
\begin{equation}\nonumber
\|\sqrt L \,g\|_{L^{p'}(\RR^n)}\leq C\|\nabla g\|_{L^{p'}(\RR^n)}, \quad
\frac 1p+\frac 1{p'}=1,
\end{equation}

\noindent for $p_-(L^*)<p<2+\eps(L^*)$, and hence, 
using the fact that $\left(p_-(L^*)\right)' = p_+(L)$, we have
\begin{equation}\label{eq5.33}
\|h\|_{L^p(\RR^n)}\leq C\|\nabla L^{-1/2}h\|_{L^p(\RR^n)}, \quad
2\leq p<p_+(L))\,,
\end{equation}

\noindent in which range of $p$ we have $\HPL = L^p(\RR^n)$ (cf. Appendix, Section \ref{s9}).
Therefore we may suppose that $p<2$. 

We claim that it is enough to show that for each $r$ as above,
\begin{equation}\label{eq5.34}
S_1\sqrt L: \dot W^{1,p}(\RR^n)\longrightarrow
L^{p,\infty}(\RR^n),\quad \quad
p=p(n,r):=\max\left\{1,\frac{rn}{n+r}\right\},
\end{equation}

\noindent 
because, given (\ref{eq5.34}), the desired estimate (\ref{eq5.32}), for the
range $p(n,r)<p<2$, follows by interpolation with (\ref{eq5.33}).
More precisely, setting $f:= L^{-1/2}h$, by 
(\ref{eq5.33}) and the boundedness of $S_1$ in $L^2$,
we have, in particular, that
\begin{equation}\label{eq5.37}
\|S_1\sqrt L f\|_{L^{2}(\RR^n)}\leq C \|f\|_{\dot{W}^{1,2}(\RR^n)}
\,.
\end{equation}
Thus, interpolating between the latter estimate and (\ref{eq5.34}), we obtain
\begin{equation}\label{eq5.38}
S_1\sqrt L: \dot W^{1,p}(\RR^n)\longrightarrow L^{p}(\RR^n)\quad
\mbox{ whenever
}\quad\max\left\{1,\frac{rn}{n+r}\right\}<p<2,
\end{equation}

\noindent and this is equivalent to $(\ref{eq5.32})$, in the remaining case $p(n,r)<p<2$.

Hence, it remains only to prove (\ref{eq5.34}), i.e., we shall show that
\begin{equation}\label{eq5.39}
\left|\left\{x\in\RR^n:\,S_1\sqrt L f(x)>\alpha\right\}\right|\leq
\frac {C}{\alpha^p}\int_{\RR^n}|\nabla
f(y)|^p\,dy,\qquad\forall\,\alpha>0,
\end{equation}

\noindent for $p$ as in (\ref{eq5.34}), where by density we may suppose that $f\in C^\infty_0$.

Our proof is based on the use of a ``Calder\'{o}n-Zygmund type" decomposition of Sobolev spaces taken from \cite{AuscherSurvey}, where it was used to establish an analogue of (\ref{eq5.39}), but for $\sqrt{L}$ rather than for $S_1 \sqrt{L}$.

\begin{lemma}\label{l5.6} (\cite{AuscherSurvey}) Suppose $n\geq 1$, $1\leq
p< \infty$ and $f\in \dot{W}^{1,p}(\RR^n)$. Then for every $\alpha>0$ there exists a
collection of cubes $\{Q_i\}_{i\in\ZZ}$ with finite overlap,  a
\Bk function $g$ and a family of functions $\{b_i\}_{i\in\ZZ}$
satisfying
\begin{eqnarray}\label{eq5.40}
&&  {\rm
supp}\,b_i\subset Q_i, \qquad \|\nabla b_i\|_{L^p(\RR^n)}\leq
C\alpha |Q_i|^{1/p}, \qquad \forall\,i\in\ZZ,\\[4pt]\label{eq5.41}
&& \|\nabla
g\|_{L^p(\RR^n)}\leq C \|\nabla f\|_{L^p(\RR^n)}, \quad \|\nabla g\|_{L^\infty(\RR^n)}\leq C\alpha,
\end{eqnarray}

\noindent such that $f$ can be represented in the form
\begin{equation}\label{eq5.42}
f=g+\sum_{i\in\ZZ}b_i,\quad\mbox{with}\quad
\sum_{i\in\ZZ}|Q_i|\leq C\alpha^{-p}\|\nabla f\|_{L^p(\RR^n)}^p.
\end{equation}

\end{lemma}

Returning to (\ref{eq5.39}) we can write using the lemma above
\begin{eqnarray}\nonumber
S_1\sqrt L f(x)&\leq &S_1\sqrt L
g(x)+\left(\dint_{\Gamma(x)}\Bigl|\sum_{i\in\ZZ} t Le^{-t^2L}
b_i(y)\chi_{(0,l(Q_i))}(t)\Bigr|^2\,\frac
{dydt}{t^{n+1}}\right)^{1/2}\\[4pt]
&&\quad +\left(\dint_{\Gamma(x)}\Bigl|\sum_{i\in\ZZ} tLe^{-t^2L}
b_i(y)\chi_{[l(Q_i), \infty)}(t)\Bigr|^2\,\frac
{dydt}{t^{n+1}}\right)^{1/2}\nonumber\\[4pt]\nonumber
&\leq &S_1\sqrt L
g(x)+\sum_{i\in\ZZ}\left(\dint_{\stackrel{t<l(Q_i)}{|x-y|<t}}|t
Le^{-t^2L}b_i(y)|^2\,\frac
{dydt}{t^{n+1}}\right)^{1/2}\\[4pt]
&&\quad
+\left(\dint_{\Gamma(x)}\Bigl|t^2Le^{-t^2L}\sum_{i\in\ZZ}\frac{b_i(y)}{l(Q_i)}\,\Bigr|^2\,\frac
{dydt}{t^{n+1}}\right)^{1/2}\nonumber\\[4pt]
& =: & I_0(x)+I_1(x)+I_2(x),\label{eq5.43}
\end{eqnarray}

\noindent for all $x\in\RR^n$. Let us assign now
\begin{eqnarray*}\nonumber
&& A_l:=\left\{x\in\RR^n:\,I_l(x)>\alpha/3\right\},\qquad l=0,1,2,
\end{eqnarray*}

\noindent so that
\begin{equation}\label{eq5.44}
\left|\left\{x\in\RR^n:\,S_1\sqrt L f(x)>\alpha\right\}\right|\leq
|A_0|+|A_1|+|A_2|.
\end{equation}

\vskip 0.08 in \noindent {\bf Step II.} Consider $A_0$. By
Chebyshev's inequality
\begin{equation}\label{eq5.45}
|A_0|\leq \frac{C}{\alpha^2}\int_{\RR^n}\left|S_1 \sqrt L
g(x)\right|^2\,dx\leq \frac{C}{\alpha^2}\int_{\RR^n}\left|\nabla
g(x)\right|^2\,dx
\end{equation}

\noindent where for the last estimate we used boundedness of $S_1$
in $L^2(\RR^n)$ and the Kato square root estimate (\cite{KatoMain}).
Combining the two statements in (\ref{eq5.41}), we obtain that the expression in (\ref{eq5.45}) is bounded by
$C\alpha^{-p} \|\nabla f\|_{L^p(\RR^n)}^p$, as desired.

\vskip 0.08 in \noindent {\bf Step III.} The contribution from
$A_2$ can be estimated as follows. By Chebyshev's inequality
\begin{eqnarray}\label{eq5.46}
|A_2|&\leq & \frac{C}{\alpha^{r}}\int_{\RR^n}\Bigl|S
\Bigl(\sum_{i\in\ZZ}\frac{b_i(y)}{l(Q_i)}\Bigr)\Bigr|^{r}\,dx,
\end{eqnarray}

\noindent with $S$ as in (\ref{eq1.8}). On the other hand, the
$L^{r}-L^2$ off-diagonal estimates for the heat semigroup imply
that $S$ is bounded in $L^{r}(\RR^n)$  (see, e.g.,
\cite{AuscherSurvey}, Theorem~6.1,  for an analogous result in the case of
vertical square function and \cite{HM}). Therefore, by H\"older's
inequality for sequences
\begin{equation}
|A_2|\leq  \frac{C}{\alpha^{r}}
\Bigg\|\sum_{i\in\ZZ}\frac{|b_i|}{l(Q_i)}\Bigg\|_{L^{r}(\RR^n)}^{r}
\leq \frac{C}{\alpha^{r}}
\Bigg\|\Bigl(\sum_{i\in\ZZ}\frac{|b_i|^{r}}{l(Q_i)^{r}}\Bigr)^{1/r}
\Bigl(\sum_{i\in\ZZ}\chi_{Q_i}\Bigr)^{1-1/r}\Bigg\|_{L^{r}(\RR^n)}^{r}.
\label{eq5.47}
\end{equation}

\noindent Now we recall that the cubes $\{Q_i\}_{i\in\ZZ}$ have
finite overlap, i.e. there exists some fixed constant $C$ such
that $\sum_{i\in\ZZ}\chi_{Q_i}(x)\leq C$ for all $x\in\RR^n$. This
implies that
\begin{equation}
|A_2| \leq \frac{C}{\alpha^{r}}
\int_{\RR^n}\sum_{i\in\ZZ}\frac{|b_i|^{r}}{l(Q_i)^{r}}\,dx.
\label{eq5.48}
\end{equation}

When $p=\frac{rn}{n+r}$ we deduce from (\ref{eq5.40}) and
Poincar\'e's inequality that
\begin{equation}\label{eq5.49}
\|b_i\|_{L^{r}(\RR^n)}\leq C \|\nabla b_i\|_{L^{p}(\RR^n)}\leq
C\alpha |Q_i|^{1/p}=C\alpha |l(Q_i)|^{1+n/r}.
\end{equation}

\noindent When $p=1>\frac{rn}{n+r}$, by H\"older's inequality
\begin{equation}\label{eq5.50}
\|b_i\|_{L^{r}(Q_i)}\leq C |Q_i|^{\frac 1r- \frac {n-1}{n}}
\|b_i\|_{L^{\frac{n}{n-1}}(Q_i)}\leq C |Q_i|^{\frac 1r- \frac
{n-1}{n}}\|\nabla b_i\|_{L^{1}(Q_i)}\leq C\alpha |l(Q_i)|^{1+n/r}.
\end{equation}

\noindent Hence, in any case,
\begin{equation}\label{eq5.51}
|A_2| \leq C\sum_{i\in\ZZ}|Q_i| \leq C\alpha^{-p}\|\nabla
f\|_{L^p(\RR^n)}^p.
\end{equation}

\vskip 0.08 in \noindent {\bf Step 4.} We now proceed to estimate
$|A_1|$. The argument here resonates with that in \cite{AuscherCoulhon}, Section~1.2. 
% Denote
 For each function $v$, define
\begin{equation}\label{eq5.52}
T_iv(x):= \left(\dint_{\stackrel{t<l(Q_i)}{|x-y|<t}}|t
Le^{-t^2L}v(y)|^2\,\frac {dydt}{t^{n+1}}\right)^{1/2}, \quad
i\in\ZZ, \quad x\in\RR^n.
\end{equation} \Bk

\noindent Then
\begin{eqnarray}\nonumber
|A_1|&\leq& \sum_{i\in\ZZ}|4Q_i|+ \left\{x\in\RR^n\setminus
\cup_{i\in\ZZ} 4Q_i:
\,\left|\sum_{i\in\ZZ}T_ib_i(x)\right|>\alpha/3\right\}\\[4pt]
&\leq & \frac{C}{\alpha^{p}}\,\|\nabla f\|_{L^p(\RR^n)}^p +
\frac{C}{\alpha^2}\int_{\RR^n}
\Bigl|\sum_{i\in\ZZ}T_ib_i(x)\chi_{\RR^n\setminus
4Q_i}(x)\Bigr|^2\,dx. \label{eq5.53}
\end{eqnarray}

\noindent The second term above (referred to as $\widetilde T$
later on) is bounded by
\begin{equation}
\frac{C}{\alpha^2}\left(\sum_{i\in\ZZ}\int_{\RR^n\setminus 4Q_i}
T_ib_i(x)u(x)\,dx \right)^2, \label{eq5.54}
\end{equation}

\noindent for some $u\in L^2(\RR^n)$ such that
$\|u\|_{L^2(\RR^n)}=1$. Therefore,
\begin{eqnarray}\nonumber
\widetilde T &\leq &
\frac{C}{\alpha^2}\left(\sum_{i\in\ZZ}\sum_{j=3}^\infty\|
T_ib_i\|_{L^2(S_j(Q_i))}\|u\|_{L^2(S_j(Q_i))} \right)^2\\[4pt]
&\leq & \frac{C}{\alpha^2}\left(\sum_{i\in\ZZ}\sum_{j=3}^\infty
\left(\dint_{\stackrel{t<l(Q_i)}{(y,t)\in {\mathcal R}(S_j(Q_i))}}
|t^2Le^{-t^2L}b_i(y)|^2\,\frac{dydt}{t^3}\right)^{1/2}\|u\|_{L^2(S_j(Q_i))}
\right)^2, \label{eq5.55}
\end{eqnarray}

\noindent where, as before,  ${\mathcal R}(S_j(Q_i))=\cup_{x\in
S_j(Q_i)}\Gamma(x)$ stands for the saw-tooth region built on the
set $S_j(Q_i)$. Then, using Lemma~\ref{l2.6} and
(\ref{eq5.49})--(\ref{eq5.50}), we see that
\begin{eqnarray}\nonumber
\widetilde T &\leq &
\frac{C}{\alpha^2}\left(\sum_{i\in\ZZ}\sum_{j=3}^\infty
\left(\int_0^{l(Q_i)}
\|t^2Le^{-t^2L}b_i\|_{L^2(2^{j+1}Q_i\setminus 2^{j-2}Q_i)}^2\,
\frac{dt}{t^3}\right)^{1/2}\|u\|_{L^2(S_j(Q_i))} \right)^2\\[4pt]
\nonumber &\leq &
\frac{C}{\alpha^2}\left(\sum_{i\in\ZZ}\sum_{j=3}^\infty
\left(\int_0^{l(Q_i)}
e^{-\frac{(2^jl(Q_i))^2}{ct^2}}t^{2\left(\frac n 2
-\frac{n}{r}-1\right)} \|b_i\|_{L^{r}(Q_i)}^2\,
\frac{dt}{t}\right)^{1/2}\|u\|_{L^2(S_j(Q_i))} \right)^2\\[4pt]
&\leq & C\left(\sum_{i\in\ZZ}\sum_{j=3}^\infty 2^{-jN}l(Q_i)^n
\left[{\mathcal M}\left(|u|^2\right)(y)\right]^{1/2}\right)^2
\label{eq5.56}
\end{eqnarray}

\noindent for any $y\in Q_i$ and any large positive number $N$.
Here ${\mathcal M}$ stands for the Hardy-Littlewood maximal
function, i.e., 
\begin{equation}\nonumber
{\mathcal M} g(x)=\sup_{Q\ni x}\frac{1}{|Q|}\int_Q|g(y)|\,dy,\qquad x\in\RR^n.
\end{equation}

\noindent Next, one can sum up the expression above in $j$ and
integrate in $y$ to obtain
\begin{multline}\label{eq5.57}
\widetilde T \leq C\left(\int_{\RR^n}\sum_{i\in\ZZ}\chi_{Q_i}(y)
\left[{\mathcal M}\left(|u|^2\right)(y)\right]^{1/2}\,dy\right)^2 \\
\leq C \left(\int_{\bigcup_{i\in\ZZ}\,Q_i} \left[{\mathcal
M}\left(|u|^2\right)(y)\right]^{1/2}\,dy\right)^2,
\end{multline}

\noindent by the finite overlap property of cubes
$\{Q_i\}_{i\in\ZZ}$. At this point we use Kolmogorov's lemma. It amounts to the fact that every sublinear operator $T$ of weak type (1,1) satisfies the property 
\begin{equation}\nonumber
\int_E |Tf(x)|^q\,dx \leq C |E|^{1-q} \|f\|_{L^1(\RR^n)}^q, \quad\mbox{for all} \quad f\in L^1(\RR^n), \quad 0<q<1, 
\end{equation}

\noindent and any set $E$ of finite Lebesgue measure.
Then, using the 
weak type $(1,1)$ boundedness of the Hardy-Littlewood maximal
function we control the expression in \eqref{eq5.57} by
\begin{equation}\label{eq5.58}
C\left(\Bigl|\bigcup_{i\in\ZZ}
Q_i\Bigr|^{1/2}\left\||u|^2\right\|_{L^1(\RR^n)}^{1/2}\right)^2
\leq C \sum_{i\in\ZZ}|Q_i|\leq \frac{C}{\alpha^{p}}\,\|\nabla
f\|_{L^p(\RR^n)}^p,
\end{equation}

\noindent as desired.  This concludes the proof of Proposition \ref{p5.5}, and thus also that of Theorem
\ref{t5.1}.\ep

\section{Sharp maximal function characterization.}\label{s6}
\setcounter{equation}{0}

Recall the sharp maximal function introduced in (\ref{eq1.22}). This Section is devoted to the proof of (\ref{eq1.23}).   More precisely, we define
$H^p_{\sharp,M,L}(\RR^n)$ to be the completion of the set 
$$\mathbb{H}^p_{\sharp,M,L}(\RR^n):= \{f\in L^2(\RR^n): {\mathcal M}^{\sharp}_Mf \in L^p(\RR^n)\}\,,$$
with respect to the norm
$$\|f\|_{H^p_{\sharp,M,L}(\RR^n)}:= \|{\mathcal M}^{\sharp}_Mf\|_{L^p(\RR^n)}.$$
We have the following:

\begin{theorem}\label{t6.1} Let $2<p<\infty$ and $M>n/4$. 
Then 
$f\in H^p_L(\RR^n) = H^p_{\sharp,M,L}(\RR^n)$, and, for all $f\in L^2(\RR^n)$,
\begin{equation}\label{eq6.1}
\|f\|_{H^p_L(\RR^n)}\approx \|{\mathcal M}^{\sharp}_Mf\|_{L^p(\RR^n)}.
\end{equation}
\end{theorem}

\bp  Recall that we have shown in the proof of Corollary \ref{c4.4}
that $L^2(\RR^n) \cap \HPL$ is dense in $\HPL$ when $2<p<\infty$ (for $p\leq 2$, the analogous
density statement holds by definition).   Consequently,
it  suffices to establish (\ref{eq6.1}).

\smallskip

\noindent {\bf Step I.}\,
First, we shall establish that for all $M\in\NN$
\begin{equation}\label{eq6.2}
{\mathcal M}^{\sharp}_M:L^p(\RR^n)\longrightarrow L^p(\RR^n),\quad\mbox{for}\quad 2<p\leq\infty,
\quad M\in\NN.
\end{equation}

\noindent Clearly, the latter estimate is an immediate consequence of the pointwise bound
${\mathcal M}^{\sharp}_Mf \leq C\left({\mathcal M}(|f|^2)\right)^{1/2}$,
where ${\mathcal M}$ denotes the Hardy-Littlewood maximal operator.  In turn, we establish
the pointwise bound as follows:
\begin{eqnarray}
{\mathcal M}^{\sharp}_Mf \nonumber  & \!\leq \!&
\sup_{Q\ni \,x}\sum_{j=0}^\infty\left(\frac{1}{|Q|}\int_Q
\left|(I-e^{-l(Q)^2{L}})^M(f\chi_{S_j(Q)})(y)\right|^2\,dy\right)^{1/2}\nonumber \\[4pt] && \leq \,\,
C_M\sup_{Q\ni \,x}\left(\frac{1}{|Q|}\int_{Q}
\left|f(y)\right|^2\,dy \right)^{1/2} \nonumber \\[4pt] &&\qquad\qquad + \,\,C_M
\sup_{Q\ni \,x}\sup_{1\leq k\leq M}\sum_{j=1}^\infty\left(\frac{1}{|Q|}\int_Q
\left|e^{-kl(Q)^2{L}}(f\chi_{S_j(Q)})(y)\right|^2\,dy\right)^{1/2}.
\label{eq6.3}
\end{eqnarray}

\noindent  Then
by the Gaffney estimate (Lemma~\ref{l2.5}), the expression above is controlled by
\begin{multline}
C \left({\mathcal M}(|f|^2)\right)^{1/2}\,+\,
C\sup_{Q\ni \,x}\sup_{1\leq k\leq M}\sum_{j=1}^\infty e^{-\frac{{\rm
dist}\,(Q,S_j(Q))^2}{cl(Q)^2}}\frac{1}{|Q|^{1/2}}\|f\|_{L^2(S_j(Q))}
\\[4pt] 
\leq C \left({\mathcal M}(|f|^2)\right)^{1/2}.
\label{eq6.4}
\end{multline}
This finishes the proof of (\ref{eq6.2}).  Since for $2\leq p<p_+(L)$,  
the spaces $H^p_L$ coincide with $L^p$,  we also have
\begin{equation}\label{eq6.5}
{\mathcal M}^{\sharp}_M:H^p_L(\RR^n)\longrightarrow L^p(\RR^n)\quad\mbox{for}\quad 2<p<p_+(L),\quad M\in\NN.
\end{equation}

\noindent  Interpolating (\ref{eq6.5}) with the property
\begin{equation}\label{eq6.6}
{\mathcal M}^{\sharp}_M:BMO_L(\RR^n)\longrightarrow L^\infty(\RR^n),\quad M>n/4,
\end{equation}

\noindent we deduce that
\begin{equation}\label{eq6.7}
{\mathcal M}^{\sharp}_M:H^p_L(\RR^n)\longrightarrow L^p(\RR^n)\quad\mbox{for}\quad 2<p<\infty,\quad M>n/4.
\end{equation}

\vskip 0.08in \noindent{\bf Step II.} Now we turn to the converse of (\ref{eq6.7}). More precisely, let us show that
\begin{equation}\label{eq6.8}
\|f\|_{H^p_L(\RR^n)}\leq C \|{\mathcal M}^{\sharp}_Mf\|_{L^p(\RR^n)},
\end{equation}

\noindent whenever $2<p<\infty$, $M>n/4$ and
$f \in L^2(\RR^n).$  Note that for such $f$, the adapted sharp function ${\mathcal M}^{\sharp}_Mf$
is well-defined.

Recall the discussion of tent spaces in Section~\ref{s3}. 
In particular, by (\ref{eq4.tentp>2}) and Corollary~\ref{c4.4}, we have, for each $M>n/4$ and 
every $2<p<\infty$,
\begin{equation}\label{eq6.13}
\|f\|_{H^p_L(\RR^n)}\leq C_{M,p} \,\Bigg\|\sup_{B\ni x}\left(\frac{1}{|B|}\dint_{\widehat
B}|(t^2L)^Me^{-t^2L}f(y)|^2\,\frac{dydt}{t}\right)^{1/2}\Bigg\|_{L^p(\RR^n)}.
\end{equation}
Thus, in order to conclude (\ref{eq6.8}) it suffices to show that, for $M>n/4$,
\begin{equation}\label{eq6.14} \Bigg\|\sup_{Q\ni x}\left(\frac{1}{|Q|}\int_0^{l(Q)}\int_{Q}|(t^2L)^{M+1}e^{-t^2L}f(x)|^2\,\frac{dxdt}{t}\right)^{1/2}\Bigg\|_{L^p(\RR^n)}\leq C
\|{\mathcal M}^{\sharp}_M f\|_{L^p(\RR^n)},\end{equation}

\noindent for $2<p<\infty$. Note that we have replaced the exponent 
$M$ by $M+1$ on the left-hand side
of \eqref{eq6.14}, but this is harmless:
since \eqref{eq6.13} holds for {\it every}  $M>n/4$, we may choose it larger at our convenience.

\vskip 0.08in \noindent{\bf Step III.}
In this part we prove that
for every cube $Q\subset\RR^n$
\begin{eqnarray}
I_Q&:=&\left(\frac{1}{|Q|}\int_0^{l(Q)}\int_{Q}
\left|(t^2L)^{M+1}e^{-t^2L}f(y)\right|^2\,\frac{dydt}{t}\right)^{1/2}\nonumber \\[4pt]
&\leq & C\sum_{j=0}^\infty 2^{-jN} \frac{1}{|2^{j}Q|^{1/2}}\,\, \sup_{l(Q)\leq \,s\,\leq \sqrt2 l(Q)}
\|(I-e^{-s^2L})^M f\|_{L^2(2^jQ)},\quad \forall\,N\in\NN.\label{eq6.15}
\end{eqnarray}

 Following the procedure  outlined in (\ref{eq3.5})--(\ref{eq3.10}) one can split
\begin{eqnarray}\label{eq6.16}
&&\hskip -0.7cm f  = {2^M}\Bigg(l(Q)^{-2}\int_{l(Q)}^{\sqrt 2
l(Q)}s(I-e^{-s^2L})^M\,ds \nonumber\\[4pt]
 &&\hskip -0.7cm\qquad + \sum_{k=1}^M C_{k,M} l(Q)^{-2}L^{-1}
e^{-kl(Q)^2L}(I-e^{-l(Q)^2L})\sum_{i=0}^{k-1}
e^{-il(Q)^2L} \Bigg)^Mf\nonumber\\[4pt]
 &&\hskip -0.7cm\quad
= C_{1,1}
T_{1,1}^{l(Q)}(I-e^{-l(Q)^2L})^M l(Q)^{-2M}L^{-M} f\nonumber\\[4pt]
 &&\hskip -0.7cm\qquad +
\sum_{i=1}^{(M+1)^M-1} C_{i,2}
T_{i,2}^{l(Q)}\left(l(Q)^{-2}\int_{l(Q)}^{\sqrt 2
l(Q)}s(I-e^{-s^2L})^Ml(Q)^{-2N_i}L^{-N_i} f\,ds\right),
\end{eqnarray}
\noindent where $C_{i,k}$ are some constants, $0\leq N_i\leq M$, and each  $T_{i,k}$ is given by a constant (possibly, zero) plus a linear combination of the terms in the form $e^{-t^2L}$ with $t\approx l(Q)^2$. In particular, $T_{i,k}$'s are bounded in $L^2(\RR^n)$
with the constant independent of $l(Q)$ (see Lemma~\ref{l2.6})) and
satisfy Gaffney estimates (\ref{eq2.19}) with $t\approx l(Q)^2$.

All the terms on the right-hand side of (\ref{eq6.16}) are essentially of the same nature, and will be handled similarly. Let us  concentrate on the first one.
The corresponding part of $I_Q$ is bounded by
\begin{multline}\label{eq6.17}\\[4pt]\!\!\!\!\!\!\!\!\!\!\!\!\!\!
\sum_{j=0}^\infty
\left(\frac{1}{|Q|}\int_0^{l(Q)}\int_{Q}
\left|\Bigl(\frac{t}{l(Q)}\Bigr)^{2M}t^2Le^{-t^2L}T_{1,1}^{l(Q)}\bigl[\chi_{S_j(Q)}(I-e^{-l(Q)^2L})^M
f\bigr](y)\right|^2\,\frac{dydt}{t}\right)^{1/2}.
\end{multline}

\noindent Since  the mapping
\begin{equation}\label{eq6.18}
f \mapsto \left(\int_0^\infty|t^2Le^{-t^2L}f(\cdot)|^2\, \frac
{dt}{t}\right)^{1/2}, 
\end{equation}

\noindent is bounded in $L^2(\RR^n)$ (a consequence of the $H^\infty$ calculus for $L$, see \cite{ADM}), and the operator $T_{1,1}^{l(Q)}$ is bounded in $L^2$, we can write
\begin{multline}\\\!\!\!\!\!\!\!\!\!\!\!\!\!\!\!
\sum_{j=0,1}\left(\frac{1}{|Q|}\int_0^{l(Q)}\int_{Q}
\left|\Bigl(\frac{t}{l(Q)}\Bigr)^{2M}t^2Le^{-t^2L}T_{1,1}^{l(Q)}\bigl[\chi_{S_j(Q)}(I-e^{-l(Q)^2L})^M
f\bigr](y)\right|^2\,\frac{dydt}{t}\right)^{1/2}\\[4pt]
\leq C\frac{1}{|Q|^{1/2}}\,\|(I-e^{-l(Q)^2L})^M f\|_{L^2(2Q)}.\label{eq6.19}
\end{multline}

\noindent Furthermore, by Gaffney estimates and Lemma~\ref{l2.4}, when $j\geq 2$ we have
\begin{eqnarray}
&& \left(\frac{1}{|Q|}\int_0^{l(Q)}\int_{Q}
\left|\Bigl(\frac{t}{l(Q)}\Bigr)^{2M}t^2Le^{-t^2L}T_{1,1}^{l(Q)}\bigl[\chi_{S_j(Q)}(I-e^{-l(Q)^2L})^M
f\bigr](y)\right|^2\,\frac{dydt}{t}\right)^{1/2}\nonumber\\[4pt]
&& \qquad \leq
\frac{C}{|Q|^{1/2}}\left(\int_0^{l(Q)}\,e^{-\frac{(2^jl(Q))^2}{cl(Q)^2}}\,\Bigl(\frac{t}{l(Q)}\Bigr)^{2M}\,\frac{dt}{t}\right)^{1/2}
\, \|(I-e^{-l(Q)^2L})^M f\|_{L^2(S_j(Q))}\nonumber\\[4pt]
&& \qquad \leq C2^{-jN}\,
\frac{1}{|2^jQ|^{1/2}}
\, \|(I-e^{-l(Q)^2L})^M f\|_{L^2(S_j(Q))},\label{eq6.20}
\end{eqnarray}

\noindent for any $N\in\NN$. Now the combination of (\ref{eq6.19}) and (\ref{eq6.20}), together with analogous considerations for the remaining terms in (\ref{eq6.16}), implies
\begin{eqnarray}
 I_Q &\leq &  C\sum_{j=0}^\infty 2^{-jN}\,
\frac{1}{|2^jQ|^{1/2}}
\, \Bigg(\|(I-e^{-l(Q)^2L})^M f\|_{L^2(S_j(Q))}\nonumber\\[4pt]&&\qquad + \Bigg\|l(Q)^{-2}\int_{l(Q)}^{\sqrt 2
l(Q)}s(I-e^{-s^2L})^M f\,ds\Bigg\|_{L^2(S_j(Q))}\Bigg)\nonumber\\[4pt]&\leq &  C\sum_{j=0}^\infty 2^{-jN}\,
\frac{1}{|2^jQ|^{1/2}}
\, \Bigg(\|(I-e^{-l(Q)^2L})^M f\|_{L^2(S_j(Q))}\nonumber\\[4pt]&&\qquad + l(Q)^{-2}\int_{l(Q)}^{\sqrt 2
l(Q)}s \|(I-e^{-s^2L})^M f\|_{L^2(S_j(Q))}\,ds\Bigg)\nonumber\\[4pt]&
\leq &C\sum_{j=0}^\infty 2^{-jN}\,
\frac{1}{|2^jQ|^{1/2}}
\, \sup_{l(Q)\leq \,s\,\leq \sqrt2 l(Q)}\|(I-e^{-s^2L})^M f\|_{L^2(S_j(Q))}, \label{eq6.21}
\end{eqnarray}
\noindent as desired.

\vskip 0.08in \noindent{\bf Step IV.} The next step is to show that
\begin{eqnarray}
&&\sup_{Q\ni x}\sum_{j=0}^\infty 2^{-jN}\,
\frac{1}{|2^jQ|^{1/2}}
\, \sup_{l(Q)\leq \,s\,\leq \sqrt2 l(Q)}\|(I-e^{-s^2L})^M f\|_{L^2(S_j(Q))} \nonumber\\[4pt]&&\qquad \leq  C {\mathcal M}_2({\mathcal M}^{\sharp}_Mf)(x),\qquad x\in\RR^n,\qquad M\in\NN,\label{eq6.22}
\end{eqnarray}

\noindent where ${\mathcal M}_2$ is an $L^2$-based version of the Hardy-Littlewood maximal function, i.e.
\begin{equation}\label{eq6.23}
{\mathcal M}_2 g(x)=\sup_{Q\ni x}\left(\frac{1}{|Q|}\int_Q|g(y)|^2\,dy\right)^{1/2},\qquad x\in\RR^n.
\end{equation}

Clearly,
\begin{equation}\label{eq6.24}
{\mathcal M}_2 g(x)=\sup_{Q\ni x}\sup_{j\in\NN\cup\{0\}}\left(\frac{1}{|2^jQ|}\int_{2^jQ}|g(y)|^2\,dy\right)^{1/2},\qquad x\in\RR^n.
\end{equation}

\noindent Hence,
\begin{eqnarray}
&&{\mathcal M}_2({\mathcal M}^{\sharp}_Mf)(x)\nonumber\\[4pt]
&&\qquad = \sup_{Q\ni x}\sup_{j\in\NN\cup\{0\}}\left(\frac{1}{|2^jQ|}\int_{2^jQ}\sup_{\widetilde Q\ni y} \frac{1}{|\widetilde Q|} \int_{\widetilde Q}|(I-e^{-l(\widetilde Q)^2L})^M f(z)|^2\,dz\,dy\right)^{1/2}.
\label{eq6.25}
\end{eqnarray}

\noindent Let us denote by $\{Q_i^j\}_{i=1}^{2^{jn}}$ some partition of $2^jQ$ into  subcubes of sidelength $l(Q)$. Then the expression above is further equal to
\begin{eqnarray}\label{eq6.26}
&&\sup_{Q\ni x}\sup_{j\in\NN\cup\{0\}}\left(\frac{1}{|2^jQ|}\sum_{i=1}^{2^{jn}}\int_{Q_i^j}\sup_{\widetilde Q\ni y} \frac{1}{|\widetilde Q|} \int_{\widetilde Q}|(I-e^{-l(\widetilde Q)^2L})^M f(z)|^2\,dz\,dy\right)^{1/2}\nonumber\\[4pt]
&&\quad \geq \sup_{Q\ni x}\sup_{j\in\NN\cup\{0\}}\left(\frac{1}{|2^jQ|}\sum_{i=1}^{2^{jn}}\int_{Q_i^j}\frac{1}{|Q_i^j|} \int_{Q_i^j}(I-e^{-l(Q)^2L})^M f(z)|^2\,dz\,dy\right)^{1/2}\nonumber\\[4pt]
&&\quad =\sup_{Q\ni x}\sup_{j\in\NN\cup\{0\}}\left(\frac{1}{|2^jQ|}\sum_{i=1}^{2^{jn}}\int_{Q_i^j}(I-e^{-l(Q)^2L})^M f(z)|^2\,dz\right)^{1/2}\nonumber\\[4pt]
&&\quad =\sup_{Q\ni x}\sup_{j\in\NN\cup\{0\}}\left(\frac{1}{|2^jQ|}\int_{2^jQ}(I-e^{-l(Q)^2L})^M f(z)|^2\,dz\right)^{1/2},
\end{eqnarray}

\noindent where we used the fact that $l(Q_i^j)$=$l(Q)$ for all $i=1,...,2^{jn}$, $j\in\NN\cup\{0\}$, to switch from $e^{-l(Q_i^j)^2L}$ to $e^{-l(Q)^2L}$ in the first inequality above. We claim that the expression in the last line of (\ref{eq6.26}) controls the left-hand side of (\ref{eq6.22}). Indeed,
\begin{eqnarray}\label{eq6.27}
&&\sup_{Q\ni x}\sum_{j=0}^\infty 2^{-jN}\,
\frac{1}{|2^jQ|^{1/2}}
\, \sup_{l(Q)\leq \,s\,\leq \sqrt2 l(Q)}\|(I-e^{-s^2L})^M f\|_{L^2(S_j(Q))} \nonumber\\[4pt]
&&\quad \leq C \sup_{Q\ni x}\sum_{j=0}^\infty 2^{-jN}\,
\sup_{l(Q)\leq \,s\,\leq \sqrt2 l(Q)}\Bigg(\frac{1}{|2^jQ_s|}\int_{2^jQ_s}|(I-e^{-s^2L})^M f(z)|^2\,dz\Bigg)^{1/2},
\end{eqnarray}

\noindent where $Q_s$ is a cube with the same center as $Q$ and sidelength $s$. Since $s\geq l(Q)$, in particular, $Q_s\supset Q\ni x$. Then the right-hand side of (\ref{eq6.27}) is bounded by
\begin{eqnarray}\label{eq6.28} &&C \sum_{j=0}^\infty 2^{-jN}\,
\sup_{Q_s\ni x}\Bigg(\frac{1}{|2^jQ_s|}\int_{2^jQ_s}|(I-e^{-l(Q_s)^2L})^M f(z)|^2\,dz\Bigg)^{1/2}\nonumber\\[4pt]
&&\quad \leq C \sup_{Q\ni x}\sup_{j\in\NN\cup\{0\}}\left(\frac{1}{|2^jQ|}\int_{2^jQ}(I-e^{-l(Q)^2L})^M f(z)|^2\,dz\right)^{1/2} \sum_{j=0}^\infty 2^{-jN}\nonumber\\[4pt]
&&\quad \leq C \sup_{Q\ni x}\sup_{j\in\NN\cup\{0\}}\left(\frac{1}{|2^jQ|}\int_{2^jQ}(I-e^{-l(Q)^2L})^M f(z)|^2\,dz\right)^{1/2}.
\end{eqnarray}
\noindent This finishes the proof of (\ref{eq6.22}).

\vskip 0.08in \noindent{\bf Step IV.} Finally, (\ref{eq6.15}), (\ref{eq6.22}) allow to conclude that
\begin{multline}\label{eq6.29} \Bigg\|\sup_{Q\ni x}\left(\frac{1}{|Q|}\int_0^{l(Q)}\int_{Q}|(t^2L)^{M+1}e^{-t^2L}f(x)|^2\,\frac{dxdt}{t}\right)^{1/2}\Bigg\|_{L^p(\RR^n)}\\[4pt]
\leq C \|{\mathcal M}_2({\mathcal M}^{\sharp}_M f)\|_{L^p(\RR^n)},\end{multline}

\noindent for every $0<p<\infty$. But since the classical Hardy-Littlewood maximal function is bounded in $L^p$ for $1<p<\infty$, the operator
${\mathcal M}_2$ is bounded in $L^p(\RR^n)$ for $2<p<\infty$, and therefore,
\begin{equation}\label{eq6.30}
\|{\mathcal M}_2({\mathcal M}^{\sharp}_M f)\|_{L^p(\RR^n)}\leq C\|{\mathcal M}^{\sharp}_M f\|_{L^p(\RR^n)},\qquad 2<p<\infty.\end{equation}

\noindent Now the combination of (\ref{eq6.29}) and (\ref{eq6.30}) yields (\ref{eq6.14}) and finishes the proof of the theorem. \ep

\section{Fractional powers of the operator $L$.}\label{s7}
\setcounter{equation}{0}

Recall that for $p_-(L)<p<r<p_+(L)$
\begin{equation}\label{eq7.1}
L^{-\alpha}:L^p(\RR^n)\longrightarrow L^r(\RR^n), \qquad
\alpha=\frac 12\left(\frac np-\frac nr\right).
\end{equation}

\noindent This result has been
proved in \cite{AuscherSurvey}, Proposition~5.3. In this section we aim to prove
the generalization of (\ref{eq7.1}) to the full scale of $H^p_L$
spaces.

\begin{theorem}\label{t7.1}
Let $0<p<r<\infty$. Then
\begin{eqnarray}\label{eq7.2}
&& L^{-\alpha}:H^p_L(\RR^n)\longrightarrow H^r_L(\RR^n), \qquad
\alpha=\frac 12\left(\frac np-\frac nr\right),\\[4pt]
\label{eq7.3} && L^{-\alpha}:H^p_L(\RR^n)\longrightarrow
BMO_L(\RR^n), \qquad
\alpha=\frac {n}{2p},\\[4pt]
\label{eq7.4} && L^{-\alpha}:BMO_L(\RR^n)\longrightarrow
\Lambda^{2\alpha}_L(\RR^n), \qquad
\alpha>0,\\[4pt]
\label{eq7.5} &&
L^{-\alpha}:\Lambda^{\beta}_L(\RR^n)\longrightarrow
\Lambda^{\beta+2\alpha}_L(\RR^n), \qquad \alpha>0,\qquad \beta>0.
\end{eqnarray}
\end{theorem}

\bp Let us denote $p_n:=2n/(n+2)$ and $p_n':=2n/(n-2)$. 
We recall that by \cite{AuscherSurvey}, we have $p_-(L)<p_n$ and  $p_+(L) > p_n'.$
We begin by claiming that it is enough to prove (\ref{eq7.2}) for
\begin{equation}\label{eq7.6}
0<p<r\leq 1\quad\mbox{such that}\quad \frac 12\left(\frac np-\frac
nr\right)\leq\frac 12\left(\frac n{p_n}-\frac n2\right)= \frac{1}{2},
\end{equation}

\noindent which, in particular, says that
\begin{equation}\label{eq7.7}
0<\alpha=\frac 12\left(\frac np-\frac nr\right)\leq\frac 12.
\end{equation}

\noindent Indeed, once (\ref{eq7.2}) has been proved for this range, by
interpolating with (\ref{eq7.1}) via Lemma~\ref{l4.5}, we may obtain
that (\ref{eq7.2}) holds for all
\begin{equation}\label{eq7.8}
0<p<r<p_+(L)\quad\mbox{such that}\quad \frac 12\left(\frac np-\frac
nr\right)\leq\frac 12,
\end{equation}

\noindent with $\alpha$ satisfying (\ref{eq7.7}). We can
then write $L^{-\alpha}=(L^{-\alpha/k})^k$ for $k$ large enough in
order to remove restrictions on $\alpha$ and, equivalently, on the
difference between $p$ and $r$, and obtain (\ref{eq7.2}) for
\begin{equation}\label{eq7.9}
0<p<r<p_+(L),\qquad \alpha=\frac 12\left(\frac np-\frac nr\right),
\end{equation}
without restriction on the size of $\alpha$.
From here the results in (\ref{eq7.2})--(\ref{eq7.5})
follow for the full range of indices by duality and another
application of the procedure with $L^{-\alpha}=(L^{-\alpha/k})^k$.

Indeed, the fact that \eqref{eq7.2} holds for $1<p<r\leq 2$ for all elliptic operators (and hence, in particular, $L^*$) together with \eqref{eq1.10} implies that \eqref{eq7.2} holds also for $2< p<r<+\infty$. Combining this with the range \eqref{eq7.9} and suitably representing the powers of $L$ as a composition of smaller powers, we cover the full range $0<p<r<+\infty$ for \eqref{eq7.2}. Furthermore, using \eqref{eq7.2} for $L^*$ with $p=1$ and Theorem~\ref{t3.5},  we arrive at \eqref{eq7.3}. Similarly, dualizing \eqref{eq7.2} for $L^*$ with $r=1$ and using, once again, Theorem~\ref{t3.5}, one obtains \eqref{eq7.4}, and, by the same procedure starting with $0<p<r<1$,  \eqref{eq7.5}.

Thus, it suffices to establish (\ref{eq7.2}) under the restrictions
(\ref{eq7.6})--(\ref{eq7.7}), and it is to this task that we now turn.
We first show that
\begin{equation}\label{eq7.10}
\|S(L^{-\alpha}m)\|_{L^r(\RR^n)}\leq C,\quad \mbox{for every
$(H^p_L,\eps,M)$-molecule $m$},
\end{equation}

\noindent where $M>\frac{n}{2}\Bigl(\frac 1p-\frac 12\Bigr)$ and
$\eps>0$. To this end, observe that by H\"older's inequality
\begin{equation}\label{eq7.11}
 \|S(L^{-\alpha}m)\|_{L^r(\RR^n)}^r \leq C\sum_{j=0}^\infty
(2^{j}l(Q))^{n(1-r/2)} \|S(L^{-\alpha}m)\|_{L^2(S_j(Q))}^r.
\end{equation}

\noindent When $j\leq 10$, we employ boundedness of $S$ in
$L^2(\RR^n)$ and (\ref{eq7.1}) to obtain the estimate
\begin{equation}\label{eq7.12}
\|S(L^{-\alpha}m)\|_{L^2(S_j(Q))} \leq \|m\|_{L^{q}(\RR^n)}.
\end{equation}

\noindent Here and throughout the proof $q$ is such that $\alpha =
\frac 12\left(\frac n{q}-\frac n2\right)$, so that $q\leq 2$ and
$q>p_n$ by (\ref{eq7.7}). Then by the  definition of the molecule
the expression above is bounded by $l(Q)^{\frac n2 -\frac nr }$.
Indeed, by H\"older inequality every $(H^p_L,\eps,M)$ - molecule
satisfies (\ref{eq3.21}) for $q\leq 2$. Therefore,
\begin{equation}\label{eq7.13}
\|m\|_{L^{q}(\RR^n)}\leq \sum_{j=0}^\infty \|m\|_{L^q(S_j(Q))}\leq
C l(Q)^{\frac nq -\frac np }=Cl(Q)^{\frac n2 -\frac nr },
\end{equation}

\noindent since $\frac 12\left(\frac np-\frac nr\right)=\alpha =
\frac 12\left(\frac n{q}-\frac n2\right)$.

Turning to the case $j\geq 10$, one can represent the molecule as
follows
\begin{eqnarray}\label{eq7.14}
 \hskip -.7cm m & = & (I-e^{-l(Q)^2L})^Mm+[I-(I-e^{-l(Q)^2L})^{M}]m\nonumber\\[4pt]
\hskip -.7cm \qquad\quad & = & (I-e^{-l(Q)^2L})^Mm+ \sum_{1\leq k\leq
M}C_{k,M}\left(\frac k{M}\, l(Q)^{2}Le^{-\frac
k{M}l(Q)^2L}\right)^{M}(l(Q)^{-2}L^{-1})^{M}m, \label{eqfp10}
\end{eqnarray}

\noindent where $C_{k,M}$ are some constants depending on $k,M$ only. Starting with the first term above, we write
\begin{multline}\label{eq7.15}
\|S(L^{-\alpha}(I-e^{-l(Q)^2L})^M m)\|_{L^2(S_j(Q))}\\[4pt] \leq \,\,
\|S(L^{-\alpha}(I-e^{-l(Q)^2L})^M (m\chi_{\widehat{S}_j(Q)}))\|_{L^2(S_j(Q))}\\[4pt]
 +\,\,\,\|S(L^{-\alpha}(I-e^{-l(Q)^2L})^M
(m\chi_{\RR^n\setminus \widehat{S}_j(Q)}))\|_{L^2(S_j(Q))},
\end{multline}

\noindent where, as before, 
\begin{equation}\label{eq7.16} \widehat{S}_j (Q) := 2^{j+2}Q \setminus 2^{j-3}Q.
\end{equation}

\noindent Then
\begin{multline}\label{eq7.17}
\|S(L^{-\alpha}(I-e^{-l(Q)^2L})^M(m\chi_{\widehat{S}_j(Q)}))\|_{L^2(\RR^n)}\\\leq
C \|m\|_{L^{q}(\widehat{S}_j(Q))}\leq C\,(2^jl(Q))^{\frac n2-\frac
nr}\,2^{-j\eps}.
\end{multline}

\noindent As for the second part of (\ref{eq7.15}), using the
notation (\ref{eq3.15}), one can write
\begin{eqnarray}
&&\hskip -0.7 cm \nonumber\|S(L^{-\alpha}(I-e^{-l(Q)^2L})^M
(m\chi_{\RR^n\setminus
\widehat{S}_j(Q)}))\|_{L^2(S_j(Q))}\\[4pt]
 &&\hskip -0.7 cm \nonumber\qquad\leq C
\left(\dint_{{\mathcal R}(S_j(Q))}
|s^2Le^{-s^2L}L^{-\alpha}(I-e^{-l(Q)^2L})^{M}(m\chi_{\RR^n\setminus
\widehat{S}_j(Q)})(x)|^2\,\frac{ds\,dx}{s}\right)^{1/2}\\[4pt]\nonumber
&&\hskip -0.7 cm \qquad \leq C \left( \int_{\widetilde
S_j(Q)}\int_0^\infty
|s^2Le^{-s^2L}L^{-\alpha}(I-e^{-l(Q)^2L})^{M}(m\chi_{\RR^n\setminus
\widehat{S}_j(Q)})(x)|^2\,\frac{ds\,dx}{s}\right)^{1/2}\\[4pt]
&&\hskip -0.7 cm \qquad +\,C \left( \int_{\RR^n\setminus
\widetilde S_j(Q)}\int_{c2^jl(Q)}^\infty
|s^2Le^{-s^2L}L^{-\alpha}(I-e^{-l(Q)^2L})^{M}(m\chi_{\RR^n\setminus
\widehat{S}_j(Q)})(x)|^2\,\frac{ds\,dx}{s}\right)^{1/2}\nonumber \\[4pt]
&& \qquad =: \,\,\, I+II.\label{eq7.18}
\end{eqnarray}

We claim that for arbitrary closed sets $E,F\subset\RR^n$
\begin{equation}\label{eq7.19}
\left\| \frac s\tau\, e^{-sL}(I-e^{-\tau L})g\right\|_{L^2(F)}
\leq C e^{-\frac{{\rm dist}\,(E,F)^2}{cs}}\|g\|_{L^2(E)},
\end{equation}

\noindent provided $s\geq \tau$ and ${\rm{ supp}}\,g\subset E$.
Indeed,
\begin{eqnarray}\nonumber
&&\left\| \frac s\tau
(e^{-sL}-e^{-(s+\tau)L})g\right\|_{L^2(F)}=\left\|\frac s\tau
\int_0^\tau\partial_r e^{-(s+r)L}g\,dr\right\|_{L^2(F)}\\[4pt]
&&\quad \leq C \frac s\tau\int_0^\tau \left\|(s+r)L
e^{-(s+r)L}g\right\|_{L^2(F)}\,\frac{dr}{s+r} \nonumber\\[4pt]
&&\quad \leq  C \,\|g\|_{L^2(E)} \,\left(\frac s\tau\int_0^\tau
e^{-\frac{{\rm
dist}\,(E,F)^2}{c(s+r)}}\,\frac{dr}{s+r}\right).\label{eq7.20}
\end{eqnarray}

\noindent Since $s+r\approx s$ for $s\geq \tau$ and
$r\in(0,\tau)$, the expression above does not exceed
\begin{equation}\label{eq7.21}
C \,\|g\|_{L^2(E)} e^{-\frac{{\rm dist}\,(E,F)^2}{cs}}
\,\left(\frac s\tau\int_0^\tau \frac{dr}{s+r}\right) \leq C
e^{-\frac{{\rm dist}\,(E,F)^2}{cs}}\|g\|_{L^2(E)}.
\end{equation}

Next, recall that
\begin{equation}\label{eq7.22}
L^{-\alpha}f=C\int_0^\infty t^{\alpha-1} e^{-tL}f\,dt.
\end{equation}

\noindent Then we obtain the estimate
\begin{eqnarray}\nonumber &&  II\leq C \left( \int_{[c2^jl(Q)]^2}^\infty \int_{\RR^n}
|sLe^{-sL}L^{-\alpha}(I-e^{-l(Q)^2L})^{M}(m\chi_{\RR^n\setminus
\widehat{S}_j(Q)})(x)|^2\,\frac{dx\,ds}{s}\right)^{1/2}\\[4pt]\nonumber
&&\quad \leq
 C \left(\int_{[c2^jl(Q)]^2}^\infty\left(\int_0^\infty
t^{\alpha-1}
\|sLe^{-sL}e^{-tL}(I-e^{-l(Q)^2L})^{M}(m\chi_{\RR^n\setminus
\widehat{S}_j(Q)})\|_{L^2(\RR^n)}
\,dt\right)^2\,\frac{ds}{s}\right)^{1/2}\\[4pt]\nonumber
&&\quad \leq C \,\Bigg(\int_{c'[2^jl(Q)]^2}^\infty\left(\int_0^\infty
t^{\alpha-1} \Bigg(\frac{l(Q)^2}{s+t}\right)^M\\[4pt]\nonumber
&&\qquad \quad \times \Bigl\|sL
e^{-sL}\left(\frac{s+t}{l(Q)^2}\right)^M
e^{-(s+t)L}(I-e^{-l(Q)^2L})^{M}(m\chi_{\RR^n\setminus
\widehat{S}_j(Q)})\Bigr\|_{L^2(\RR^n)}
\,dt\Bigg)^2\,\frac{ds}{s}\Bigg)^{1/2}.\nonumber
\end{eqnarray}

\noindent To estimate the last line above, we split further $e^{-(s+t)L}(I-e^{-l(Q)^2L})^{M}=\left[e^{-\frac{(s+t)}{M}\,L}(I-e^{-l(Q)^2L})\right]^{M}$ and use  Lemma~\ref{l2.6} and (\ref{eq7.19}) with $\tau=l(Q)^2$ and $(s+t)/M$ in place of $s$  (assuming that $(s+t)/M\geq l(Q)^2$) and otherwise, if $c'[2^jl(Q)]^2\leq (s+t)/M\leq l(Q)^2$, just directly  Lemma~\ref{l2.6}. All in all, 

\begin{eqnarray}\nonumber &&  II\leq C \,\Bigg(\int_{c'[2^jl(Q)]^2}^\infty\left(\int_0^\infty
t^{\alpha-1} \Bigg(\frac{l(Q)^2}{s+t}\right)^M
\,dt\Bigg)^2\,\frac{ds}{s}\Bigg)^{1/2}\|m\|_{L^2(\RR^n)}\\[4pt]
&&\qquad\qquad\qquad\qquad  \leq C (2^jl(Q))^{\frac n2-\frac nr}2^{j\left(\frac
np-\frac n2-2M\right)}\leq C (2^jl(Q))^{\frac n2-\frac
nr}2^{-j\eps}, \label{eq7.23}
\end{eqnarray}

\noindent with $\eps$ denoting minimum between $\eps$ from the
definition of the $(p,\eps,M)$ molecule and $\frac np-\frac
n2-2M$. We do not distinguish them in the notation as soon as
$\eps>0$. 

In order to estimate $I$, let us denote by $S^*$ the vertical
version of square function, i.e.
\begin{equation}\label{eq7.24}
S^*f(x)=\left(\int_0^\infty |t^2Le^{-t^2L}f(x)|^2\,\frac
{dt}{t}\right)^{1/2},\qquad x\in\RR^n,
\end{equation}

\noindent and record the following result:
\begin{equation}\label{eq7.25}
\|S^*e^{-tL}(I-e^{-\tau L})^{M}f\|_{L^2(F)}\leq C\,
\left(\frac{\max\{t,\tau\}}{{\rm
dist}\,(E,F)^2}\right)^{M}\,\|f\|_{L^2(E)},
\end{equation}

\noindent for arbitrary closed sets $E,F\subset\RR^n$, $f\in
L^2(E)$ and $t,\tau>0$. For $t=0$ this has been established in
Theorem~3.2, \cite{HM}, and the proof of (\ref{eq7.25}) follows
the same path. Then
\begin{eqnarray}\nonumber
&&\hskip -0.7cm I\leq C \int_{0}^{(M+1)l(Q)^2}t^{\alpha-1} \|S^*
e^{-tL}(I-e^{-l(Q)^2L})^{M}(m\chi_{\RR^n\setminus
\widehat{S}_j(Q)})\|_{L^2(\widetilde S_j(Q))}
\,dt\\[4pt]\nonumber
&&\hskip -0.7cm \quad +\,C \int_{(M+1)l(Q)^2}^\infty
t^{\alpha-1}\left(\int_0^\infty \|sLe^{-sL}
e^{-tL}(I-e^{-l(Q)^2L})^{M}(m\chi_{\RR^n\setminus
\widehat{S}_j(Q)})\|_{L^2(\widetilde S_j(Q))}^2\,\frac{ds}{s}\right)^{1/2}
\,dt \\[4pt]
&& \qquad =: \, I_1+I_2\label{eq7.26}
\end{eqnarray}

\noindent where we first used \eqref{eq7.22}, then Minkowski inequality to switch the $L^2$ and $L^1$ norms, then split the integral in $t$, and then made a substitution $s^2$ to $s$ in the second term.  Then,  by (\ref{eq7.25})
\begin{equation}\label{eq7.27}
I_1\leq C l(Q)^{2\alpha} 2^{-2Mj}\|m\|_{L^2(\RR^n)}\leq C
(2^jl(Q))^{\frac n2-\frac nr}2^{-j\eps}.
\end{equation}

\noindent Going further,
\begin{eqnarray}\nonumber
&&\hskip -0.7cm I_2\leq C \int_{(M+1)l(Q)^2}^\infty
t^{\alpha-1}\Bigg(\int_0^t \left(\frac st\right)^2 \left(\frac {l(Q)^2}t\right)^{2M}\\[4pt]\nonumber
&&\hskip -0.7cm \qquad\qquad \times
\Bigg\|tLe^{-tL}e^{-sL}\left(\frac
t{l(Q)^2}\right)^M(I-e^{-l(Q)^2L})^{M}(m\chi_{\RR^n\setminus
\widehat{S}_j(Q)})\Bigg\|_{L^2(\widetilde
S_j(Q))}^2\,\frac{ds}{s}\Bigg)^{1/2}
\,dt \\[4pt]\nonumber
&&\hskip -0.7cm \qquad +\,C \int_{(M+1)l(Q)^2}^\infty
t^{\alpha-1}\Bigg(\int_t^\infty \left(\frac {l(Q)^2}s\right)^{2M}\\[4pt]
&&\hskip -0.7cm \qquad\qquad \times \Bigg\|sLe^{-sL}
e^{-tL}\left(\frac s{l(Q)^2}\right)^M
(I-e^{-l(Q)^2L})^{M}(m\chi_{\RR^n\setminus
\widehat{S}_j(Q)})\Bigg\|_{L^2(\widetilde
S_j(Q))}^2\,\frac{ds}{s}\Bigg)^{1/2} \,dt. \nonumber
\end{eqnarray}

\noindent According to (\ref{eq7.19}), the expression above is
bounded by
\begin{eqnarray}\nonumber
&&\hskip -0.7cm C \|m\|_{L^2(\RR^n)}\int_{(M+1)l(Q)^2}^\infty
t^{\alpha-1}\left(\frac
{l(Q)^2}t\right)^{M}e^{-\frac{(2^jl(Q))^2}{ct}} \Bigg(\int_0^t
\left(\frac st\right)^2 \,\frac{ds}{s}\Bigg)^{1/2}
\,dt \\[4pt]
&&\hskip -0.7cm \qquad +\,C \|m\|_{L^2(\RR^n)}
\int_{(M+1)l(Q)^2}^\infty t^{\alpha-1}\Bigg(\int_t^\infty \left(\frac
{l(Q)^2}s\right)^{2M} e^{-\frac{(2^jl(Q))^2}{cs}}
\,\frac{ds}{s}\Bigg)^{1/2} \,dt. \label{eq7.28}
\end{eqnarray}

\noindent Here, to estimate the first term, we used (\ref{eq7.19}) with $\frac t{M+1}$ in place of $s$ and $l(Q)^2$ in place of $\tau$, splitting $tLe^{-tL}(I-e^{-l(Q)^2L})^{M}=tLe^{-\frac{t}{M+1}L}\left[e^{-\frac{t}{M+1}L} (I-e^{-l(Q)^2L})\right]^{M}$. Similarly, for the second term we employed (\ref{eq7.19}) with $\frac s{M+1}$ in place of $s$ and $l(Q)^2$ in place of $\tau$, and split $sLe^{-sL}(I-e^{-l(Q)^2L})^{M}=sLe^{-\frac{s}{M+1}L}\left[e^{-\frac{s}{M+1}L} (I-e^{-l(Q)^2L})\right]^{M}$. 

Now, making the change of variables $t\mapsto r$,
$r:=-\frac{(2^jl(Q))^2}{ct}$, in the first line of  \eqref{eq7.28}, we control
it by
\begin{equation}\label{eq7.29}
C(2^jl(Q))^{2\alpha}2^{-2jM}\|m\|_{L^2(\RR^n)}\leq C
(2^jl(Q))^{\frac n2-\frac nr}2^{-j\eps}.
\end{equation}

\noindent In order to control the second term in (\ref{eq7.28}),
let us take some $\delta>0$ and write
\begin{eqnarray}\nonumber
&&C \|m\|_{L^2(\RR^n)} \int_{(M+1)l(Q)^2}^\infty
t^{\alpha-1}\Bigg(\int_t^\infty \left(\frac {l(Q)^2}s\right)^{2M}
e^{-\frac{(2^jl(Q))^2}{cs}} \,\frac{ds}{s}\Bigg)^{1/2} \,dt\\[4pt]\nonumber
&&\, \leq C \|m\|_{L^2(\RR^n)} \int_{(M+1)l(Q)^2}^\infty
t^{\alpha-1}\left(\frac
{l(Q)^2}t\right)^{\alpha+\delta}\Bigg(\int_t^\infty \left(\frac
{l(Q)^2}s\right)^{2M-2\alpha-2\delta}
e^{-\frac{(2^jl(Q))^2}{cs}} \,\frac{ds}{s}\Bigg)^{1/2} dt\\[4pt]
&&\qquad \leq \,C
l(Q)^{2\alpha}2^{-j(2M-2\alpha-2\delta)}\|m\|_{L^2(\RR^n)}\,\leq \,C
(2^jl(Q))^{\frac n2-\frac nr}2^{-j\eps},\label{eq7.30}
\end{eqnarray}

\noindent provided $\delta>0$ is small enough.

All in all, we have the desired control for $SL^{-\alpha}$ acting
on the first term in (\ref{eqfp10}). The second one can be handled
by a similar argument, since $(l(Q)^{-2}L^{-1})^Mm$ satisfies the
same size conditions as a molecule itself and
$\left(l(Q)^{2}Le^{-l(Q)^2L}\right)^{M}$ behaves much as
$(I-e^{-l(Q)^2L})^M$. Roughly speaking, these two operators
exhibit the same cancellation and decay properties (it can be
seen, e.g., from the argument of Theorem~\ref{t3.5}).

This finishes the proof of (\ref{eq7.10}), and it remains only to pass
to (\ref{eq7.2}), under the conditions (\ref{eq7.6})--(\ref{eq7.7}).  In particular, $\alpha\leq 1/2,$
so by (\ref{eq7.1}), and the fact that $p_+(L)> 2n/(n-2)$ (cf. \cite{AuscherSurvey}),
we then have that 
\begin{equation}\label{alphabound}L^{-\alpha}: L^2(\RR^n)\to L^q(\RR^n),  \qquad
\frac{1}{q}= \frac{1}{2}- \frac{2\alpha}{n}.\end{equation}

Now by density, as usual it is enough to work with $f \in \HPHML$,
so that there is an $L^2$ convergent molecular decomposition
$f=\sum \lambda_j m_j$, with $\sum|\lambda_j|^p \lesssim \|f\|_{\HPHML}^p.$
Consequently, (\ref{alphabound}) 
implies that $$L^{-\alpha} f = \sum \lambda_j\, L^{-\alpha} m_j\quad\, {\rm in}\,\, 
L^q(\RR^n),$$
and therefore also, since $q<p_+(L),$ that
$$ S \left(L^{-\alpha} f\right) \leq \sum |\lambda_j|\, S\left(L^{-\alpha} m_j\right)$$
(here we have used that  $S:L^q\to L^q$ whenever
$p_-(L)<q<p_+(L)$, by a slight modification of an argument in \cite{AuscherSurvey}, Theorem~6.1).
It is now immediate that (\ref{eq7.10}) implies
(\ref{eq7.2}), under the conditions (\ref{eq7.6})--(\ref{eq7.7}), and as we have observed above, the conclusion of Theorem \ref{t7.1} follows.

 \ep

\section{Functional calculus and fractional powers of $L$ in smoothness spaces.}\label{s8}
\setcounter{equation}{0}

\subsection{Functional calculus and fractional powers of $L$ in $H^p_L$-$BMO_L$-$\Lambda^\alpha_L$ spaces}\label{s8.1}

Recall from Section~\ref{s2.1} that $L$ has a bounded holomorphic functional calculus on $L^2$ and (\ref{eq2.5}) holds.
In general, these properties do not extend to all $L^p$, $1<p<\infty$. Otherwise, the heat semigroup would be bounded in all $L^p$, $1<p<\infty$, as an $H^\infty$ function, which would contradict Proposition~\ref{p2.1}. However, the functional calculus can be extended
to a full scale of
$H^p_L$-$BMO_L$-$\Lambda^\alpha_L$ spaces.

\begin{lemma}\label{l8.1} The operator $L$ defined in
(\ref{eq1.1})--(\ref{eq1.3}) has  a \Bk bounded holomorphic
functional calculus in $H^p_L(\RR^n)$, $0< p< \infty$,
$BMO_L(\RR^n)$ and $\Lambda^\alpha_L(\RR^n)$, $\alpha>0$, in the
following sense.

When $0<p\leq 2$, for every non-trivial $\psi\in H^\infty(\Sigma_\mu^0)$ the
operator $\psi(L)$ originally defined on $L^2(\RR^n)$ extends by
continuity to a bounded operator on $H^p_L(\RR^n)$ satisfying
\begin{equation}\label{eq8.1}
\|\psi(L)f\|_{H^p_L(\RR^n)}\leq C \|\psi\|_{L^\infty(\Sigma_\mu^0)}
\|f\|_{H^p_L(\RR^n)} \quad \mbox{for every}\quad f\in H^p_L(\RR^n).
\end{equation}

\noindent For $p>2$ the operator $\psi(L)$ can be defined on
$H^p_L(\RR^n)$ by duality:
\begin{equation}\label{eq8.2}
\forall f\in H^p_L(\RR^n),\,\, p>2, \,\, \forall g\in
H^{p'}_{L*}(\RR^n)\,\,\qquad\langle \psi(L)f,g\rangle:=\langle f,\psi(L^*)g\rangle,
\end{equation}

\noindent and satisfies (\ref{eq8.1}). In the same way $\psi(L)$ can
be defined on $BMO_L(\RR^n)$ and $\Lambda^\alpha_L(\RR^n)$,
$\alpha>0$, and
\begin{equation}\label{eq8.3}
\|\psi(L)f\|_{BMO_L(\RR^n)}\leq  C
\|\psi\|_{L^\infty(\Sigma_\mu^0)} \|f\|_{BMO_L(\RR^n)} \quad
\mbox{for every}\quad f\in BMO_L(\RR^n),\end{equation}\vskip -.6cm
\begin{equation}\label{eq8.4}
\|\psi(L)f\|_{\Lambda^\alpha_L(\RR^n)}\leq  C
\|\psi\|_{L^\infty(\Sigma_\mu^0)} \|f\|_{\Lambda^\alpha_L(\RR^n)}
\quad \mbox{for every}\quad f\in \Lambda^\alpha_L(\RR^n),\quad
\alpha>0.
\end{equation}
\end{lemma}

\bp Let $0<p\leq 2$ and $\beta>\frac{n}{2}\left(\max\{\frac
1p,1\}-\frac 12\right)$. Now take $\psi\in\Psi_{\beta,\beta}(\Sigma_{\mu}^0)$ and
build $\widetilde \psi\in\Psi_{\beta,\beta}(\Sigma_{\mu}^0)$ using (\ref{eq4.18}) so that (\ref{eq4.17}) is satisfied.
Then for any $g\in H^p_L(\RR^n)$
\begin{equation}\label{eq8.5}
Q_{\psi}g\in T^p(\RR_+^{n+1})\quad \mbox{and}\quad \|Q_{\psi}g\|_{T^p(\RR_+^{n+1})}\leq C \|g\|_{H^p_L(\RR^n)}.
\end{equation}

\noindent Furthermore, by Proposition~\ref{p4.1}
\begin{equation}\label{eq8.6}
Q_{\psi}\circ f(L)\circ \pi_{\widetilde \psi}:T^p(\RR_+^{n+1})\longrightarrow T^p(\RR_+^{n+1}),
\end{equation}

\noindent and hence, by (\ref{eq8.5})
\begin{equation}\label{eq8.7}
Q_{\psi}\circ f(L)=Q_{\psi}\circ f(L)\circ \pi_{\widetilde \psi}\circ Q_{\psi}:H^p_L(\RR^n)\longrightarrow T^p(\RR_+^{n+1}).
\end{equation}

\noindent By virtue of (\ref{eq4.23}) the property (\ref{eq8.7}) implies that
\begin{equation}\label{eq8.8}
f(L):H^p_L(\RR^n)\longrightarrow H^p_L(\RR^n),
\end{equation}

\noindent thereby concluding the case $0<p\leq 2$.

Now the functional calculus of $L$ in $H^p_L$ for $p>2$, 
$BMO_L$ and $\Lambda^{\alpha}_L$, $\alpha>0$,   follows from
(\ref{eq1.10}) and Theorem~\ref{t3.5}.\ep

\subsection{Classical scales of function spaces measuring
smoothness}\label{s8.2}

So far we have worked with a few different scales of function spaces on
$\RR^n$: $L^p(\RR^n)$, $1<p\leq \infty$, Hardy  spaces $H^p(\RR^n)$,
$0<p\leq 1$, homogeneous Sobolev spaces $\dot W^{s,p}(\RR^n)$,
$s\in\RR$, $1<p<\infty$ (cf. 
(\ref{eq5.36})), and their counterparts for $p\leq 1$ and $s=1$, namely
the regular Hardy spaces $H^{1,p}(\RR^n)$ defined in (\ref{eq5.21}).
All of them belong to (or can be identified with the members of) a
more extensive scale of the Triebel-Lizorkin spaces, $\dot
F^{p,q}_s(\RR^n)$, $s\in\RR$, $0<p,q<\infty$.

Let us denote by ${\mathcal F}$ the Fourier transform operator. We fix a Schwartz function $\varphi$ such that:
\begin{enumerate}
\item ${\rm supp}\,{\mathcal{F}}(\varphi)\subseteq
\{\xi\in{\mathbb R}^n:\,\frac{1}{2}\leq|\xi|\leq 2\}$,

\item $\left|{\mathcal F}(\varphi)(\xi)\right|\geq c>0$ uniformly for
$\frac{3}{5}\leq|\xi|\leq\frac{5}{3}$,

\item $\sum_{i\in {\mathbb Z}}|{\mathcal F}(\varphi)(2^i\xi)|^2=1$ if $\xi\neq 0$,
\end{enumerate}
\noindent and let $\varphi_i(x):=2^{in}\varphi (2^ix)$, $i\in\ZZ$, $x\in\RR^n$.
Then for $s\in{\mathbb R}$, $0<p<\infty$ and $0<q\leq
\infty$,
\begin{equation}\label{eq8.9}
\dot{F}^{p,q}_s({\mathbb R}^n):=\Bigl\{f\in{\mathcal S}'/{\mathcal
P}:\, \|f\|_{\dot{F}^{p,q}_s({\mathbb R}^n)}
:=\Bigl\|\Bigl(\sum_{i\in{\mathbb{Z}}}(2^{is}|\varphi_i*f|)^q
\Bigr)^{\frac{1}{q}}\Bigr\|_{L^p}<\infty\Bigr\},
\end{equation}

\noindent where ${\mathcal S}'/{\mathcal P}$ is the space of
tempered distributions on $\RR^n$ modulo polynomials. We have
\begin{eqnarray}
\label{eq8.10} L^p({\mathbb{R}}^n) & \approx & \dot
F_0^{p,2}({\mathbb{R}}^n),\quad 1<p<\infty,
\\[4pt]
\label{eq8.11} \dot W^{s,p}({\mathbb{R}}^n) & \approx & \dot
F_s^{p,2}({\mathbb{R}}^n), \quad   1<p<\infty, \qquad s\in\RR,
\\[4pt]
\label{eq8.12} H^p({\mathbb{R}}^n) & \approx & \dot
F_0^{p,2}({\mathbb{R}}^n), \quad  0<p\leq 1,
\\[4pt]
\label{eq8.13} H^{1,p}({\mathbb{R}}^n) & \approx & \dot
F_1^{p,2}({\mathbb{R}}^n), \quad  0<p\leq 1.
\end{eqnarray}

\noindent The details on the identifications in (\ref{eq8.10})
and (\ref{eq8.12}) are presented in \cite{FrJa} (Remark 7.8 and
Appendix B). The identifications (\ref{eq8.11}) and (\ref{eq8.13}) will be
discussed after Lemma~\ref{l8.3}. We shall use the following
basic properties of Triebel-Lizorkin spaces.
\begin{lemma}\label{l8.2} The space
\begin{equation}\label{eq8.14}
{\mathcal Z}(\RR^n):=\{\varphi\in {\mathcal S}(\RR^n):\,(D^\alpha
{\mathcal F} \varphi)(0)=0 \mbox{ for every multiindex }\alpha\}
\end{equation}

\noindent is a dense subspace of $\dot{F}^{p,q}_s({\mathbb R}^n)$
for all $s\in{\mathbb R}$, $0<p,q<\infty$.
\end{lemma}
\begin{lemma}\label{l8.3} The operator $\Delta^\alpha$, $\alpha\in\RR$, is an
isomorphism from $\dot{F}^{p,q}_s({\mathbb R}^n)$ onto
$\dot{F}^{p,q}_{s-2\alpha}({\mathbb R}^n)$, $s\in{\mathbb R}$,
$0<p,q<+\infty$. Also, for any $m\in{\mathbb{N}}$,
\begin{eqnarray}\label{eq8.15}
\dot F^{p,q}_s({\mathbb{R}}^n) & = & \{f\in {\mathcal S}'/{\mathcal
P}:\,D^\alpha f\in \dot F^{p,q}_{s-m}({\mathbb{R}}^n),
\,\,\forall\,\alpha\mbox{ with }|\alpha|= m\}.
\end{eqnarray}
\end{lemma}

Lemma~\ref{l8.2} is proved in \cite{Tr83}, Section 5.1.3, and
Lemma~\ref{l8.3} directly follows from Theorem 5.2.3 in
\cite{Tr83}. Note that Lemma~\ref{l8.3} together with
(\ref{eq8.10}) and (\ref{eq8.12}) implies (\ref{eq8.11})
and (\ref{eq8.13}).

Finally, we would like to record the following consequence of the
Kato estimate.
\begin{lemma}\label{l8.4} Let $L$ be an operator defined by (\ref{eq1.1})--(\ref{eq1.3}). Then $L^\alpha$, $-1/2\leq\alpha\leq 1/2$, is an
isomorphism from $\dot{W}^{s,2}({\mathbb R}^n)$ onto
$\dot{W}^{s-2\alpha,2}({\mathbb R}^n)$, $-1\leq s\leq 1$.
\end{lemma}

\bp The Kato estimate (\ref{eq1.4}) implies that 
$L^{1/2}$ maps the Sobolev space $\dot W^{1,2}(\RR^n)$ isomorphically onto $L^2(\RR^n)$. 
Using this observation and interpolation, one can further show
that
\begin{equation}\label{eq8.16}
\mbox{$L^\alpha$, $0\leq \alpha\leq 1/2$, is an isomorphism between
$\dot W^{2\alpha,2}(\RR^n)$ and $L^2(\RR^n)$,} \end{equation}

\noindent (see, e.g. the proof of Proposition~5.3 in
\cite{AuscherSurvey} for the details). Now we write
$L^\alpha=L^{s/2}\circ L^{-s/2+\alpha}$ and use duality and
(\ref{eq8.16}) to finish the argument. \ep

Interchanging the order in which $L^p$ and $\ell^q$ norms are taken
in (\ref{eq8.9}), one would obtain the homogeneous Besov spaces
$\dot B_s^{p,q}$, $s\in\RR$, $0<p,q\leq\infty$. There are also
appropriate versions of (\ref{eq8.9}) corresponding to $p=\infty$
or $q=\infty$; see, e.g., {{\cite{FrJa}}}, Sections~1,2, for the definitions. Since
we aim to concentrate on the properties of the operator $L$ in
Sobolev spaces and their counterparts for $p\leq 1$, we do not
further elaborate on this point. However, below we will use the notation $\dot
F_s^{p,2}$ in place of $\dot W^{s,p}$ and $H^{s,p}$ for uniformity and to avoid repetition when considering
$p>1$ and $p\leq 1$.

%\begin{corollary}\label{CorKato} Let $L$ be an operator defined by (\ref{eq1.1})--(\ref{eq1.3}). There exists $\eps=\eps(L)>0$ such that
% for all $p$ with $2-\eps(L)<p<2+\eps(L)$ the operator
% $L^\alpha$, $-1/2\leq\alpha\leq 1/2$, is an
%isomorphism from $\dot{W}^{s,p}({\mathbb R}^n)$ onto
%$\dot{W}^{s-2\alpha,p}({\mathbb R}^n)$, $-1\leq s\leq 1$.
%\end{corollary}

%\bp The Corollary is a direct consequence of Lemma~\ref{TrKato} and extrapolation result in \ref{AlMu}. \ep

\subsection{Weighted tent spaces}\label{s8.3}

Let $s\in\RR$, $0<p,q<\infty$, and consider the spaces
\begin{equation}\label{eq8.17}
T^{p,q}_s(\RR^{n+1}_+):=\{F:\RR^{n+1}_+\longrightarrow
\CC;\,\|F\|_{T_s^{p,q}(\RR^{n+1}_+)}:=\|{\Sq}_s^q
F\|_{L^p(\RR^n)}<\infty\},
\end{equation}

\noindent where
\begin{equation}\label{eq8.18}
{\Sq}_s^q F(x):= \left(\dint_{\Gamma(x)}
|F(y,t)|^q\frac{dydt}{t^{sq+n+1}}\right)^{1/q},\qquad x\in\RR^n.
\end{equation}

\noindent When $s=0$, these are the classical tent spaces we discussed in Section~\ref{s4}. They were first introduced and studied in
\cite{CMS}. In particular, the authors established the complex interpolation of tent spaces for $s=0$ and $p,q\geq 1$ (when the underlying spaces are
Banach). Later on the complex interpolation of tent spaces was
proved for $0<p,q<\infty$ and $s=0$ in \cite{Be}, \cite{VerbIP} (see
also \cite{AlMi1}, \cite{AlMi2}, \cite{BeCe}). We stated a partial case of this result in (\ref{eq4.4}). 
However, for the
applications we have in mind we need to show that the tent spaces
interpolate in $s,p$ and $q$ for the full range of indices.

\begin{lemma}\label{l8.5}
For all $s_0,s_1\in\RR$, $0< p_0,p_1<\infty$, $0< q_0,q_1<\infty$,
\begin{equation}\label{eq8.19}
\left[T_{s_0}^{p_0,q_0}(\RR^{n+1}_+),T_{s_1}^{p_1,q_1}(\RR^{n+1}_+)\right]_\theta=T_{s}^{p,q}(\RR^{n+1}_+),
\quad 0<\theta<1,
\end{equation}

\noindent where $s=(1-\theta)s_0+\theta s_1$,
$\frac{1}{p}=\frac{1-\theta}{p_0}+\frac\theta{p_1}$ and
$\frac{1}{q}=\frac{1-\theta}{q_0}+\frac\theta{q_1}$.
\end{lemma}

\bp As we already mentioned (see the discussion preceding
Lemma~\ref{l4.5}), extension of the complex interpolation method
to quasi-Banach spaces is not straight forward, and over the years
several approaches to this issue have been developed. Here we continue to follow the method of complex interpolation of analytically
convex spaces which have been employed in the classical
Hardy-Sobolev-Besov-Triebel-Lizorkin scales in \cite{KaMi},
\cite{MeMi}, \cite{KMM}, and for the tent spaces with $s=0$ in
\cite{VerbIP}.

According to Theorem 7.9 in \cite{KMM} (see also \cite{KaMi}), we
have
\begin{equation}\label{eq8.20}
\left[T_{s_0}^{p_0,q_0}(\RR^{n+1}_+),T_{s_1}^{p_1,q_1}(\RR^{n+1}_+)\right]_\theta=
\left(T_{s_0}^{p_0,q_0}(\RR^{n+1}_+)\right)^{(1-\theta)}\left(T_{s_1}^{p_1,q_1}(\RR^{n+1}_+)\right)^\theta,
\end{equation}

\noindent provided that $T_{s_i}^{p_i,q_i}(\RR^{n+1}_+)$,
$i=0,1$, are analytically convex and separable. Here the
space on the right-hand side of (\ref{eq8.20}) is interpreted as a
set of functions $F:\RR^{n+1}_+\longrightarrow \CC$ such that
$|F|\leq |G|^{1-\theta}|H|^{\theta}$ for some $G\in
T_{s_0}^{p_0,q_0}$ and $H\in T_{s_1}^{p_1,q_1}$, equipped with the
natural infimum norm.

The fact that the tent spaces are separable is fairly obvious
(note that $p,q<\infty$). Furthermore, any tent space is a
quasi-Banach lattice (a quasi-Banach space with a partial order),
and a quasi-Banach lattice $X$ is analytically convex if it is
lattice $r$-convex for some $r>0$, i.e.
\begin{equation}\label{eq8.21}
\left\|\Bigl(\sum_{j=1}^m|f_j|^r\Bigr)^{1/r}\right\|_X\leq
\Bigl(\sum_{j=1}^m\left\|f_j\right\|_X^r\Bigr)^{1/r}
\end{equation}

\noindent for any finite family $\{f_j\}_{1\leq j\leq m}\subset X$
(see Theorem 7.8 in \cite{KMM}). The elements of $T_{s}^{p,q}$
satisfy (\ref{eq8.21}) with $r=\min\{p,q\}$ by Minkowski inequality.
Hence, the spaces (\ref{eq8.17}) are analytically convex and
(\ref{eq8.20}) applies.

Now recall the factorization results from \cite{VerbIP} for the tent
spaces without weight:
\begin{equation}\label{eq8.22}
T_{0}^{p,q}(\RR^{n+1}_+)=
T_{0}^{p_0,q_0}(\RR^{n+1}_+)\,\cdot\,T_{0}^{p_1,q_1}(\RR^{n+1}_+),\qquad
\frac 1p=\frac{1}{p_0}+\frac{1}{p_1},\quad \frac
1q=\frac{1}{q_0}+\frac{1}{q_1},
\end{equation}

\noindent where $0<p,q \leq\infty$ and the product in (\ref{eq8.22})
is interpreted similarly to (\ref{eq8.20}). Since for all functions
$F,G:\RR^{n+1}_+\longrightarrow \CC$ and $s\in\RR$ we have
$\frac{F}{t^{s_0}}\,\frac{G}{t^{s_1}}=\frac{FG}{t^{s_0+s_1}}$, the
formula (\ref{eq8.22}) entails
\begin{equation}\label{eq8.23}
T_{s}^{p,q}(\RR^{n+1}_+)=
T_{s_0}^{p_0,q_0}(\RR^{n+1}_+)\,\cdot\,T_{s_1}^{p_1,q_1}(\RR^{n+1}_+),
\end{equation}
\noindent with $s=s_0+s_1$, $\frac{1}{p}=\frac{1}{p_0}+\frac 1{p_1}$
and $\frac{1}{q}=\frac 1{q_0}+\frac 1{q_1}$. Furthermore, it can be
checked directly that $\left(T_s^{p,q}\right)^r=T_{sr}^{p/r,q/r}$,
so that (\ref{eq8.23}) implies
\begin{equation}\label{eq8.24}
\left(T_{s_0}^{p_0,q_0}\right)^{(1-\theta)}\left(T_{s_1}^{p_1,q_1}\right)^\theta
=T_{s_0(1-\theta)}^{\frac{p_0}{(1-\theta)},\frac{q_0}{(1-\theta)}}\,\cdot\,T_{s_1\theta}^{\frac{p_1}{\theta},\frac{q_1}{\theta}}
=T_{s}^{p,q}(\RR^{n+1}_+),
\end{equation}
\noindent for $s=(1-\theta)s_0+\theta s_1$,
$\frac{1}{p}=\frac{1-\theta}{p_0}+\frac\theta{p_1}$ and
$\frac{1}{q}=\frac{1-\theta}{q_0}+\frac\theta{q_1}$. Together with
(\ref{eq8.20}) this finishes the proof.
 \ep

\subsection{Hardy-Sobolev spaces associated to $L$: general theory}\label{s8.4}

Let us now define a smooth version of the Hardy spaces
$H^{s,p}_L(\RR^n)$, $0\leq s\leq 1$, $0<p\leq 2$, as a completion of
$L^{-s/2}(L^2\cap H^{p}_L)$ in the norm
\begin{equation}\label{eq8.25}
\|f\|_{H^{s,p}_L(\RR^n)}:=\|S
L^{s/2}f\|_{L^p(\RR^n)}=\|L^{s/2}f\|_{H^p_L(\RR^n)}.
\end{equation}

\noindent Recall that $L^{-s/2}$ is an isomorphism of $L^2$ onto the
space $\dot W^{s,2}$, hence, $L^{-s/2}(L^2\cap H^{p}_L)$ is a
subspace of $\dot W^{s,2}$, in particular, $L^{s/2}f$ is
well-defined for every $f\in L^{-s/2}(L^2\cap H^{p}_L)$. Moreover,
it follows that
\begin{equation}\label{eq8.26}
\dot W^{s,2}(\RR^n)\cap H^{s,p}_L(\RR^n) \mbox{ is dense in }
H^{s,p}_L(\RR^n), \quad \mbox{ for all }\quad 0\leq s\leq
1,\,\,0<p\leq 2.
\end{equation}

\begin{lemma}\label{l8.6}
The operator $L^\alpha$, $-1/2\leq \alpha \leq 1/2$, is an
isomorphism of $H_{L}^{s,p}$ onto $H_L^{s-2\alpha,p}$ provided
$0\leq s-2\alpha \leq 1$, $0\leq s\leq 1$ and $0<p\leq 2$.
\end{lemma}

\bp This result is a direct consequence of the definitions. Indeed,
by definition $L^2\cap H^p_L$ is dense in $H^p_L$ and
\begin{equation}\label{eq8.27}
\|L^{-\alpha}f\|_{H^{2\alpha,p}_L}=\|SL^{\alpha}L^{-\alpha}f\|_{L^p}=\|f\|_{H^p_L},\quad \forall 
f\in L^2\cap H^p_L, \quad 0\leq \alpha\leq 1/2.
\end{equation}

\noindent Hence, the operator $L^{-\alpha}$ extends by continuity to
$L^{-\alpha}:H^p_L\to H^{2\alpha,p}_L$ and its range is closed in
$H^{2\alpha,p}_L$. On the other hand, its range contains
$L^{-\alpha}(L^2\cap H^{p}_L)$, a dense subset of $H^{2\alpha,p}_L$,
and therefore, the range of $L^{-\alpha}$ in $H^{2\alpha,p}_L$
actually coincides with $H^{2\alpha,p}_L$. Then $L^{-\alpha}$ is an
isomorphism of $H^p_L$ onto $H^{2\alpha,p}_L$, $0\leq \alpha\leq
1/2$. Using this fact and writing $L^{\alpha}=L^{s/2}\circ
L^{\alpha-s/2}$ we finish the proof of the Lemma. \ep

 Clearly, $H^{s,p}_L$ are analogues of the Sobolev spaces
adapted to the elliptic operator $L$. In particular,
Lemmas~\ref{l8.6}, \ref{l8.3}  and the remark after
(\ref{eq1.13}) show that
\begin{equation}\label{eq8.28}
H^{s,p}_\Delta(\RR^n)\approx\dot W^{s,p}(\RR^n),\qquad 0\leq s \leq
1,\quad 1<p\leq 2.
\end{equation}

 As their counterparts for $L=\Delta$, the spaces
$H^{s,p}_L(\RR^n)$ are amenable to complex interpolation, satisfy
natural duality properties, admit some version of the molecular
decomposition etc. If necessary, the scale of $H^{s,p}_L$ spaces can
be extended to the full range of $p$ and $s$ analogously to the
Triebel-Lizorkin spaces.  We do not pursue this subject in the
present paper, and only mention the results which are important for
the applications we have in mind.

\begin{lemma}\label{l8.7}
The operator $L$ has bounded holomorphic functional calculus in $H^{s,p}_L(\RR^n)$
for all $0\leq s \leq 1$ and $0<p\leq 2$, in the sense that for every $\varphi\in
H^{\infty}(\Sigma_\mu^0)$
\begin{equation}\label{eq8.29}
\varphi(L):H^{s,p}_L(\RR^n)\longrightarrow H^{s,p}_L(\RR^n),
\end{equation}

\noindent with the norm bounded by $\|\varphi\|_{L^\infty(\Sigma_\mu^0)}$.

 Moreover, for every $\varphi\in
\Psi'(\Sigma_\mu^0)$ and for all $0\leq \alpha,\beta\leq 1$ and $0<p\leq q\leq 2$
\begin{equation}\label{eq8.30}
\varphi(L):H_{L}^{\alpha,p}(\RR^n)\longrightarrow H_{L}^{\beta,q}(\RR^n),
\end{equation}

\noindent and
\begin{equation}\label{eq8.31}
\|\varphi(L)f\|_{H_{L}^{\beta,q}(\RR^n)}\leq C
\left\|z^{\frac{\beta-\alpha}{2}+\frac 12\left(\frac np-\frac
nq\right)}\varphi\right\|_{L^\infty(\Sigma_\mu^0)}\|f\|_{H_{L}^{\alpha,p}(\RR^n)},
\end{equation}

\noindent whenever the $L^\infty$ norm on the right-hand side is
finite.

\end{lemma}

\bp The Lemma follows directly from Lemmas~\ref{l8.6} and \ref{l8.1} as soon as we observe that
\begin{equation}\label{eq8.32}
\varphi(L)=\left(L^{\frac{\beta-\alpha}{2}+\frac 12\left(\frac
np-\frac
nq\right)}\varphi(L)\right)\,L^{-\frac{\beta-\alpha}{2}-\frac
12\left(\frac np-\frac nq\right)},
\end{equation}

\noindent and by our assumptions the function $z\mapsto z^{\frac{\beta-\alpha}{2}+\frac
12\left(\frac np-\frac nq\right)}\varphi(z)$ belongs to $H^\infty(\Sigma_\mu^0)$.\ep

\begin{lemma}\label{l8.8}
For all $0\leq s_0,s_1\leq 1$ and $0<
p_0,p_1\leq 2$
\begin{equation}\label{eq8.33}
\left[H^{s_0,p_0}_L(\RR^n),H^{s_1,p_1}_L(\RR^n)\right]_\theta=H^{s,p}_L(\RR^n),
\quad 0<\theta<1,
\end{equation}

\noindent where $s=(1-\theta)s_0+\theta s_1$ and
$\frac{1}{p}=\frac{1-\theta}{p_0}+\frac\theta{p_1}$.
\end{lemma}

\bp Similarly to the case $s_0=s_1=0$ we prove (\ref{eq8.33}) via
the reduction to the interpolation of tent spaces, this time, using
the weighted tent spaces discussed in Section~\ref{s8.3}. Recall the
operators $Q_{\psi}$ and $\pi_{\psi}$ introduced in (\ref{eq4.5}) and (\ref{eq4.6}), respectively. Let $\mu\in(\omega,\pi/2)$.
Using Proposition~\ref{p4.3}, Lemma~\ref{l8.6} and the fact that multiplication by $t^{-s}$ is an isomorphism from
$T_s^{p,2}(\RR^{n+1}_+)$ onto $T^p(\RR^{n+1}_+)$, we can verify that
\begin{equation}\label{eq8.34}
Q_{\psi}:H^{s,p}_L(\RR^n)\longrightarrow T_s^{p,2}(\RR^{n+1}_+),\quad
\mbox{and}\quad \pi_{\widetilde\psi}:T_s^{p,2}(\RR^{n+1}_+)\longrightarrow
H^{s,p}_L(\RR^n),
\end{equation}
\noindent for $\psi\in \Psi_{\alpha,\beta}(\Sigma_\mu^0)$ and $\widetilde\psi\in \Psi_{\beta,\alpha}(\Sigma_\mu^0)$,
where $\alpha>\frac s2$ and $\beta>\frac{n}{2}\left(\max\{\frac
1p,1\}-\frac 12\right)-\frac s2$.

Now for any given $(s_0,p_0)$ and $(s_1,p_1)$, $s_0\leq s_1$, we choose $\psi\in \Psi_{\alpha,\beta}(\Sigma_\mu^0)$ and $\widetilde\psi\in \Psi_{\beta,\alpha}(\Sigma_\mu^0)$, where $\alpha>\frac {s_1}{2}$ and $\beta>\frac{n}{2}\left(\max\{\frac
1p,1\}-\frac 12\right)-\frac {s_1}{2}$. Then the corresponding $Q_{\psi}$ and $\pi_{\widetilde\psi}$ satisfy (\ref{eq8.34}) for all $s_0\leq s\leq s_1$ and all $p$ between $p_0$ and $p_1$. The rest of the proof follows the same lines as the proof of Lemma~\ref{l4.5}.   We omit the remaining details, except to note that by Lemma \ref{l8.6} and Theorem \ref{t7.1},  the $H^{s,p}_L$ spaces under consideration embed into $H^p_L$ or $\Lambda_L^\alpha$ spaces falling under the scope of
Proposition \ref{p10.1} below, and thus may all be embedded into a common ambient Banach
space. 
\ep

\subsection{Hardy-Sobolev spaces associated to $L$: identifications with classical scales}\label{s8.5}

\begin{proposition}\label{p8.9} For every $p$ such that $\frac{p_-(L)n}{n+p_-(L)}<p\leq 2$
\begin{equation}\label{eq8.35}
H^{1,p}_L(\RR^n)\approx \dot F^{p,2}_1(\RR^n).
 \end{equation}
\end{proposition}

\bp {\bf Step I}. First, we would like to show that
\begin{equation}\label{eq8.36}
H^{1,p}_L(\RR^n)\hookrightarrow \dot F^{p,2}_1(\RR^n),
\quad\mbox{for \quad $\frac{n}{n+1}<p\leq 2$ }.
 \end{equation}

By Propositions~\ref{p5.2}, \ref{p5.4} and (\ref{eq8.10}),
(\ref{eq8.12}) we have
\begin{equation}\label{eq8.37}
\nabla L^{-1/2}: H^p_L(\RR^n)\longrightarrow \dot F_0^{p,2}(\RR^n),
\quad\mbox{if \quad $\frac{n}{n+1}<p<2+\eps(L)$ }.
\end{equation}

\noindent On the other hand, according to Lemma~\ref{l8.6} the
operator $L^{1/2}$ is an isomorphism of $H^{1,p}_L$ onto $H^p_L$ for
$0<p\leq 2$. Hence, $L^{1/2}g\in H^p_L$ for every $g\in H^{1,p}_L$,
and
\begin{equation}\label{eq8.38}
\|\nabla g\|_{\dot F_0^{p,2}(\RR^n)}= \|\nabla L^{-1/2}
L^{1/2}g\|_{\dot F_0^{p,2}(\RR^n)}\leq C
\|L^{1/2}g\|_{H^p_L(\RR^n)}\leq C
\|g\|_{H^{1,p}_L(\RR^n)},\,\,\forall\,g\in H^{1,p}_L,
\end{equation}

\noindent if  $\frac{n}{n+1}<p<2+\eps(L)$. This gives the desired norm
estimate (see Lemma~\ref{l8.3}). It remains to show that  the
elements of $H^{1,p}_L(\RR^n)$ can be seen as tempered distributions
modulo polynomials.

Indeed, by (\ref{eq8.26}) for every $g\in H^{1,p}_L$ there is a
sequence $\{g_n\}_{n=1}^{\infty}\subset \dot W^{1,2}\cap H^{1,p}_L$
converging to $g$ in $H^{1,p}_L$ norm. Then

\begin{equation}\label{eq8.39}
\{g_n\}_{n=1}^{\infty}\subset \dot W^{1,2}\approx \dot
F_1^{2,2}\subset {\mathcal S}'/{\mathcal P}, \end{equation}

\noindent  in particular, $g_n$, $n=1, 2,...$, are tempered
distributions modulo polynomials. Also, $\{g_n\}_{n=1}^{\infty}$ is
a Cauchy sequence in $H^{1,p}_L$ norm. Hence,

\begin{equation}\label{eq8.40}
\mbox{$\{\nabla g_n\}_{n=1}^{\infty}$ is Cauchy in $\dot F_0^{p,2}$
norm} \end{equation}

\noindent by (\ref{eq8.38}). Combining (\ref{eq8.39}),
(\ref{eq8.40}) and Lemma~\ref{l8.3}, we conclude that
$\{g_n\}_{n=1}^{\infty}\subset \dot F_1^{p,2}$ and
$\{g_n\}_{n=1}^{\infty}$ is Cauchy in $\dot F_1^{p,2}$. Now $g$ can
be identified with the limit of $\{g_n\}$ in $\dot F_1^{p,2}$.

\vskip 0.08in {\bf{Step II}}. Now we concentrate on the inverse
inclusion, and show that
\begin{equation}\label{eq8.41}
H^{1,p}_L(\RR^n)\hookleftarrow \dot F^{p,2}_1(\RR^n), \quad\mbox{for
\quad $\frac{p_-(L)n}{n+p_-(L)}<p\leq 2$}.
 \end{equation}

It follows from (\ref{eq5.20}), (\ref{eq5.38}),
(\ref{eq8.11}) and (\ref{eq8.13})  that
\begin{equation}\label{eq8.42}
S_1\sqrt L: \dot F^{p,2}_1(\RR^n)\longrightarrow L^p(\RR^n),
\quad\mbox{for \quad $\frac{p_-(L)n}{n+p_-(L)}<p\leq 2$}.
 \end{equation}

\noindent Combining this with (\ref{eq5.19}) we have
\begin{equation}\label{eq8.43}
\|f\|_{H^{1,p}_L(\RR^n)}\leq C \|f\|_{\dot F^{p,2}_1(\RR^n)},
\,\,\forall \,\,f\in  \dot F^{p,2}_1(\RR^n), \quad \mbox{
$\frac{p_-(L)n}{n+p_-(L)}<p\leq 2$},
 \end{equation}

\noindent and it remains to show that $f$ actually belongs to
$H^{1,p}_L(\RR^n)$, i.e., that it can be approximated by the
elements of $L^{-1/2}(L^2\cap H^{p}_L)$.

According to Lemma~\ref{l8.2}, ${\mathcal Z}(\RR^n)$ is a dense
subset of $\dot F^{p,2}_1(\RR^n)$. Then every $f$ in (\ref{eq8.43})
can be approximated  in $\dot F^{p,2}_1(\RR^n)$ norm by a sequence
$\{f_n\}_{n=1}^{\infty}\subset {\mathcal Z}(\RR^n)$. The operator
$\sqrt L$ maps $\dot W^{1,2}\approx \dot F_1^{2,2}$ to $L^2(\RR^n)$
and ${\mathcal Z}(\RR^n)$ is a subset of $\dot F_1^{2,2}$.
Hence,
\begin{equation}\label{eq8.44}
\sqrt L f_n\in L^2(\RR^n), \quad n=1,2,....
 \end{equation}

\noindent Since, in addition, $\|\sqrt Lf_n\|_{H^p_L(\RR^n)}$ is
finite for every $n=1,2,...$ by (\ref{eq8.42}), we can conclude that
$\{\sqrt L f_n\}_{n=1}^{\infty}\subset L^2\cap H^p_L$ and therefore,
$\{f_n\}_{n=1}^{\infty}\subset L^{-1/2}(L^2\cap H^p_L)$.

By our assumptions $\{f_n\}_{n=1}^{\infty}$ is Cauchy in $\dot
F^{p,2}_1(\RR^n)$ norm. Then by (\ref{eq8.42}), it is also Cauchy in
$H^{1,p}_L(\RR^n)$ norm and belongs to $L^{-1/2}(L^2\cap H^p_L)$.
Now we identify its limit in $H^{1,p}_L(\RR^n)$ with $f\in \dot
F^{p,2}_1(\RR^n)$, and derive (\ref{eq8.41}) with the appropriate
norm estimate. \ep

%\noindent {\it Remark.}\, At this point, let us discuss in details the statement \eqref{eq1.17}. Assume, for instance, that $f\in L^2(\RR^n)$, and let us start with the equivalence of the second and the third term. 

\subsection{Functional calculus and fractional powers of $L$ in Sobolev and regular Hardy spaces}\label{s8.6}

In this section we restrict ourselves to the case $n\geq 3$. One can derive analogues of all the results below for $n=2$ following the same arguments. We will not state them for the sake of brevity.

\begin{theorem}\label{t8.10} Let $L$ be an elliptic operator satisfying (\ref{eq1.1})--(\ref{eq1.3}), and let $p(L)$ and $\eps(L)$ retain the same significance as before. Assume that $-1\leq s \leq 1$ and $0<p<\infty$
are such that either of the conditions (1) or (2) below is satisfied
\begin{eqnarray}
&& {\rm(1)}\,\,-1\leq s\leq 0\quad\mbox{and}\nonumber\\[4pt]
&& \label{eq8.45} \qquad\qquad \mbox{ $\max\left\{0,\frac 1n \,s+1-\frac{1}{p_-(L^*)}\right\}<\frac 1p<\left(\frac{1}{2+\eps(L^*)}-1+\frac{1}{p_-(L)}\right)s +\frac{1}{p_-(L)}$},\end{eqnarray}\begin{eqnarray}
&&\hskip -0.7cm
\label{eq8.46}
{\rm(2)}\,\,\,\,0\leq s\leq 1\quad\mbox{and}\nonumber\\[4pt]
&& \qquad\qquad \mbox{ $\left(\frac{1}{2+\eps(L)}-1+\frac{1}{p_-(L^*)}\right)s +1-\frac{1}{p_-(L^*)}<\frac 1p<\frac 1n \, s+\frac{1}{p_-(L)}$}.
\end{eqnarray}

\noindent Then $L$ has bounded holomorphic functional calculus in $\dot F_{s}^{p,2}(\RR^n)$,
in the sense that for every $\varphi\in
H^{\infty}(\Sigma_\mu^0)$
\begin{equation}\label{eq8.47}
\varphi(L):\dot F_{s}^{p,2}(\RR^n)\longrightarrow \dot F_{s}^{p,2}(\RR^n),
\end{equation}

\noindent with the norm bounded by $\|\varphi\|_{L^\infty(\Sigma_\mu^0)}$.

Moreover, for every $\varphi\in
\Psi'(\Sigma_\mu^0)$

\begin{equation}\label{eq8.48}
\varphi(L):\dot F_{\alpha}^{p,2}(\RR^n)\longrightarrow \dot F_{\beta}^{q,2}(\RR^n),
\end{equation}

\noindent and
\begin{equation}\label{eq8.49}
\|\varphi(L)f\|_{\dot F_{\beta}^{q,2}(\RR^n)}\leq C
\left\|z^{\frac{\beta-\alpha}{2}+\frac 12\left(\frac np-\frac
nq\right)}\varphi\right\|_{L^\infty(\Sigma_\mu^0)}\|f\|_{\dot F_{\alpha}^{p,2}(\RR^n)},
\end{equation}

\noindent whenever $p\leq q$,  and the pairs $(\alpha,1/p)$, $(\beta,1/q)$ satisfy (\ref{eq8.45}) or (\ref{eq8.46}).
 %i.e.,
%at least one of the conditions (\ref{eq8.45}), (\ref{eq8.46}) holds with $s=\alpha$ and at least one of the conditions  (\ref{eq8.45}), (\ref{eq8.46})  holds with $s=\beta$, $p=q$.
\end{theorem}

 \noindent {\it Remark.} Above, the expression ``the pair $(\alpha,1/p)$ satisfies (\ref{eq8.45}) or (\ref{eq8.46})" means that either of the conditions (\ref{eq8.45}), (\ref{eq8.46}) holds with $\alpha$ in place of $s$. Similarly, ``the pair $(\beta,1/q)$ satisfies (\ref{eq8.45}) or (\ref{eq8.46})" means that either of the conditions (\ref{eq8.45}), (\ref{eq8.46}) holds with $\beta$ in place of $s$ and $q$ in place of $p$. Finally, the expression ``the pairs $(\alpha,1/p)$, $(\beta,1/q)$ satisfy (\ref{eq8.45}) or (\ref{eq8.46})" means that both ``the pair $(\alpha,1/p)$ satisfies (\ref{eq8.45}) or (\ref{eq8.46})" and ``the pair $(\beta,1/q)$ satisfies (\ref{eq8.45}) or (\ref{eq8.46})", in the sense outlined above.

\vskip 0.08in

The range of $s$ and $p$ satisfying either  (\ref{eq8.45}) or (\ref{eq8.46}) can be identified with a polygon on the $(s,1/p)$ plane. The shape of such a polygon depends on whether $\frac{n+p_-(L^*)}{np_-(L^*)}<1$ (in which case we will denote the corresponding polygon by ${\mathcal R}_1(L)$) or $\frac{n+p_-(L^*)}{np_-(L^*)}\geq 1$ (then the polygon will be denoted by ${\mathcal R}_2(L)$).

 First, assume that $\frac{n+p_-(L^*)}{np_-(L^*)}<1$. The region ${\mathcal R}_1(L)$ consists of the open polygon with vertices
\begin{equation}\label{eq8.50}
\begin{array}{l}
 B_L=\left(-1,1-\frac
{1}{2+\eps(L^*)}\right),\qquad\qquad\,
E_L=\left(0,\frac{1}{p_-(L)}\right), \quad\quad\,\,
C_L=\left(1,\frac{n+p_-(L)}{np_-(L)}\right),\\[4pt]
 A_L=\left(-1,1-\frac{n+p_-(L^*)}{np_-(L^*)}\right),\qquad
F_L=\left(0,1-\frac{1}{p_-(L^*)}\right),\qquad D_L=\left(1,\frac
{1}{2+\eps(L)}\right),
\end{array}
\end{equation}
\noindent together with the sides $A_LB_L$ and $C_LD_L$. It is shown on Figure~3.

%%%%%%%%%%%%%%%%%%%%%%%%%%%%%%%%%%%%%%%%%%%%%%%% picture 2 %%%%%%%%%%%%%%%%%
\setlength{\unitlength}{0.4 cm}

\begin{picture}(14,19)(-7,-3)

\thinlines
\put(-4,1){\line(1,0){24}}    % x-Coordinate axis
\put(8,0){\line(0,1){14}}    % y-Coordinate axis

\put(7.92,13.8){\vector(0,1){0}}     %
\put(8.08,13.8){\vector(0,1){0}}      %  Arrow on the y-Coordinate axis
\put(8,14){\vector(0,1){0}}           %  (made up of 3 triangles)

\put(19.8,0.92){\vector(1,0){0}}     %
\put(19.8,1.08){\vector(1,0){0}}      %  Arrow on the x-Coordinate axis
\put(20,1){\vector(1,0){0}}           %  (made up of 3 triangles)

\multiput(8.00,7.0)(0.5,0){21}{\line(-1,0){0.25}}    % horizontal dashed line
\multiput(8.00,9.91)(0.5,0){21}{\line(-1,0){0.25}}    % horizontal dashed line
%\multiput(8.00,11.50)(0.5,0){20}{\line(1,0){0.25}}    % horizontal dashed line

\multiput(8.00,7.0)(-0.5,0){20}{\line(-1,0){0.25}}    % horizontal dashed line
\multiput(8.00,4.2)(-0.5,0){20}{\line(-1,0){0.25}}    % horizontal dashed line
%\multiput(8.00,2.50)(-0.5,0){20}{\line(-1,0){0.25}}    % horizontal dashed line

\put(8.00,12.00){\circle*{0.2}} \put(8.00,7.00){\circle*{0.2}}

\put(18.00,1){\circle*{0.2}} \put(-2.00,1.00){\circle*{0.2}}

%vertices
%high
\put(8.00,9.91){\circle*{0.2}} \put(-2.00,7.45){\circle*{0.3}}
\put(18.00,6.55){\circle*{0.3}} \put(18.00,9.91){\circle*{0.3}}

%low
\put(-2.00,2.5){\circle*{0.3}} \put(18.00,11.5){\circle*{0.3}}
\put(8.00,4.1){\circle*{0.2}} \put(-2.00,4.2){\circle*{0.3}}

%labels on axis
\put(8.1,7.5){\scriptsize{{$\frac{\rm 1}{\rm 2}$}}}
\put(-2,0){\rm\scriptsize -1} \put(18,0){\rm \scriptsize 1}
\put(8.1,12){\rm \scriptsize 1}
\put(8.15,13.5){$\textstyle{{\frac{1}{p}}}$}
\put(19.5,0.25){$\textstyle{s}$}

%labels with formulas
\put(16,12){\scriptsize{{$C_L=\left(1,\frac{n+p_-(L)}{np_-(L)}\right)$}}}
\put(-4,1.5){\scriptsize{{$A_L=\left(-1,1-\frac{n+p_-(L^*)}{np_-(L^*)}\right)$}}}
\put(4.3,10.3){\scriptsize{{$E_L=\left(0,\frac{1}{p_-(L)}\right)$}}}
\put(8.1,3.5){\scriptsize{{$F_L=\left(0,1-\frac{1}{p_-(L^*)}\right)$}}}
\put(-4,8.3){\scriptsize{{$B_L=\left(-1,\frac
{1+\eps(L^*)}{2+\eps(L^*)}\right)$}}}
\put(16,5){\scriptsize{{$D_L=\left(1,\frac
{1}{2+\eps(L)}\right)$}}}
\put(18.4,9.7){\scriptsize{{$H_L$}}} \put(-3,4){\scriptsize{{$G_L$}}}

\linethickness{0.12\unitlength} \thicklines
\put(-2,2.5){\line(0,1){5}}            % vertical1
\put(18,6.50){\line(0,1){5}}            % vertical1

\thicklines
\put(-2,7.44){\line(4,1){10}}            % high slant
\put(8,9.91){\line(6,1){10}}            % high slant
\put(-2,2.50){\line(6,1){10}}            % low slant
\put(8,4.1){\line(4,1){10}}            % low slant

\put(-3.00,-1.25){{\bf Figure 3 -- the region ${\mathcal R}_1(L)$.}}

\end{picture}
%%%%%%%%%%%%%%%%%%%%%%%%%%%%%%%%%%%%%%%% end picture %%%%%%%%%%%%%%%%%%%%%%%%

For the case $\frac{n+p_-(L^*)}{np_-(L^*)}\geq 1$ we define the second region, ${\mathcal R}_2(L)$, as an open polygon with the vertices
\begin{equation}\label{eq8.51}
\begin{array}{l}
 B_L=\left(-1,\frac
{1+\eps(L^*)}{2+\eps(L^*)}\right),\qquad\qquad\,\,
E_L=\left(0,\frac{1}{p_-(L)}\right), \quad\quad\,\,
C_L=\left(1,\frac{n+p_-(L)}{np_-(L)}\right),\\[4pt]\quad
\widetilde A_L=(-1,0),\qquad \quad\widetilde F_L=\left(0,\frac{n}{p_-(L^*)}-n\right),\\[4pt]\quad
F_L=\left(0,1-\frac{1}{p_-(L^*)}\right),\quad D_L=\left(1,\frac
{1}{2+\eps(L)}\right),
\end{array}
\end{equation}
\noindent together with the sides $\widetilde A_LB_L$ and $C_LD_L$. Its picture is a modified version of Figure~3, much as Figure~2 is a modification of Figure~1.

\noindent {\it Proof of Theorem~\ref{t8.10}}. Let us introduce auxiliary points $O=(0,0)$, $B=(-1,1/2)$ and $D=(1,1/2)$.
As we already mentioned, the statement of Theorem~\ref{t8.10} was proved for all $p,q$ which in addition to the aforementioned restrictions
satisfy $p,q>p_-(L)$ (see \cite{AuscherSurvey}, Sections~5.3, 5.4). Thus, the interior of the polygon $G_LB_LE_LH_LD_LF_L$ is already covered (i.e. the statement of the Theorem holds with ${\mathcal R}_1$ substituted by $G_LB_LE_LH_LD_LF_L$). The same argument applies to the segments $G_LB_L$ and $H_LD_L$.

Next, (\ref{eq1.13}), (\ref{eq8.35}) and Lemma~\ref{l8.8} together with the well-known results on the complex interpolation of Triebel-Lizorkin spaces lead to the conclusion that $H^{s,p}_L(\RR^n)\approx \dot F_s^{p,2}(\RR^n)$ whenever $(s,1/p)$ belongs to $OE_LC_LD$ or the segment $C_LD$.
Then, by Lemma~\ref{l8.7}, the statement of the theorem holds in $OE_LC_LD$ and on the segment $C_LD$.

Combining these observations, we recover the result on the entire ${\mathcal R}_1$ or ${\mathcal R}_2$ using duality and interpolation. \ep

\vskip 0.08in \noindent {\it Remark.} When $\frac{n+p_-(L^*)}{np_-(L^*)}\geq 1$, then Theorem~\ref{t8.10} can be complemented by the corresponding results for $p=\infty$. Specifically, consider the spaces $\dot F_{\alpha}^{\infty,\infty}$. For $-1\leq \alpha <0$ they can be seen, e.g.,  as the dual spaces for $\dot F_{1}^{p,2}(\RR^n)$ with $p=\frac{n}{n+\alpha +1}$ (see, e.g., \cite{FrJa}, Remark 5.14, and references therein). Then for every $\varphi\in
H^{\infty}(\Sigma_\mu^0)$
\begin{equation}\label{eq8.52}
\varphi(L):\dot F_{\alpha}^{\infty,\infty}(\RR^n)\longrightarrow \dot F_{\alpha}^{\infty,\infty}(\RR^n),
\end{equation}

\noindent whenever $-1\leq\alpha <n\left(\frac{1}{p_-(L^*)}-1\right)$.  In the same way the spaces $\dot F_{\alpha}^{\infty,\infty}$ can be incorporated in (\ref{eq8.48})--(\ref{eq8.49}), that is, we can say that (\ref{eq8.48})--(\ref{eq8.49}) hold whenever $p\leq q$ and $(\alpha,1/p)$, $(\beta,1/q)$ belong to $\widetilde {\mathcal R}_2={\mathcal R}_2\cup \widetilde A_L \widetilde F_L$, where the segment $\widetilde A_L \widetilde F_L$ corresponds to the classes $\dot F_{s}^{\infty,\infty}$.

The Theorem~\ref{t8.10} and sharpness results in Section~\ref{s2.2} lead to the complete description of all functions spaces on Hardy-Sobolev-Triebel-Lizorkin scale where one can develop functional calculus for an arbitrary elliptic operator satisfying (\ref{eq1.1})--(\ref{eq1.3}).

\begin{corollary}\label{c8.11} Let $L$ be an elliptic operator satisfying (\ref{eq1.1})--(\ref{eq1.3}), and assume that $s\in\RR$, $p\in (0,\infty)$ are such that
\begin{equation}\label{eq8.53}
-1\leq s\leq 1\quad\mbox{and\quad $\max\left\{0,\frac 1n\, s+\frac{n-2}{2n}\right\}\leq \frac 1p \leq \frac 1n\, s+\frac{n+2}{2n}$}.
\end{equation}

\noindent
Then $L$ has a bounded holomorphic functional calculus in $\dot F_{s}^{p,2}(\RR^n)$, in the sense that
\begin{equation}\label{eq8.54}
\varphi(L):\dot F_{s}^{p,2}(\RR^n)\longrightarrow \dot F_{s}^{p,2}(\RR^n),\quad \mbox{for every $\varphi\in
H^{\infty}(\Sigma_\mu^0)$},
\end{equation}

\noindent with the norm bounded by $\|\varphi\|_{L^\infty(\Sigma_\mu^0)}$.

More generally, if $0<p\leq q<\infty$ and the pairs $\alpha,p$ and $\beta,q$ satisfy (\ref{eq8.53}), i.e.
\begin{eqnarray}\label{eq8.55}
&&-1\leq \alpha\leq 1\quad\mbox{and\quad $\max\left\{0,\frac 1n\, \alpha+\frac{n-2}{2n}\right\}\leq \frac 1p \leq \frac 1n\, \alpha+\frac{n+2}{2n}$},\\[4pt]
\label{eq8.56}
&&-1\leq \beta\leq 1\quad\mbox{and\quad $\max\left\{0,\frac 1n\, \beta+\frac{n-2}{2n}\right\}\leq \frac 1q \leq \frac 1n\, \beta +\frac{n+2}{2n}$},
\end{eqnarray}

\noindent then
\begin{equation}\label{eq8.57}
\varphi(L):\dot F_{\alpha}^{p,2}(\RR^n)\longrightarrow \dot F_{\beta}^{q,2}(\RR^n),
\end{equation}

\noindent with
\begin{equation}\label{eq8.58}
\|\varphi(L)f\|_{\dot F_{\beta}^{q,2}(\RR^n)}\leq C
\left\|z^{\frac{\beta-\alpha}{2}+\frac 12\left(\frac np-\frac
nq\right)}\varphi\right\|_{L^\infty(\Sigma_\mu^0)}\|f\|_{\dot F_{\alpha}^{p,2}(\RR^n)},
\end{equation}

\noindent for every $\varphi\in
\Psi'(\Sigma_\mu^0)$ such that the $L^\infty$ norm on the right-hand side of (\ref{eq8.58}) is
finite.

 These result are sharp for all $n\geq 3$. For every $-1\leq s\leq 1$, $0<p<\infty$ not satisfying (\ref{eq8.53})
there exists an elliptic operator $L$ such that the heat semigroup is not bounded in $\dot F_{s}^{p,2}(\RR^n)$ and hence, the property (\ref{eq8.54})
does not hold. Similarly, (\ref{eq8.57}), (\ref{eq8.58}) need not hold if $\alpha,p$ or $\beta,q$ do not satisfy (\ref{eq8.55})--(\ref{eq8.56}).
\end{corollary}

The Corollary~\ref{c8.11} extends to the case $p=\infty$ in the vein of remark after the proof of Theorem~\ref{t8.10}.

As we mentioned in the introduction, 
the range of indices $s$ and $p$ satisfying (\ref{eq8.53}) can be described as a region on $(s,1/p)$ plane.

Assume first that $n\geq 4$. We denote by ${\mathcal R}_1$ a {\it closed} polygon on $(s,1/p)$ plane with vertices at
\begin{equation}\label{eq8.59}
\begin{array}{l}
 A=\left(-1,\frac {n-4}{2n}\right),\qquad\qquad\,\,
B=\left(-1,\frac 12\right), \\[4pt]
C=\left(1,\frac {n+4}{2n}\right),\qquad\qquad\,\,
D=\left(1,\frac 12\right).
\end{array}
\end{equation}

\noindent On an $(s,1/p)$ plane the Region~${\mathcal R}_1$ is shown on Figure~1.

Now let $n\leq 4$, and let ${\mathcal R}_2$ be a {\it closed} polygon on $(s,1/p)$ plane with vertices at
\begin{equation}\label{eq8.60}
\begin{array}{l}
\widetilde A=\left(-1,0\right),\qquad\qquad\,\,
B=\left(-1,\frac 12\right), \\[4pt]
C=\left(1,\frac {n+4}{2n}\right),\qquad\qquad\,\,
D=\left(1,\frac 12\right),\\[4pt]
\widetilde F=\left(\frac{2-n}{2},0\right).
\end{array}
\end{equation}

\noindent The region ${\mathcal R}_2$ is depicted on Figure~2.

Observe that for $n=4$ we have ${\mathcal R}_1={\mathcal R}_2$, and the corresponding picture can be seen as an extreme case of ${\mathcal R}_1$ (with $A=(-1,0)$ and $C=(1,1)$) or an extreme case of ${\mathcal R}_2$ (with $\widetilde A=\widetilde F=(-1,0)$). 

In general, as dimension decreases, the slope of the line $BC$ becomes larger, while $B$ is fixed and $C$ moves up along the line $\{s=1\}$. When $n=4$, $C=(1,1)$ and for $n\leq 3$ the point $C$ corresponds to $p<1$. Strictly speaking, the Figure~2 shows ${\mathcal R}_2$ for $n=3$, and as we mentioned above, $n=4$ is its extreme case.

All in all, $s\in [-1,1]$ and $p\in (0,\infty]$ satisfy (\ref{eq8.53}) if and only if the point $(s,1/p)$ belongs to ${\mathcal R}_1$ ($n\geq 4$) or  to ${\mathcal R}_2$ ($n\leq 4$). As before, the segment $\widetilde A\widetilde F$ corresponds to the spaces $\dot F_s^{\infty,\infty}$.

\vskip 0.08in \noindent {\it Proof of Corollary~\ref{c8.11}}.\, The Corollary follows from Theorem~\ref{t8.10} and the fact that $p_-(L)<\frac{2n}{n+2}$ for every elliptic operator $L$. The sharpness is a consequence of Proposition~\ref{p2.1}. Indeed, if $n\geq 4$ and for some point $(s_0,1/p_0)\not\in{\mathcal R}_1$ the heat semigroup $e^{-tL}$, $t>0$, is bounded in $\dot F_{s_0}^{p_0,2}(\RR^n)$ for all elliptic operators $L$, then by interpolation the heat semigroup is bounded in all $\dot F_{s}^{p,2}(\RR^n)$ with $(s,1/p)$ in the linear span of $(s_0,1/p_0)$ and ${\mathcal R}_1$. In particular, there exists $p\not\in \left[\frac{2n}{n+2},\frac{2n}{n-2}\right]$
such that the heat semigroup is bounded in $L^p$ for all $L$, which contradicts Proposition~\ref{p2.1}. Similarly, when $n=3$, we discover such a contradiction starting with any $(s_0,1/p_0)\not\in{\mathcal R}_2$.\ep

\section{Appendix 1:  Relationships between $H^p_L$ and classical $H^p$}\label{s9}

In this Appendix, we establish  (\ref{eq1.13}) - (\ref{eq1.13b}).   We note that the containments
in \eqref{eq1.13a} (resp. \eqref{eq1.13b}) are strict if $1< p_-(L)$  (resp. $p_+(L)<\infty).$  For
example, see item (vi) in Proposition \ref{embed} below, and its proof.

We recall that classical $H^p(\RR^n ) = L^p(\RR^n),$
if $1<p<\infty$, that
$(p_-(L),p_+(L))$ is the interior of the interval of $L^p$ boundedness
of the heat semigroup $e^{-tL}$,
and that $p_-(L)<2n/(n+2)$ and $p_+(L) > 2n/(n-2)$, if $n>2$.  
For $\alpha>0$, 
we let $\Lambda^\alpha(\RR^n)$
denote the classical homogeneous ``Lip$_\alpha$" spaces (cf. (\ref{eq9.lip}) below), 
and in the case $\alpha = 0$,
we let $\Lambda^0(\RR^n),\, \Lambda^0_L(\RR^n)$ denote, respectively, the classical and $L$-adapted
BMO spaces $BMO(\RR^n)$ and $BMO_L(\RR^n)$.   We define null spaces
$$\mathcal{N}_p(L):=\{f\in L^p(\RR^n) \cap W^{1,2}_{loc}: Lf=0\}, \quad p_+(L)\leq p<\infty,$$
and
$$\mathcal{N}_\alpha(L):=\{\varphi\in \Lambda^\alpha(\RR^n) 
\cap W^{1,2}_{loc}: L\,\varphi=0\}, \quad 0\leq \alpha.$$
\begin{proposition}\label{embed}
We have the following containments and continuous embeddings:
\begin{enumerate}
\item[(i)] $L^2(\RR^n)\cap \HPL  \subset L^2(\RR^n)\cap H^p(\RR^n),\quad$ $n/(n+1)<p\leq 1,$ and
\begin{equation}\label{eq9.4}
\|f\|_{H^p(\RR^n)} \lesssim \|f\|_{\HPL}\,, \qquad f\in L^2(\RR^n) \cap \HPL.
\end{equation}
\item[(ii)] $L^2(\RR^n)\cap \HPL  \subset L^2(\RR^n)\cap L^p(\RR^n),\quad$ $1<p\leq p_-(L),$
and
\begin{equation}\label{eq9.3**}\|f\|_{L^p(\RR^n)}\lesssim\|f\|_{\HPL}\,,
\qquad f\in L^2(\RR^n) \cap \HPL.\end{equation}
\item[(iii)] $L^p(\RR^n)/\mathcal{N}_p(L)\hookrightarrow \HPL,\quad$  $p_+(L) \leq p<\infty,$ and
\begin{equation}\label{eq9.lpnorm}\|f\|_{\HPL}\leq C\|f\|_{L^p(\RR^n)},\quad p_+(L) \leq p<\infty.
\end{equation}
\item[(iv)] $\Lambda^\alpha(\RR^n)/\mathcal{N}_\alpha(L)
 \hookrightarrow \Lambda_L^\alpha(\RR^n),\quad$ 
$0\leq \alpha < 1$,\footnote{In the presence of pointwise heat kernel bounds,
the case $\alpha = 0$ of (iv) was previously obtained in \cite{DL}.} and
\begin{equation}\label{eq9.alphanorm}
\|\varphi\|_{\Lambda_L^\alpha(\RR^n)}\leq C\|\varphi\|_{\Lambda^\alpha(\RR^n)},\quad 0 \leq \alpha<1.\end{equation}
\end{enumerate}
Moreover,
\begin{enumerate}
\item[(v)] $\HPL = L^p(\RR^n),\quad$ $p_-(L)<p<p_+(L).$
\smallskip
\item[(vi)] $\HPL \neq L^p(\RR^n),\quad$ $1<p\leq p_-(L)$ or $p_+(L)\leq p<\infty.$
\end{enumerate}
Finally,  for each $p>2n/(n-2),\, n\geq 3$ (resp., for each $\alpha \in [0,1)$), 
there is an operator $L$ and a 
non-trivial $u \in L^p(\RR^n)$ (resp., $u\in \Lambda^\alpha(\RR^n)$) such that $Lu=0$
weakly in $\RR^n$.  Thus, for each such $p$ or $\alpha$,
there is an operator $L$ for which the corresponding null space $\mathcal{N}_p(L)$ or
$\mathcal{N}_\alpha(L)$ is non-trivial.
\end{proposition}
\bp
We carry out the proof in the following order:  (iv), (v), (iii), (i), (ii), (vi) and then conclude by
presenting  examples of non-trivial global null solutions.

\smallskip

\noindent {\it Proof of (iv).} 
Fix $\varphi \in \Lambda^\alpha,\, 0\leq \alpha <1$.  
By definition, for $n/(n+1)<p\leq1$ (as is the case if $0\leq \alpha = n(p^{-1}-1)<1$), 
an $\HPL$-molecule is, in particular,
a classical $H^p(\RR^n)$-molecule (since the operator $L$ kills constants).   
Consequently, by the classical duality results \cite{FeSt,DRS} we have that 
$\varphi \in {\bf M}^{M,\,*}_{\alpha,L^*}\,$, the ambient space
in which $\Lambda_{L}^{\alpha}$ is defined (cf. (\ref{eq1.19}) and the related discussion,
bearing in mind that in our present context, the roles of $L$ and $L^*$ have been reversed).
Also, $\|\varphi\|_{\Lambda_{L}^{\alpha}} = 0$ 
for $\varphi \in \mathcal{N}_\alpha(L)$, by definition of the $\Lambda_{L}^{\alpha}$
norm (cf. (\ref{eq1.20}), but with 
$L$ in place of $L^*$).  Thus,
to prove (iv), it suffices to show that $\varphi$ 
satisfies the norm estimate \eqref{eq9.alphanorm}.  
To this end, we fix a cube $Q
\subset \RR^n,$ and use the fact that $e^{-tL}1=1$ to write
\begin{multline*}\frac{1}{|Q|^{\alpha/n}}\left(\frac{1}{|Q|}\int_Q
\left|(I-e^{-l(Q)^2{L}})^M\varphi(x)\right|^2\,dx\right)^{1/2}\\
= \frac{1}{|Q|^{\alpha/n}}\left(\frac{1}{|Q|}\int_Q
\left|(I-e^{-l(Q)^2{L}})^M(\varphi-\varphi_Q)(x)\right|^2\,dx\right)^{1/2},\end{multline*}
where  $\varphi_Q:= \fint_Q \varphi$.  It is then a routine matter to verify that
this last expression is bounded uniformly in $Q$ by either $\|\varphi\|_{BMO}$ (if $\alpha = 0$), or by
\begin{equation}\label{eq9.lip}
\|\varphi\|_{\Lambda^\alpha(\RR^n)}:= \sup_{x\neq y}\frac{|\varphi(x)-\varphi(y)|}{|x-y|^\alpha}
\end{equation}
(if $0<\alpha<1$), using a dyadic  annular decomposition plus
the Gaffney estimates, much as in the proof of (\ref{eq6.2}).  We omit the details.

\smallskip

\noindent {\it Proof of (v).} Recall that $L^2(\RR^n)\cap\HPL$ is dense in $\HPL$ (by definition,
if $0<p\leq2$, and as proved in Corollary \ref{c4.4}, if $2<p<\infty$).  Of course, $L^2(\RR^n)\cap
L^p(\RR^n)$ is dense in $L^p(\RR^n)$.  Therefore, it is enough to show that
$L^2(\RR^n)\cap L^p(\RR^n) = L^2(\RR^n)\cap\HPL,$ with equivalence of norms.

One direction is easy:  fix $f \in L^2(\RR^n)\cap L^p(\RR^n),\, p_-(L)<p<p_+(L)$.  By
Corollary \ref{c4.4}, for appropriate $\psi$ we have that
\begin{equation}\label{ref1}
\|f\|_{\HPL} \approx  \left\|\left(\dint_{\Gamma(\cdot)}|\psi(t^2L)f(y)|^2\,\frac
{dydt}{t^{n+1}}\right)^{1/2}\right\|_{L^p(\RR^n)} \leq\, C\|f\|_{L^p(\RR^n)}\,,\end{equation}
where the last step essentially follows by the argument used in \cite{AuscherSurvey}, Theorem~6.1,  where the case
$\psi(z) = \sqrt ze^{-z}$ for the vertical (rather than conical) square function was treated. The appropriate modifications are fairly straightforward.

Conversely, suppose that $f \in L^2(\RR^n)\cap\HPL$, and let $g \in  L^2(\RR^n)\cap L^{p'}(\RR^n)$,
with $\|g\|_{L^{p'}(\RR^n)}=1$.   By the Calder\'{o}n reproducing formula
(\ref{eq4.17}), for appropriate $\psi,\, \widetilde{\psi}$ we have that
\begin{multline*}\left|\int_{\RR^n} f\, \overline{g}\,\right| = \left|\int_{\RR^n} 
\pi_{\widetilde\psi,L}\circ Q_{\psi,L} \,f\, \overline{g}\,\right|\\[4pt] \leq\,\,
\|Q_{\psi,L} \,f\|_{T^p(\RR^{n+1}_+)}\,\|Q_{\widetilde{\psi},L^*} \,g\|_{T^{p'}(\RR^{n+1}_+)}\,\,
\leq\,\,C\,
\|f\|_{\HPL}\,\|g\|_{ L^{p'}(\RR^n)},
\end{multline*}
where in the last step we have used (\ref{eq4.15}) and the square function bounds of
\cite{AuscherSurvey} (cf. the second inequality in \eqref{ref1} and the references thereafter).  The latter are applicable to the adjoint operator $L^*$ in $L^{p'}(\RR^n)$ 
since $p_+(L^*) = \left(p_-(L)\right)'$.  Taking the supremum over all such $g$, we obtain that
$$\|f\|_{L^{p}(\RR^n)} \leq C \|f\|_{\HPL},$$ as desired.

\smallskip

\noindent {\it Proof of (iii).} We
interpolate the inclusion map between 
$p=2$ and $p=\infty$ (i.e., $\alpha = 0$
in (iv)), to obtain \eqref{eq9.lpnorm}.  In turn, Theorem \ref{t6.1} implies that $\|f\|_{\HPL} = 0$ for $f \in 
\mathcal{N}_p(L)$, whence (iii) follows.

\smallskip

\noindent {\it Proof of (i).}
We suppose that $n/(n+1)<p\leq1$.  
As noted above, an $\HPL$-molecule is also
a classical $H^p(\RR^n)$-molecule, if $n/(n+1) <p\leq1$.  Consequently, by 
\eqref{hpequivalence} and
the molecular decomposition of classical $H^p$ spaces, we have that,
$L^2(\RR^n) \cap \HPL \subset L^2(\RR^n) \cap H^p(\RR^n)$ and \eqref{eq9.4} holds.

\smallskip

\noindent{\it Remark:} by the density of $L^2(\RR^n) \cap \HPL$ in $\HPL$,
one may now extend the identity map by continuity
to produce an 	``embedding" $\mathcal{J}: \HPL \to H^p(\RR^n),$ 
which equals the identity on $L^2(\RR^n) \cap \HPL$.  It remains an open question 
to determine whether, in general,
this embedding is necessarily 1-1. 

We further remark that, in the case $p=1$, the containment $L^2 \cap H^1_L \subset L^2 \cap H^1$
amounts to saying that, for $f \in L^2 \cap H^1_L$, the limits
of the molecular decomposition $f = \sum \lambda_j m_j$, in $H^1_L,\,H^1$ and $L^1$,
are all the same.  It is not known whether the same can be said for an {\it arbitrary} element
of $H^1_L,$ except in the special case that the kernel of the heat semigroup
$e^{-tL}$ enjoys a {\it pointwise} Gaussian upper bound.  In that case, it is a routine 
matter to verify that one has the 1-1 embedding $H^1_L \hookrightarrow H^1$.

\smallskip

\noindent {\it Proof of (ii).}
Let $f \in L^2(\RR^n)\cap \HPL,\, 1<p\leq p_-(L)$, and let $g\in L^2(\RR^n)\cap L^{p'}(\RR^n),$ 
so that in particular, by (iii) above, we have that $g\in H^{p'}_{L^*}(\RR^n)$ (here we are using that
$(p_-(L))' = p_+(L^*)$).  Then for such $f,g$,
we have that $$\left|\int_{\RR^n} f\,\bar{g}\,\right| = |\langle f,g\rangle|\leq 
\|f\|_{\HPL}\,\|g\|_{H^{p'}_{L^*}(\RR^n)}\lesssim\|f\|_{\HPL}\,\|g\|_{L^{p'}(\RR^n)},$$
where $\langle \cdot,\cdot\rangle$ denotes the $\HPL-H^{p'}_{L^*}(\RR^n)$ duality pairing, and where
in the last step we have used the $L^*$ version of \eqref{eq9.lpnorm}.
Taking a supremum over all $g$ as above, with $\|g\|_{L^{p'}(\RR^n)}
=1$, we obtain that $f\in L^p$ and satisfies \eqref{eq9.3**}. 

\smallskip

\noindent {\it Proof of (vi).}  By duality, it suffices to treat the case $1<p\leq p_-(L),$ since
$p_+(L) = \left(p_-(L^*)\right)'.$  Moreover, it is enough to treat the case $p=p_-(L)$:  indeed,
if (vi) holds in that case, then it must also hold for $1<p<p_-(L)$, or else
we would reach a contradiction
by interpolating with the case $p=2$.

We therefore suppose that $p = p_-(L) >1$.  
We recall 
that by  \cite{AuscherSurvey}, 
the Riesz transform $\nabla L^{-1/2}$ fails to be bounded on $L^{p},$ if $p=p_-(L)$
(cf. \eqref{eq1.5}).
Thus, by
Proposition \ref{p5.4}, we must have that $\HPL$ cannot equal $L^p(\RR^n)$ if
$p=p_-(L).$

\smallskip

To conclude the proof of the proposition, it remains to construct examples to show that the null spaces
$\mathcal{N}_\alpha(L),\, 0\leq\alpha<1$ and $\mathcal{N}_p(L),\, 2n/(n-2)<p<\infty,$
may be non-trivial.  To this end, we recall the examples of Frehse \cite{Frehse}, discussed above in 
Section \ref{s2}, namely that
for each $q<n/2$ and $\lambda >0$, there exists 
$L:=-{\rm div} A\nabla$, with $A$ complex elliptic, 
$L^\infty(\RR^n)$ and 
$C^\infty(\RR^n\setminus\{0\})$, for which
the $W^{1,2}_{loc}$ function
\begin{equation}\label{eq9.7*}
u(x):=\frac{x_1}{|x|^q}\,e^{i\lambda \ln |x|}\end{equation}
is a global weak solution of the equation $Lu=0$ in $\mathbb{R}^n.$  
Taking $\alpha = 1-q$, we then have that $u$ in \eqref{eq9.7*} belongs to 
$\Lambda_\alpha(\RR^n)$ if $0<q\leq 1$;  in fact, if $q=1$ we even have the stronger statement that
$u \in L^\infty(\RR^n)$.  Thus, $u \in \mathcal{N}_\alpha(L)$.

To exhibit an $L$ for which $\mathcal{N}_p(L)$ is non-trivial
is a bit more delicate, although matters will still depend on the construction in \cite{Frehse}.
Fix now $p>2n/(n-2)$ and choose $q<n/2$ such that $p(q-1) >n$.  We observe that for such
$p,q$, the solution $u$ in \eqref{eq9.7*} belongs to $L^p$ ``at infinity", i.e., in the complement of
any ball centered at the origin.  However, $u$ is not in $L^p$ in any neighborhood of the origin,
so we shall have to work a little harder to produce a null solution that belongs globally to
$L^p.$

Let $L:=-{\rm div}A\nabla$ be the complex elliptic matrix constructed in \cite{Frehse}, for which 
$u$ in \eqref{eq9.7*} is a global weak solution in $\RR^n$ (the matrix $A$
is given explicitly in \eqref{eq2.11} above).  We note that $A$ is smooth away from the origin,
and that $|\nabla A(x)|\leq C$ if, say, $|x|> 1/4.$  Fix a smooth cut-off function 
$\eta \in C^\infty_0(|x|\leq 3/8),$ with $0\leq \eta \leq 1$, and $\eta(x) \equiv 1$ if $|x|\leq 1/4$.
Let {\bf 1} denote the $n\times n$ identity matrix, and define an auxiliary matrix 
$$A_1 := \eta {\bf 1} + (1-\eta) A.$$
Then $A_1\in C^\infty(\RR^n)$ is complex elliptic (in the sense of \eqref{eq1.2}),
with $\|\nabla A_1\|_{L^\infty(\RR^n)} \leq C.$   Set $L_1:= -{\rm div} A_1\nabla.$

Next, we smoothly truncate $u$ away from 0.  Let $0\leq \Phi \in C^\infty(\RR^n)$, with
$\Phi (x) \equiv 1$ if $|x|\geq 1$, and $\Phi(x) \equiv 0$ if $|x|\leq 1/2$, and define
$$w:= u \,\Phi.$$ We observe that
$$L_1\, w= L \,w = - {\rm div}(u A\nabla \Phi)- A\nabla u\cdot \nabla\Phi =: f 
\,\,\in C^\infty_0\left(\frac12\leq |x|\leq 1\right).$$
We now fix $r:= 2n/(n-2)$ and $r' = 2n/(n+2)$.   Recall that 
by \cite{AuscherSurvey}, we have that
\begin{equation}\label{eq9.8*}L_1^{-1}: L^{r'}(\RR^n)\to \dot{W}^{1,2}(\RR^n)\cap L^r(\RR^n).
\end{equation}
Thus, $$w_1:= L_1^{-1} f\,\,\in \dot{W}^{1,2}(\RR^n)\cap L^r(\RR^n).$$
On the other hand, since $q<n/2$, the solution $u$ in \eqref{eq9.7*}, and 
hence also $w$, do {\bf not} belong to $L^r(\RR^n)$, nor to $\dot{W}^{1,2}(\RR^n)$
(this is related to the failure of semigroup bounds for $L_1$ in $L^p$, when $p>n/(q-1)$).
Consequently, $v:= w-w_1$ is non-trivial, and solves
$L_1 v = 0,$
globally in $\RR^n$ in the weak sense.

It therefore remains only to show that $v \in L^p(\RR^n)$ (in spite of the failure of functional calculus for
$L_1$ in $L^p$), where we recall that $p>2n/(n-2)$ was fixed above.  We begin with
the following
\begin{lemma}\label{l9.9*}  Let $r=2n/(n-2)$.  Suppose that $A\in C^1(\RR^n)$ 
is complex elliptic (in the sense of \eqref{eq1.2}), and that $\|\nabla A\|_{L^\infty(\RR^n)} \leq C_0.$
Set $L:= -{\rm div} A \nabla$, and suppose that $v\in W^{1,2}_{loc}$ is a global weak solution of
$Lv=0.$  Then there are constants $C_1$ and 
$\kappa$, depending only on $n, \,C_0$ and ellipticity, such that
for every {\bf unit} cube $Q\subset \RR^n$, we have that
\begin{equation}\label{eq9.10*}
\|v\|_{L^\infty(Q)} \leq C_1 \left(\int_{\kappa Q} |v|^r\right)^{1/r},
\end{equation}
where $\kappa Q$ denotes the concentric dilate of the unit cube $Q$, with side length $\kappa$.
\end{lemma}

Let us momentarily take the lemma for granted, and conclude the proof of Proposition \ref{embed}.
We apply Lemma \ref{l9.9*} to the operator $L_1$ and to the solution $v=w-w_1$ constructed above.
We recall that $w\in L^p(\RR^n),\, w_1 \in L^r(\RR^n),$ with $p>r:= 2n/(n-2).$
Let $\{Q_j\}$ be an enumeration of the dyadic grid of unit cubes in $\RR^n$,
and we observe that for $\kappa$ as in the lemma,
$$\sum a_j := \sum \int_{\kappa Q_j} |w_1|^r \approx \int_{\RR^n} |w_1|^r < \infty,$$
since the dilated cubes $\kappa Q_j$ have bounded overlaps.  
We now consider
\begin{multline*}
\int_{\RR^n}|v|^p= \sum \int_{Q_j}|v|^p \\\lesssim \sum  \left(\int_{\kappa Q_j} |v|^r\right)^{p/r} \lesssim
\sum  \left(\int_{\kappa Q_j} |w|^r\right)^{p/r}+\sum (a_j)^{p/r}\\=:\sum_1 +\sum_2,\end{multline*}
where in the first inequality we have used \eqref{eq9.10*}.
By H\"{o}lder's inequality, we have
$$\sum_1\lesssim \sum \int_{\kappa Q_j}|w|^p\lesssim \int_{\RR^n}|w|^p <\infty.$$
Moreover,
$$\sum_2 \,\leq\, \left(\sum a_j\right)^{p/r}\,<\,\infty, 
$$
since $p>r$.  This concludes the proof of Proposition
\ref{embed}, modulo the proof of  Lemma \ref{l9.9*}.
\ep
\begin{proof}[Proof of Lemma \ref{l9.9*}]  The inequality \eqref{eq9.10*} is a 
variant of standard classical estimates.
For the reader's convenience, we provide a proof here using a well known perturbation argument
(e.g., as in the argument on pages 87-88 in the monograph of Giaquinta \cite{Gi}), plus an iteration scheme.   

For the moment, we fix an arbitrary (i.e., not necessarily unit)
cube $Q$, of side length $\ell(Q)$, and a point $x_0 \in Q$, and define a constant coefficient
complex elliptic operator $L_0 :=-{\rm div}A_0\nabla,$ where $A_0:= A(x_0).$
By standard results for constant coefficient operators, we have that $\Gamma_0$, the fundamental solution for $L_0$, belongs to $C^\infty(\RR^n\setminus\{0\})$ and satisfies
\begin{equation}\label{eq9.11*}
|\Gamma_0 (x) |\lesssim |x|^{2-n},\quad |\nabla \Gamma_0 (x) |\lesssim |x|^{1-n},
\quad |\nabla^2 \Gamma_0 (x) |\lesssim |x|^{-n},
\end{equation}
where the implicit constants depend only upon ellipticity and dimension.

Let $\phi_Q$ be a smooth non-negative cut-off function supported in $3Q$, with $\phi_Q \equiv 1$
on $2Q$, and satisfying 
$\|\nabla \phi_Q\|_\infty \lesssim \ell(Q)^{-1},\,
\|\nabla^2 \phi_Q\|_\infty \lesssim \ell(Q)^{-2}.$
We now write
\begin{multline*}v(x_0)= v(x_0)\,\phi_Q(x_0) = \int \overline{\nabla_y\Gamma_0(x_0-y)}\cdot A_0\nabla
\left(v(y)\phi_Q(y)\right) dy\\[4pt]=\int \overline{\nabla_y\Gamma_0(x_0-y)}\cdot A_0\nabla
v(y)\,\phi_Q(y) dy \,+\,\int \overline{\nabla_y\Gamma_0(x_0-y)}\cdot A_0\nabla\phi_Q(y)\,
v(y)dy\\[4pt]= \int \overline{\nabla_y\left(\Gamma_0(x_0-y)\,\phi_Q(y)\right)}\cdot (A_0-A(y))\nabla
v(y) dy -\int \overline{\Gamma_0(x_0-y)}\nabla\phi_Q(y)\cdot A_0\nabla
v(y)dy\\[4pt] +\,\, \int \overline{\nabla_y\Gamma_0(x_0-y)}\cdot A_0\nabla\phi_Q(y)\,
v(y)dy\,=:\, I+II+III,
\end{multline*}
where we have used in term $I$ that $Lv=0$.

By \eqref{eq9.11*} and the definition of $\phi_Q$, we have that
$$|III|\lesssim \frac{1}{|Q|}\int_{3Q\setminus 2Q} |v|.$$
The same bound holds for $II$, as may be seen by integrating by parts to move the
gradient away from $v$.
Similarly, integrating by parts in term $I$ yields the estimate
\begin{multline*}|I|\lesssim \int |\nabla \Gamma_0| \,|\nabla \phi_Q|\, |v|\,+\,
 \int |\Gamma_0| \,|\nabla^2 \phi_Q|\, |v|
\\[4pt]+\,\|\nabla A\|_\infty 
\int_{3Q} |\nabla^2\Gamma_0(x_0-y)|\,|x_0-y|\, |v(y)| dy\,+\,\|\nabla A\|_\infty 
 \int |\Gamma_0| \,|\nabla \phi_Q|\, |v| \\[4pt]+\,
\|\nabla A\|_\infty 
\int_{3Q} |\nabla \Gamma_0(x_0-y)|\, |v(y)| dy\,=:\,I_1+I_2+I_3+I_4+I_5.\end{multline*}
The terms $I_1,\,I_2$ satisfy the same bound as do $II$ and $III$.
For the remaining terms, we have 
$$|I_3+I_4+I_5|\lesssim \int_{3Q}|x_0-y|^{1-n}\,|v(y)|\,dy=:I_Qv\,(x_0).$$
Combining our estimates, we obtain
\begin{equation}\label{eq9.12*}|v(x)|\lesssim \frac{1}{|Q|}\int_{3Q} |v| \,+\,I_Qv\,(x)\,,\qquad \forall x\in Q.
\end{equation}
By H\"{o}lder's inequality, we have
$$I_Qv\,(x) \lesssim \ell(Q) \left(\frac{1}{|Q|}\int_{3Q}|v|^t\right)^{1/t},$$
for any $t>n$, and each $x\in Q$, so that also
\begin{equation}\label{eq9.12**}|v(x)|\lesssim \frac{1}{|Q|}\int_{3Q} |v| \,+\,
\ell(Q) \left(\frac{1}{|Q|}\int_{3Q}|v|^t\right)^{1/t}\,,\qquad \forall x\in Q.
\end{equation}
Iterating (that is, using \eqref{eq9.12*} with $Q$ replaced by $3Q$), 
we obtain for $x\in Q$,
\begin{multline*}|v(x)|\lesssim \frac{1}{|Q|}\int_{3Q} |v| \,+\,
\ell(Q) \left(\frac{1}{|Q|}\int_{3Q}|v|^t\right)^{1/t}\\[4pt]
\lesssim \frac{1}{|Q|}\int_{3Q} |v| \,+\, \frac{\ell(Q)}{|Q|}\int_{9Q} |v|+
\ell(Q) \left(\frac{1}{|Q|}\int_{3Q}|I_{3Q}v|^t\right)^{1/t}\\[4pt]
\lesssim \frac{1}{|Q|}\int_{3Q} |v| \,+\,\frac{\ell(Q)}{|Q|}\int_{9Q} |v|+
(\ell(Q))^2 \left(\frac{1}{|Q|}\int_{9Q}|v|^s\right)^{1/s},
\end{multline*}
where in the last step $1/t=1/s -1/n$  and
we have used the fractional integral theorem.
Iterating further, and taking $Q$ to be a unit cube, we obtain the conclusion of the lemma.
\end{proof}

\section{Appendix 2:  Embedding of $\HPL$ spaces into an ambient Banach space}
\label{s10}  We shall continue to use the notational convention that
$\Lambda^0_L(\RR^n) := BM0_L(\RR^n)$.  In this appendix, we prove the following:

\begin{proposition}\label{p10.1}  Let $0<p_0<1$, and $0\leq \alpha_0<\infty$.  Then there exists a Banach space
$\mathcal{B}=\mathcal{B}(p_0,\alpha_0)$ such that 
the spaces $\HPL,\, p_0\leq p<\infty,$ and 
$\Lambda^\alpha_L(\RR^n),\, 0\leq \alpha\leq \alpha_0$, are all 
continuously embedded into $\mathcal{B}.$
\end{proposition}

\bp  We shall realize the space $\mathcal{B}$ as the dual of an appropriate normed space
${\bf M}_0 = {\bf M}_0(p_0,\alpha_0)$, which in turn will be a subspace of the intersection of
${\bf M}^{\eps_0,M}_{\alpha_0,L^*}$ 
(cf. Section \ref{s1}) and $\mathcal{D}((L^*)^M)$ (the domain of $(L^*)^M$
in $L^2(\RR^n)$), where $\eps_0>0$  and 
\begin{equation}\label{eq10.1}
M>\max\left(\frac12(\alpha_0 + n/2),\frac n2\left(\frac{1}{p_0}-\frac12\right)\right).\end{equation}  
More precisely, for such 
$\eps_0$ and $M$ fixed, we define
${\bf M}_0 = {\bf M}_0(p_0,\alpha_0)$ as the collection of all $\varphi \in L^2(\RR^n)$ such that
$\varphi$ belongs to the $\mathcal{R}((L^*)^k)$, the range of $(L^*)^k$ in $L^2(\RR^n)$,
and also to $\mathcal{D}((L^*)^k)$, for each $k = 0,1,....,M,$ and satisfies
\begin{equation}\label{eq10.2}  \|\varphi\|_{{\bf M}_0} :=  \sup_{j\geq 0}
2^{j(n/2+\alpha_0+\eps_0)}\sum_{k=-M}^M\|(L^*)^{k}\varphi\|_{L^2(S_j(Q_0))}<\infty,\end{equation}
where $Q_0$ is the unit cube centered at $0$ and $S_j(Q_0)$, $j\in\NN$, are the corresponding dyadic annuli (see (\ref{eq3.2})).  We note that $\|\cdot\|_{{\bf M}_0}$ clearly defines a norm.  We observe also
that it is easy to construct elements of ${\bf M}_0$:  just set $\varphi = (L^*)^M e^{-L^*} f$, where $f\in L^2$ with support in $Q_0$.  The bound $\|\varphi\|_{{\bf M}_0} \leq C \|f\|_{L^2(Q_0)}$ follows 
immediately from Gaffney estimates.

We now set $\mathcal{B}:= {\bf M}_0'$, the dual space of ${\bf M}_0,$
and we consider first the embedding $\Lambda^\alpha_{L}(\RR^n)\hookrightarrow \mathcal{B}$, for 
$0\leq\alpha\leq\alpha_0.$ Suppose that $\varphi \in {\bf M}_0$, with $\|\varphi\|_{{\bf M}_0}=1$.
Then $\varphi$ is an $(H^p_{L^*},(\alpha_0-\alpha)+\eps_0,M)$-molecule adapted to $Q_0$ (cf. (\ref{eq3.3})), up to multiplication by some harmless constant $C$, 
with $\alpha = n(1/p-1)$, for every $p$ such that $n/(n+\alpha_0)\leq p\leq 1.$  Thus, by Lemma
\ref{l3.1}, for every $g \in \Lambda^\alpha_L(\RR^n), 0\leq \alpha \leq \alpha_0,$ we have
$$|\langle \varphi, g\rangle|\leq C \|g\|_{\Lambda^\alpha_L(\RR^n)} = C\|\varphi\|_{{\bf M}_0} 
\|g\|_{\Lambda^\alpha_L(\RR^n)},$$
whence it follows that $\Lambda^\alpha_{L}(\RR^n)\hookrightarrow \mathcal{B}$.

Next, we consider the embedding $\HPL \hookrightarrow \mathcal{B},\, p_0\leq p\leq 1.$
Since ${\bf M}_0 \subset L^2(\RR^n),$ by (\ref{hpequivalence}) and Definition \ref{def 2.4},
it is enough to show that, given $\eps>0$, 
\begin{equation}\label{10.3}
\big|\int_{\RR^n} \varphi(x) \, m(x) \, dx\, \big|\,\leq \,C\, \|\varphi\|_{{\bf M}_0} ,
\end{equation}
for every $(H^p_L,\eps,M)$-molecule $m$.  We fix such a molecule $m$, associated to a cube $Q$. 
It is clear from the definitions (cf. (\ref{eq10.2}) and (\ref{eq3.3})) that for $k=0,1,...,M$,
\begin{equation}\label{eq10.4}
\|(L^*)^k\varphi\|_{L^2(\RR^n)}\leq C \|\varphi\|_{{\bf M}_0}\quad {\rm and } \quad 
\|\left(\big(\ell(Q)\big)^2 L\right)^{-k}m\|_{L^2(\RR^n)}\leq C \ell(Q)^{n/2-n/p},\end{equation}
Thus, for  $\ell(Q) \geq 1, $ the bound (\ref{10.3}) follows immediately from Schwarz's inequality
and (\ref{eq10.4}) with $k=0$.
On the other hand, if $\ell(Q) < 1$, we have
\begin{multline*}\big|\int_{\RR^n} \varphi(x) \, m(x) \, dx\, \big| 
=\,\ell(Q)^{2M}\, \big|\int_{\RR^n} (L^*)^M\varphi(x) \,\left(\big(\ell(Q)\big)^2 L\right)^{-M}m(x) \, dx\, \big|\\
\leq \,C\,\|\varphi\|_{{\bf M}_0} \, \ell(Q)^{2M +n/2-n/p}
\end{multline*}
by (\ref{eq10.4}) with $k=M$.  Since $p\geq p_0$, for $M$ as in (\ref{eq10.1}), we obtain (\ref{10.3}).

Finally, we suppose that $1<p<\infty,$ and let $f\in
L^2(\RR^n) \cap \HPL$.  Setting $\psi(\zeta):=\zeta^Me^{-\zeta}$,
by the Calder\'{o}n reproducing formula (\ref{eq4.17}) and duality, we have
\begin{multline*}\big|\int_{\RR^n} \varphi(x) \, f(x) \, dx\, \big|\,\leq \,C\, \|Q_{\psi,L} f\|_{T^p(\RR^{n+1}_+)}
\, \|Q_{\psi,L^*} \varphi\|_{T^{p'}(\RR^{n+1}_+)}\\
\leq \,C\,\|f\|_{\HPL}\, \|Q_{\psi,L^*} \varphi\|_{T^{p'}(\RR^{n+1}_+)}.
\end{multline*}
It is therefore enough to show that, for $\|\varphi\|_{{\bf M}_0} = 1$,
\begin{equation}\label{10.6}\, \|Q_{\psi,L^*} \varphi\|_{T^{p'}(\RR^{n+1}_+)}\equiv \|\Sq (Q_{\psi,L^*} \varphi)\|_{L^{p'}(\RR^n)} \leq C, \qquad 1<p'<\infty,\end{equation}
where we remind the reader that the ``area integral" $\Sq$ is defined in (\ref{eq4.2}).
We first note that (\ref{10.6}) with $p'=2$ follows immediately by
standard quadratic estimates and the case $k=0$ of (\ref{eq10.4}).   Moreover,
$\|\varphi\|_{H^1_L(\RR^n)}\leq C$ (indeed, as mentioned above,
$\varphi$ is an $(H^1_L,\alpha_0+\eps_0,M)$-molecule adapted to $Q_0$, 
up to multiplication by a harmless constant),
so that by Proposition \ref{p4.3}, we have
$$\|Q_{\psi,L^*} \varphi\|_{T^{1}(\RR^{n+1}_+)}=\|\Sq(Q_{\psi,L^*} \varphi)\|_{L^{1}(\RR^n)} \leq C.$$
Combining the latter bound with that for $p'=2$, we obtain immediately  (\ref{10.6}) in the case
$1<p'<2.$

Similarly, to handle the case $2<p'<\infty,$ it is enough to show that 
$\Sq(Q_{\psi,L^*} \varphi)\in L^\infty(\RR^n).$  To this end, we write
\begin{multline*}\left(\Sq(Q_{\psi,L^*} \varphi)(x)\right)^2  :=
\iint_{|x-y|<t} |(t^2L^*)^Me^{-t^2L^*}\!\varphi\,(y)|^2 \frac{dy dt}{t^{n+1}}\\[4pt]
\leq\, \int_0^1\!\int_{\RR^n}|t^{2M}e^{-t^2L^*}(L^*)^M\varphi\,(y)|^2 \frac{dy dt}{t^{n+1}}
\,+\,\int_1^\infty\!\int_{\RR^n}|(t^2L^*)^Me^{-t^2L^*}\!\varphi\,(y)|^2 \frac{dy dt}{t^{n+1}}\\[4pt]
\leq \,\int_0^1t^{4M-n-1} dt \,+\,\int_1^\infty t^{-n-1} dt\,\leq C\,,
\end{multline*}
where in the next-to-last 
inequality we have used (\ref{eq10.4}) with $k=M$ in the first term and with $k=0$ in the
second, along with $L^2$ boundedness of $(t^2L^*)^k e^{-t^2L^*}$ for every non-negative integer
$k$, and in the very last step we have used that $M>n/4$, by (\ref{eq10.1}) and the fact that $p_0\leq 1.$

\ep

\end{document}